\documentclass[11pt]{amsart}
\usepackage[margin=1in]{geometry} 
\usepackage{amsmath,amssymb}
\usepackage{graphicx} 
\usepackage{enumerate}
\usepackage{mathtools}
\usepackage{xcolor,soul}
\usepackage{cite}
\usepackage{tikz}
\usepackage{tikz-cd}
\tikzcdset{scale cd/.style={every label/.append style={scale=#1},
    cells={nodes={scale=#1}}}}
\usetikzlibrary{arrows}
\usepackage{dbnsymb}
\usepackage{stmaryrd}
\usepackage{comment}
\usepackage{arcs}
\usepackage[pagebackref,breaklinks=true]{hyperref}\hypersetup{colorlinks,
  linkcolor={green!50!black},
  citecolor={green!50!black},
  urlcolor=blue
}
\usepackage{bbold}
\usetikzlibrary{matrix}
\usepackage{mathabx}
\usepackage{multicol}

\usepackage{caption}
\usepackage[labelformat=simple]{subcaption}

\usepackage{amsmath, amsthm}
\usepackage{relsize}
\usetikzlibrary{cd}
\usetikzlibrary{decorations.pathreplacing}
\usepackage{tikz,calc}

\usepackage{color}

\definecolor{lgray}{gray}{0.82}
\definecolor{rredd}{rgb}{120,0,0}

\newtheorem{thm}{Theorem}[section]
\newtheorem{prop}[thm]{Proposition}
\newtheorem{lem}[thm]{Lemma}
\newtheorem{cor}[thm]{Corollary}

\newcommand{\theoremname}{Theorem:}

\newtheorem*{main*}{Main Result}

\usepackage{bbm}
\newcommand{\Q}{\mathbb{Q}}
\newcommand{\C}{\mathbb{C}}
\newcommand{\R}{\mathbb{R}}
\newcommand{\Z}{\mathbb{Z}}
\newcommand{\tZ}{\tilde{Z}}

\newcommand{\calA}{\mathcal{A}}
\newcommand{\tcalA}{\tilde{\mathcal{A}}}
\newcommand{\A}{\mathcal{A}}
\newcommand{\calB}{\mathcal{B}}
\newcommand{\tA}{\tilde{\mathcal{A}}}

\newcommand{\calK}{{\mathcal{K}}}
\newcommand{\tcalK}{\tilde{\mathcal{K}}}

\newcommand{\calT}{{\mathcal{T}}}
\newcommand{\calC}{\mathcal{C}}
\newcommand{\calI}{\mathcal{I}}
\newcommand{\tcalC}{\tilde{\mathcal{C}}}

\newcommand{\ctT}{\mathbb{C}\tilde{\mathcal{T}}}
\newcommand{\tT}{\tilde{\mathcal{T}}}
\newcommand{\tTnab}{\tilde{\mathcal{T}}_{\scriptscriptstyle\nabla}}

\newcommand{\nab}{{\scriptscriptstyle\nabla}}
\newcommand{\calD}{\mathcal{D}}
\newcommand{\D}{\mathcal{D}}

\DeclareMathOperator{\As}{FA}

\newcommand{\bb}{\bar{\beta}}
\newcommand{\Cp}{\C\pi}
\newcommand{\Cpa}{\lvert \C\pi\rvert}

\newcommand{\hAs}{\widehat{\As}}
\newcommand{\hAsa}{\lvert\widehat{\As}\rvert}
\newcommand{\Ctp}{\C\tilde{\pi}}
\newcommand{\Ctpa}{\lvert\C\tilde{\pi}\rvert}

\newcommand{\tbeta}{\tilde{\beta}}
\newcommand{\cl}{\mathbb{1}}
\newcommand{\Cpba}{|\overline{\C\pi}|}

\newcommand{\tCp}{{\C\tilde{\pi}}}
\newcommand{\rot}{\operatorname{rot}}

\newcommand{\upsidein}{\mathbin{\rotatebox[origin=c]{-90}{$\in$}}}

\newcommand{\tpi}{\tilde{\pi}}

\newcommand{\gr}{\operatorname{gr}}
\newcommand{\id}{\operatorname{id}}

\newcommand{\im}{\operatorname{im}}

\newcommand{\Tql}{T \sqcup_q \uparrow}

\def\yellowm#1{{\setlength{\fboxsep}{0pt}\colorbox{yellow}{$#1$}}}
\def\glosm#1#2{{\label{g:#1}\yellowm{#2}}}
\def\glosi#1#2#3#4{{\item[{#2}] #3~\hfill #4}}

\newcommand{\poskink}{\hspace{1mm}\begin{tikzpicture}[baseline=0.7ex,scale=1.5, >=stealth]
\draw[thick, ->] (0,1.4ex) -- (0,2.2ex);
\draw[thick] (0, 0.4ex) .. controls (0, 1.4ex) and (0.3em, 1.4ex) .. (0.3em, 0.9ex);
\draw[thick] (0.1em, 0.8ex) .. controls (0.12em, 0.4ex) and (0.3em, 0.4ex) .. (0.3em, 0.9ex);
\draw[thick] (0, 0) -- (0, 0.4ex);
\end{tikzpicture}\hspace{1mm}
}

\newcommand{\negkink}{\hspace{1mm}\begin{tikzpicture}[baseline=0.7ex,scale=1.5, >=stealth]
\draw[thick,->] (0,1.4ex) -- (0,2.2ex);
\draw[thick] (0, 1.4ex) .. controls (0, 0.4ex) and (0.3em, 0.4ex) .. (0.3em, 0.9ex);
\draw[thick] (0.1em, 1ex) .. controls (0.12em, 1.4ex) and (0.3em, 1.4ex) .. (0.3em, 0.9ex);
\draw[thick] (0, 0) -- (0, 0.4ex);
\end{tikzpicture}\hspace{1mm}
}

\newcommand{\doublekink}{\hspace{1mm}\begin{tikzpicture}[baseline=0.7ex,scale=1.5, >=stealth]
\draw[thick, ->] (0,1.4ex) -- (0,2ex);
\draw[thick] (0, 1.4ex) .. controls (0, 0.4ex) and (0.3em, 0.4ex) .. (0.3em, 0.9ex);
\draw[thick] (0, 0.4ex) .. controls (0, 1.4ex) and (0.3em, 1.4ex) .. (0.3em, 0.9ex);
\draw[thick] (0, 0) -- (0, 0.4ex);
\draw[fill=black] (0.03em, 0.9ex) circle (0.1em);
\end{tikzpicture}\hspace{1mm}
}

\newcommand{\frchange}{
\begin{tikzpicture}[baseline=0.7ex,scale=1.5, >=stealth]
\draw[thick, ->] (0,0) -- (0,2ex);
\draw[thick] (-0.15em, 1.1ex) -- (0.15em, 0.8ex);
\end{tikzpicture}
}

\newcommand{\circarrow}{\hspace{1mm}\begin{tikzpicture}[scale=1.5, baseline=0.7ex, >=stealth]
\draw[->, thick] (0,1.25ex) -- (0,2.5ex);
\draw[thick] (0,1ex) circle (0.25ex);
\draw[thick] (0, 0) -- (0, 0.75ex);
\end{tikzpicture}\hspace{1mm}
}

\newcommand{\arrowup}{\hspace{1mm}\begin{tikzpicture}[scale=1.5, baseline=0.7ex, >=stealth]
\draw[->, thick] (0,0ex) -- (0,2.2ex);
\end{tikzpicture}\hspace{1mm}
}

\newcommand{\shortchord}{\hspace{.6mm}\begin{tikzpicture}[baseline=0.7ex, >=stealth, scale=1.5]
\draw[->, thick] (0,0) -- (0,2.5ex);
\draw[thick, densely dotted] (0,0.5ex) arc (-90:90:0.6ex);
\end{tikzpicture}\hspace{.4mm}
}

\usetikzlibrary{arrows.meta}

\newcommand{\circleSkel}{
 \begin{tikzpicture}[scale=.2, baseline=.0mm  ]
\draw[thick] (0,.5) circle (.7cm);
\end{tikzpicture} 
}

\newcommand{\dncom}{
\begin{tikzpicture}[scale=.3, baseline=.0mm  ]
\put(1.5,0) {\rotatebox[origin=c]{90}{$\circlearrowleft$}}
\draw[ - ](.1,-.4)--(1.2,1.2);
\end{tikzpicture} 
}

\newcommand{\botSkel}{
\begin{tikzpicture}[scale=.2, baseline=.0mm  ]
\draw[thick](0,0) .. controls (0,1.5) and (1.75,1.5) .. (1.75,0);

\end{tikzpicture} 
}

\newcommand{\swap}{
\begin{tikzpicture}[baseline=-2.75, scale=.2]
\draw[-{Stealth[ length=1.25mm, width=1.25mm]},thick ](-1,-1)--(1,1);
\draw[-{Stealth[ length=1.25mm, width=1.25mm]},thick ](1,-1)--(-1,1);
\end{tikzpicture}}

\newcommand{\smoothchord}{
\begin{tikzpicture}[baseline=-2.75, scale=.2]
\draw[-{Stealth[ length=1.25mm, width=1.25mm]},thick ](-0.75,-1)--(-0.75,1);
\draw[-{Stealth[ length=1.25mm, width=1.25mm]},thick ](0.75,-1)--(0.75,1);
\end{tikzpicture}}

\newcommand{\pcross}{
\begin{tikzpicture}[baseline=-2.75, scale=.2]
\draw[-{Stealth[ length=1.25mm, width=1.25mm]},thick ](-1,-1)--(1,1);
\draw[-{Stealth[ length=1.25mm, width=1.25mm]},thick ](-.25,.25)--(-1,1);
\draw[thick ](.3,-.3)--(1,-1);
\end{tikzpicture}}

\newcommand{\ncross}{
\begin{tikzpicture}[baseline=-2.75,scale=.2]
\draw[-{Stealth[ length=1.25mm, width=1.25mm]},thick ](1,-1)--(-1,1);
\draw[-{Stealth[ length=1.25mm, width=1.25mm]},thick ](.25,.25)--(1,1);
\draw[thick ](-.3,-.3)--(-1,-1);
\end{tikzpicture}}

\newcommand{\dpcross}{
\begin{tikzpicture}[baseline=-2.75,scale=.2]
\draw[-{Stealth[ length=1.25mm, width=1.25mm]},thick ](1,-1)--(-1,1);
\draw[-{Stealth[ length=1.25mm, width=1.25mm]},thick ](-1,-1)--(1,1);
\draw[fill=black] (0,0) circle (.2);
\end{tikzpicture}}

\newcommand{\smooth}{
\begin{tikzpicture}[baseline=-2.75,scale=.2]
\draw[-{Stealth[ length=1.25mm, width=1.25mm]},thick ] (-.8,-1) to[out=45,in=-45] (-1,1);
\draw[-{Stealth[ length=1.25mm, width=1.25mm]},thick ] (.8,-1) to[out=135,in=215] (1,1);
\end{tikzpicture}}

\newcommand{\sschordsame}{
\begin{tikzpicture}[baseline=-2.75,scale=.2]
\draw[-{Stealth[ length=1.25mm, width=1.25mm]},thick ] (-.8,-1) to[out=45,in=-45] (-1,1);
\draw[-{Stealth[ length=1.25mm, width=1.25mm]},thick ] (1.4,-1) to[out=135,in=215] (1.6,1);
\draw[thick, densely dotted](-.5,0)--(1.1,0);
\end{tikzpicture}}

\newcommand{\sschordop}{
\begin{tikzpicture}[baseline=-2.75,scale=.2]
\draw[-{Stealth[ length=1.25mm, width=1.25mm]},thick ] (-.8,-1) to[out=45,in=-45] (-1,1);
\draw[-{Stealth[ length=1.25mm, width=1.25mm]},thick ] (1.4,1) to[out=215,in=135] (1.6,-1);
\draw[thick, densely dotted](-.5,0)--(1.1,0);
\end{tikzpicture}}

\newcommand{\smoothop}{
\begin{tikzpicture}[baseline=-2.75,scale=.2]
\draw[-{Stealth[ length=1.25mm, width=1.25mm]},thick ] (1.4,1.5)--(1.4,.3)--(-1,.3)--(-1,1.5);
\draw[-{Stealth[ length=1.25mm, width=1.25mm]},thick ] (-1,-1.5)--(-1,-.3)--(1.4,-.3)--(1.4,-1.5); 
\end{tikzpicture}}

\newcommand{\smoothsame}{
\begin{tikzpicture}[baseline=-2.75,scale=.2]
\draw[-{Stealth[ length=1.25mm, width=1.25mm]},thick ] (1.4,-1.5)--(1.4,-.6)--(-1,.6)--(-1,1.5);
\draw[-{Stealth[ length=1.25mm, width=1.25mm]},thick ] (-1,-1.5)--(-1,-.6)--(1.4,.6)--(1.4,1.5); 
\end{tikzpicture}}

\newcommand{\chordwithswap}{
\begin{tikzpicture}[baseline=-3.5,scale=.2]
\draw[-{Stealth[ length=1.25mm, width=1.25mm]},thick ] (1.4,-2)--(1.4,-.6)--(-1,.6)--(-1,1.5);
\draw[-{Stealth[ length=1.25mm, width=1.25mm]},thick ] (-1,-2)--(-1,-.6)--(1.4,.6)--(1.4,1.5); 
\draw[thick, densely dotted](-1,-1.3)--(1.4,-1.3);
\end{tikzpicture}}

\newcommand{\swapswap}{
\begin{tikzpicture}[baseline=-3.5,scale=.2]
\draw[-{Stealth[ length=1.25mm, width=1.25mm]},thick ] 
(1.4,-2)--(1.4,-1.5)--(-1.4,-.4)--(1.4,.5)--(1.4,1.5);
\draw[-{Stealth[ length=1.25mm, width=1.25mm]},thick ] (-1.4,-2)--(-1.4,-1.5)--(1.4,-.4)--(-1.4,.5)--(-1.4,1.5);
\end{tikzpicture}}

\newcommand{\powerchord}{
\begin{tikzpicture}[baseline=-2.75,scale=.2]
\draw[-{Stealth[ length=1.25mm, width=1.25mm]},thick ] (-2,-4)--(-2,4);
\draw[-{Stealth[ length=1.25mm, width=1.25mm]},thick ] (2,-4)--(2,4);
\draw[thick, densely dotted](-2,-2)--(2,-2);
\draw[thick, densely dotted](-2,-1)--(2,-1);
\draw[thick, densely dotted](-2,-0)--(2,0);
\draw[thick, densely dotted](-2,1)--(2,1);
\draw [decorate,decoration={brace,amplitude=2pt,mirror}](2.5,-2.5) -- (2.5,1.5) node [black,midway, right] {$k$};
\end{tikzpicture}}

\newcommand{\powerswap}{
\begin{tikzpicture}[baseline=-2.75,scale=.2]
\draw[-{Stealth[ length=1.25mm, width=1.25mm]},thick ] (-2,-5)--(-2,-3)--(2,-1.5)--(-2,0)--(2,1.5)--(-2,3)--(-2,5);
\draw[-{Stealth[ length=1.25mm, width=1.25mm]},thick ] (2,-5)--(2,-3)--(-2,-1.5)--(2,0)--(-2,1.5)--(2,3)--(2,5);
\draw [decorate,decoration={brace,amplitude=4pt,mirror}](3,-3) -- (3,3) node [black,midway, right, xshift=1mm] { $k$};
\end{tikzpicture}}

\newcommand{\dpointwithxarrowsup}{
\begin{tikzpicture}[baseline=-2.75, scale=.2]
\draw[-{Stealth[ length=1.25mm, width=1.25mm]},thick ](-.8,-1.5) to[out=45,in=-90] (.3,0)to[out=90,in=-30] (-1,1.5);
\draw[thick] (.8,-1.5) to[out=155,in=-90] (-.5,-.2)to[out=90,in=190] (-.3,.4);
\draw[-{Stealth[ length=1.25mm, width=1.25mm]},thick ](.3,1)--(.9,1.5);
\draw[thick] (.7,-1.45)--(.8,-1.5);
\draw[fill=black] (-.1,-.9) circle (.2);
\end{tikzpicture}}

\newcommand{\bigslant}[2]{{\raisebox{.2em}{$#1$}\left/\raisebox{-.2em}{$#2$}\right.}}

\usetikzlibrary{arrows}
\usepackage{wrapfig}

\theoremstyle{definition}
\newtheorem{definition}[thm]{Definition}
\newtheorem{example}[thm]{Example}

\theoremstyle{remark}

\newtheorem{remark}[thm]{Remark}

\numberwithin{equation}{section}

\usepackage[T1]{fontenc}

\begin{document}

\title{Goldman-Turaev formality from the Kontsevich integral}
\author[D. Bar-Natan]{Dror~Bar-Natan}
\address{
  Department of Mathematics\\ University of Toronto\\ Toronto, Ontario, Canada
}
\email{drorbn@math.toronto.edu}
\urladdr{http://www.math.toronto.edu/~drorbn}

\author[Z. Dancso]{Zsuzsanna Dancso}
\address{School of Mathematics and Statistics\\ The University of Sydney\\ Sydney, NSW, Australia}
\email{zsuzsanna.dancso@sydney.edu.au}

\author[T. Hogan]{Tamara Hogan}
\address{Department of Mathematics \\ University of Toronto \\ Toronto, Ontario, Canada}
\email{tamara.hogan@utoronto.ca}
\urladdr{https://www.tamaramaehogan.com/}

\author[J. Liu]{Jessica Liu}
\address{Department of Mathematics\\ University of Toronto\\ Toronto, Ontario, Canada}
\email{chengjin.liu@mail.utoronto.ca}

\author[N. Scherich]{Nancy Scherich}
\address{Department of Mathematics and Statistics\\ Elon University\\ Elon, North Carolina}
\email{nscherich@elon.edu}
\urladdr{http://www.nancyscherich.com}



\begin{abstract}
We present a new solution to the formality problem for the framed Goldman--Turaev Lie bialgebra in genus zero, constructing Goldman-Turaev homomorphic expansions (formality isomorphisms) from the Kontsevich integral.
Our proof uses a three dimensional derivation of the Goldman-Turaev Lie biaglebra arising from a low-degree Vassiliev quotient -- the {\em emergent} quotient -- of tangles in a thickened punctured disk, modulo a Conway skein relation. This is in contrast to Massuyeau's 2018 proof using braids. A feature of our approach is a general conceptual framework which is applied to prove the compatibility of the homomorphic expansion with both the Goldman bracket and the technically challenging Turaev cobracket. 

\end{abstract}
\maketitle

\tableofcontents

\section{Introduction}

In 1986, Goldman defined a Lie bracket \cite{Gol}  on the space of homotopy classes of free loops on a compact oriented surface.
Shortly after, in 1991, Turaev defined a cobracket 
(\cite{Tur}, improved by \cite{AKKN_formality}) on the same space.  
This bracket and cobracket make the space of free loops into a Lie bialgebra -- known as the Goldman-Turaev Lie bialgebra -- which forms the basis for the field of string topology \cite{CS} and has been an object of study from many perspectives. 

In this paper we describe a new solution to the {\em formality problem}\footnote{The term {\em formality} is inspired by rational homotopy theory \cite{SW}, though we do not use any specific rational homotopy theory tools here.} for the Goldman-Turaev Lie bialgebra in genus zero. Given a filtered algebraic structure $\calB$ 
$$\calB=F_0\supseteq F_1\supseteq F_2 \cdots$$ the {\em formality problem for $\calB$} is to find a {\em formality isomorphism} between the completion $\varprojlim_i \calB/F_i$ and the degree completed associated graded structure $\calA:= \prod F_i/F_{i+1}$.  Finding a formality isomorphism is equivalent to finding a {\em homomorphic expansion}\footnote{In the sense of \cite{BND-KTGS, BND-WKO1}.} for $\calB$, that is, a filtered homomorphism $$Z:\calB\to \calA$$ with the {\em universality} property that the associated graded map of $Z$ is the identity: $\gr Z=\id_\calA$.

To construct a homomorphic expansion for the Goldman-Turaev Lie bialgebra, we use a 3-dimensional structure on a space of tangles in a handlebody, and recover the Goldman bracket and Turaev cobracket maps as ``shadows'' of certain natural operations on tangles. More precisely, these tangle operations induce the Goldman--Turaev operations as {\em connecting homomorphisms} in a specific sense explained in Section~\ref{sec:conceptsum}. 
We show the Kontsevich integral is homomorphic with respect to these tangle operations.
 Our main result is summarised as follows:

\begin{main*}
    Let $\C\tT$ denote the  space of formal linear combinations of tangles in  a thickened punctured disc $D_p \times I$. Modding out by crossing changes, the Kontsevich integral on $D_p\times I$ descends to a homomorphic expansion for the Goldman--Turaev Lie bialgebra of homotopy curves in $D_p$.
\end{main*}

In more detail, consider the space  $\C\tT$ of formal linear combinations of framed tangles in the handlebody $\D_p\times I$, where $\D_p$ denotes a disc with $p$ punctures, and $I$ is the unit interval. We call the vertical lines above the punctures {\em poles}, and the connected components of the tangle {\em strands}, as in Figure \ref{fig:3D_polestudio}. 
\begin{figure}
    \centering
    \includegraphics[width=0.3\linewidth]{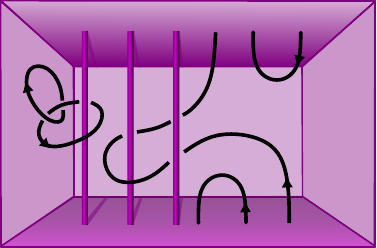}
    \caption{A tangle in $D_p\times I$. The lines above the punctures (poles) are the vertical (purple) lines on the left.}
    \label{fig:3D_polestudio}
\end{figure}
There are two Vassiliev filtrations on $\C\tT$: one which considers both strand-strand and strand-pole crossing changes, called the {\em total} filtration; and one which only allows strand-strand crossing changes, called the {\em strand} filtration (also studied in \cite{HM}). The quotient by the first step of the strand filtration leaves only homotopy classes of curves in $D_p$: this is where the Goldman--Turaev Lie bialgebra lives.

Everything important in this paper already appears, in fact, in the quotient by the second step of the strand filtration. We call this quotient the space of {\em emergent} tangles, situated in between homotopy classes of curves (where there is no notion of over or under strands) and classical tangles, where unlimited knotting can exist. In emergent tangles almost all knotting is eliminated save for the dependence on at most a single crossing change. The term emergent knotting is first used in \cite{Kuno25}, where Kuno studies a related infinitesimal question: the pentagon equations which describe associators in the emergent quotient. The {\em emergent} terminology is explained in detail in the Appendix by Bar-Natan to \cite{Kuno25}.

The two relevant operations on $\C\tT$ and the emergent quotient are the commutator associated to the stacking product of $\tT$, and the difference between a tangle and its vertical mirror image.
Passing to a Conway skein quotient of $\tT$, (Section~\ref{sec:Conway}), the Goldman bracket is induced by the commutator as a ``connecting homomorphism''. The Turaev cobracket is induced by the difference between a tangle and its flip, save for some technical details.

Our main result is to show that the Kontsevich integral is a homomorphic expansion with respect to tangle operations and the induced Goldman bracket and Turaev cobracket: in other words, it intertwines these operations with their associated graded counterparts. To do so, we use the naturality of connecting homomorphisms (Section~\ref{sec:conceptsum}) as a guiding principle.  

\medskip

One of our future goals is to extend this result to the framed Goldman--Turaev Lie bialgebras in higher genera. We expect that the derivation of the Goldman--Turaev operations from tangle operations following Section~\ref{sec:conceptsum} extends to higher genus with minor modifications. For addressing formality in higher genus, the missing link is the Kontsevich integral. In the case of the torus, one would replace this with a homomorphic expansion built from {\em elliptic associators}, as in \cite{En, Hum, CEE}. 

\subsection*{Motivation: Kashiwara--Vergne theory}
 The Kashiwara--Vergne (KV) equations originally arose from the study of convolutions on Lie groups and Lie algebras \cite{KV}. The equations were reformulated algebraically in terms of automorphisms of free Lie algebras in \cite{AT}, in this form they are a refinement of the Baker-Campbell-Hausdorff formula for products of exponentials of non-commuting variables. Shortly thereafter, \cite{AET} gave an explicit formula, constructing Kashiwara--Vergne solutions from Drinfel'd associators. The reverse direction -- whether all KV solutions arise from associators in this way -- is a significant open problem raised in \cite{AT} known as the Alekseev--Torossian conjecture. 
 
Kashiwara--Vergne theory has multiple topological interpretations in which Kashiwara--Vergne solutions correspond to  {\em homomorphic expansions} -- equivalently, formality isomorphisms -- of topological objects.

One of these topological interpretations is due to the first two authors \cite{WKO2}, who showed that homomorphic expansions of  welded foams -- a class of 4-dimensional tangles -- are in one to one correspondence with solutions of the KV-equations. More recently, a series of papers by Alekseev, Kawazumi, Kuno and Naef \cite{AKKN_formality,akkn_g0,AKKN_highergenus} drew an analogous connection between KV solutions and homomorphic expansions for the Goldman-Turaev Lie bialgebra for the disc with two punctures (up to non-negligible technical differences). This latter correspondence was used to generalise the Kashiwara--Vergne equations by considering different surfaces, including those of higher genus \cite{AKKN_highergenus}. 

In other words, there is an intricate algebraic connection between four dimensional welded foams and the Goldman--Turaev Lie bi-algebra, which strongly suggests that there is a topological connection as well. Beyond the inherent interest in tangles and formality, this paper has two further goals:
\begin{enumerate}
\item To work towards finding a connection between the two-dimensional Goldman--Turaev Lie bialgebra and four-dimensional welded foams, by constructing a three-dimensional derivation of the Goldman-Turaev Lie bialgebra, with homomorphic expansions which descend to Goldman-Turaev expansions.
\item To better understand the relationship between Drinfel'd associators and Kashiwara--Vergne solutions: homomorphic expansions for (parenthesised) tangles are indexed by Drinfel'd associators, and homomorphic expansions for the Goldman--Turaev Lie bialgebra are closely tied to Kashiwara--Vergne solutions. 

\end{enumerate}

\subsection*{Related work}
Two further previous results address Goldman--Turaev formality in genus zero. The first is due to Massuyeau \cite{Mas}, posted in 2015, who used braids and the Kontsevich integral to show the formality of the Turaev cobracket. In comparison, the new contribution of our approach is two-fold: the derivation of the cobracket via tangles is conceptually simpler, and Section~\ref{sec:conceptsum} offers a general conceptual framework, which we expect will provide insight into the higher genus case and the relationship to welded foams \cite{WKO2}, in order to link the different topological interpretations of Kashiwara--Vergne theory described above. 

Another approach to Goldman--Turaev formality in genus zero, due to Alekseev and Naef \cite{AN}, constructs formality isomorphisms from the Knizhnik-Zamolodchikov connection. This proof uses a direct argument showing that the formality isomorphism is compatible with the Goldman bracket and Turaev cobracket, in place of a 3-dimensional topological derivation of these maps.

The paper \cite{Kuno25} inverstigates emergent versions of the pentagon equations, studying the infinitesimal (associated graded) side of the formality problem. In parallel to our paper, this is also a step towards better understanding the relationships between Drinfel'd associators, Kashiwara--Vergne solutions, emergent tangles, and welded foams. 

Further related work includes the recent paper \cite{ANR}, which derives explicit formulas for the Kirillov-Kostant-Souriau coaction maps of open path regularized holonomies of the Knizhnik-Zamolodchikov equation (close cousins of the Turaev cobrackets). The paper \cite{Ren} introduces a reduced coaction equation on the associated graded side, and studies the corresponding Lie algebra.

\subsection*{Organisation} Section~\ref{sec:conceptsum} gives a general algebraic framework for how the Goldman--Turaev operations are induced by tangle operations. In Section \ref{sec:Prelims} we give a brief overview of the Kontsevich integral and the Goldman Turaev Lie bialgebra.
In Section \ref{sec:TangleSetUp}, we define tangles in handlebodies, relevant operations and Vassiliev
filtrations. We  identify the associated graded space of tangles as a space of chord diagrams, and introduce the Conway skein quotient. 
In Section \ref{sec:IdentifyingGTinCON}, we identify the Goldman--Turaev Lie bialgebra in a low filtration degree, and prove the main result in two theorems: Theorem \ref{thm:Cube_for_bracket} and \ref{thm:cobrackethomomorphic}. 

\vspace{5mm}\noindent \emph{Acknowledgements.} We are truly grateful to the Anonymous Referee for many insightful comments and suggestions which substantially improved on the first version of this paper. We thank Anton Alekseev, Gwena\"el Massuyeau, and Yusuke Kuno for many fruitful conversations.  DBN was supported by NSERC RGPIN 262178 and RGPIN-2018-04350, and by The Chu Family Foundation (NYC). ZD was partially supported by the ARC DECRA Fellowship DE170101128. NS was supported by the National Science Foundation under Grant No. DMS-1929284 while in residence at the Institute for Computational and Experimental Research in Mathematics in Providence, RI, during the Braids Program. NS was also supported by the National Science Foundation under Grant No. DMS-2532699. We thank the Sydney Mathematical Research Institute and the University of Sydney for their hospitality and funding for multiple research visits.


\section{Conceptual summary}\label{sec:conceptsum}

We induce the genus zero Goldman-Turaev operations from tangle operations, in the spirit of ``connecting homomorphisms'': this Section is a summary of the basic approach. We provide some proofs which are not immediate and use the words {\em homomorphic expansions}, and {\em Goldman-Turaev operations} without definition, only mentioning their basic properties which make this conceptual outline coherent; the definitions follow in Section~\ref{sec:Prelims}.

The idea formalises the following basic algebra fact. Consider a homomorphism of modules $\lambda: B \to E$ and two submodules $A\subseteq B$ and $D\subseteq E$.
If $\lambda(A)=0$ and $\lambda(B) \subset D$, then $\lambda$ induces a map $\eta:B/A\to D$, by setting $[b] \mapsto \lambda(b)$. We now restate this as a commutative diagram \eqref{eq:inducedconnhom} which makes sense in any Abelian category.

The top and bottom rows of \eqref{eq:inducedconnhom} are exact and the right and left vertical maps are zero, and therefore, by minor diagram chasing, the middle vertical map $\lambda$ induces a unique map $\eta: C \to D$. In other words, there is a unique map $\eta$ which fits into a commutative square with the quotient map $B\to C$, $\lambda: B\to E$, and the inclusion $D\to E$. The induced map $\eta$ can also be seen as a degenerate case of a connecting homomorphism. In our applications $\lambda$ is a difference of two maps $\lambda_1$ and $\lambda_2$, whose values differ in $E$ but coincide in the quotient $F$. 
\begin{equation}\label{eq:inducedconnhom}
\begin{tikzcd}
&A \arrow[rr]\arrow[d, "0",swap]  && B \arrow[rr]\arrow[d,"\lambda=\lambda_1-\lambda_2"] &&C\arrow[d,"0"] \arrow[r] \arrow[dllll, dashed, "\eta", near start,swap, controls={+(-1.7,1.3) and +(.5,0.8)}]&0\\
0\arrow[r]& D \arrow[rr]
&& E \arrow[rr]
&& F
\end{tikzcd}
\end{equation}

In Section~\ref{sec:IdentifyingGTinCON} we present two constructions which produce the Goldman bracket and the Turaev cobracket, respectively, as induced homomorphisms $\eta$, from corresponding tangle operations $\lambda_1$ and $\lambda_2$. 
The following example is a schematic version of what will become the argument for the Goldman bracket.
\begin{example} 
Let $A$ be an associative algebra, and let $\{L_i\}$ denote the lower central series of $A$. That is, $L_0:=A$, and $L_{i+1}:=[L_i,A]$. Then the $L_i$ are Lie ideals, and let $M_i=AL_i=L_iA$ denote the two-sided ideal generated by $L_i$. The quotient $A/M_1$ is the abelianization of $A$, denoted by $A^{ab}$. Then we have the following diagram:\begin{equation}\label{eq:SnakeExample}
\begin{tikzcd}
0\arrow[r] &K \arrow[rr]\arrow[d, "0",swap]  && 
\frac{A}{M_2} \otimes \frac{A}{M_2}
\arrow[r] \arrow{d}{[\cdot, \cdot]} & A^{ab} \otimes A^{ab} \arrow[d,"0"] \arrow[r] \arrow[dlll, dashed, "\eta", near start,swap, out=120, in=70, looseness=.6]&0\\
0\arrow[r]& 
\frac{M_1}{M_2} \arrow[rr]
&& \frac{A}{M_2} \arrow[r]
& A^{ab} \arrow[r] & 0
\end{tikzcd}
\end{equation} 
Here $\lambda$ is the algebra commutator, which is indeed the difference between two maps: the multiplication ($\lambda_1$) and the multiplication in the opposite order $(\lambda_2)$. The kernel $K$ of the projection to $A^{ab}\otimes A^{ab}$ is generated by the subalgebras $\left\{\frac{M_1}{M_2} \otimes \frac{A}{M_2}, \frac{A}{M_2} \otimes \frac{M_1}{M_2}\right\}$ in $\frac{A}{M_2} \otimes \frac{A}{M_2}$.
The map $\eta$ is a well defined commutator map $A^{ab}\otimes A^{ab} \to \frac{M_1}{M_2}$, given by $\eta(x\otimes y)=[x,y]$  mod $M_2.$ \qed
\end{example}

The goal of this paper is to construct homomorphic expansions (aka formality isomorphisms) for the Goldman-Turaev Lie bialgebra from the Kontsevich integral.
In outline, this follows from the naturality property of the construction above, under the associated graded functor, as follows. 

Given a short exact sequence 
\[\begin{tikzcd}
	0 & A & B & C & 0,
	\arrow[,from=1-1, to=1-2]
	\arrow["\iota",hook, from=1-2, to=1-3]
	\arrow["\pi",two heads, from=1-3, to=1-4]
	\arrow[from=1-4, to=1-5]
\end{tikzcd}\]
and a descending filtration on $B$ 
\begin{align*}
    B=B^0 \supseteq B^1 \supseteq B^2 \supseteq \dots \supseteq B^n \supseteq \dots, 
\end{align*}
there is an induced filtration on $A$ given by
\begin{align*}
    A=A^0 \supseteq A^1 \supseteq A^2 \supseteq \dots \supseteq A^n \supseteq \dots, 
\end{align*}
where $A^i = \iota^{-1}(\iota A \cap B^i)$. Similarly, there is an induced filtration on $C$ given by 
\begin{align*}
    C=C^0 \supseteq C^1 \supseteq C^2 \supseteq \dots \supseteq C^n \supseteq \dots 
\end{align*}
where $C^i = \pi(B^i)$. 

\begin{lem}
    If the rows of the diagram \eqref{eq:inducedconnhom} are exact and filtered so that the filtrations on the left and right are induced from the filtration in the middle, then the induced map $\eta$ is also filtered.
\end{lem}

\begin{proof}
    Basic diagram chasing: given $c\in C^n$, since $C^n=\pi(B^n)$, there is a $b\in B^n$ such that $\pi(b)=c$. Since $\lambda$ is filtered, $\lambda(b)\in E^n$, and $\lambda(b)\in \iota(D)$ by exactness. Since $D^n=\iota^{-1}(\iota(D)\cap E^n)$, we have that $\lambda(b)=\iota(d)$ for a $d\in D^n$. By uniqueness of the induced map, $d=\eta(c)$.
\end{proof}

The associated graded functor is a functor from the category of filtered algebras (or vector spaces) to the category of graded algebras (or vector spaces). For a filtered algebra 
\[
    A=A^0 \supseteq A^1 \supseteq A^2 \supseteq \dots \supseteq A^n \supseteq \dots, 
\]
the (degree completed) associated graded algebra is defined to be 
\[
\gr A= \Pi_{n=0}^\infty A^{n}/A^{n+1}.
\]
The associated graded map of a filtered map is defined in the natural way (as in the proof of Lemma~\ref{lem:ExaxtGr} below). In general, $\gr$ is not an exact functor, but it does preserve exactness for the special class of filtered short exact sequences where the filtrations on $A$ and $C$  are induced from the filtration on $B$:
\begin{lem}\label{lem:ExaxtGr}
    If in the filtered short exact sequence
\[\begin{tikzcd}
	0 & A & B & C & 0
	\arrow[,from=1-1, to=1-2]
	\arrow["\iota",hook, from=1-2, to=1-3]
	\arrow["\pi",two heads, from=1-3, to=1-4]
	\arrow[from=1-4, to=1-5]
\end{tikzcd}\]
the filtrations on $A$ and $C$ are induced from the filtration on $B$, then the associated graded sequence is also exact:
\[\begin{tikzcd}
	0 & \gr A & \gr B & \gr C & 0.
	\arrow[from=1-1, to=1-2]
	\arrow["\gr \iota", hook, from=1-2, to=1-3]
	\arrow["\gr \pi",two heads, from=1-3, to=1-4]
	\arrow[from=1-4, to=1-5]
\end{tikzcd}\]

\end{lem}
\begin{proof}
    Since $\gr$ is a functor, we know that $\gr \pi \circ \gr \iota = 0$, hence $\im \gr \iota \subseteq \operatorname{ker}\gr\pi$. It remains to show that $\operatorname{ker}\gr\pi\subseteq \operatorname{im}\gr\iota$. 

    Let $[b] \in B^n/B^{n+1}$, and assume that $\gr\pi([b]) = 0$. Since $\gr\pi([b]) = [\pi(b)]\in C^n/C^{n+1}$, we have $\gr\pi([b]) = 0$ if and only if $\pi(b)\in C^{n+1}$. 
    As the filtration on $C$ is induced from $B$, we know that $C^{n+1} = \pi(B^{n+1})$. Thus, $\pi(b) \in \pi(B^{n+1})$. Or in other words, there exists $x \in B^{n+1}$ such that $\pi(b) = \pi(x)$. This implies that $\pi(b-x) = 0$ and hence that $b-x \in \iota(A)$ by exactness.  

    Therefore, $b = x+\iota(a)$ for some $x \in B^{n+1}$ and $a \in A$. It follows that $[b] = [\iota(a)] = \gr\iota([a])$ in $B^n/B^{n+1}$ and hence $\operatorname{ker}\gr\pi \subseteq \operatorname{im}\gr\iota$ as required. 
\end{proof}

\begin{cor}\label{cor:gr_induced_is_unique}
    If the rows of the diagram in Equation~\ref{eq:inducedconnhom} are exact, and the filtrations on the left and right are induced from the filtration in the middle, then the rows of the associated graded diagram are also exact, and the unique connecting homomorphism is $\gr \eta$. 
\begin{equation}
\begin{tikzcd}
0\arrow[r] &\gr A \arrow[r]\arrow[d, "0",swap]  & 
\gr B
\arrow[r] \arrow{d}{\gr \lambda} & \gr C \arrow[d,"0"] \arrow[r] \arrow[dll, dashed, "\gr \eta", near start,swap, out=120, in=60, looseness=.9]&0\\
0\arrow[r]& 
\gr D\arrow[r]
& \gr E\arrow[r]
& \gr F \arrow[r] & 0
\end{tikzcd}
\end{equation} 
\end{cor}

\begin{proof}
    The exactness of the rows is Lemma~\ref{lem:ExaxtGr}. The induced map is $\gr \eta$ as $\gr \eta$ makes the diagram commute, and the induced map is unique.
\end{proof}

An expansion for an algebraic structure $X$ is a filtered homomorphism $Z: X \to \gr X$ (with special properties as explained in Section \ref{subsec:FramedKon}). Thus, if expansions exist for each of the spaces $A$ through $F$, we obtain a multi-cube:

\begin{equation}\label{eq:Cube}
\begin{tikzcd}[row sep=scriptsize, column sep=small]
& & A \arrow[dl] \arrow[rr] \arrow[dd,"Z_A",near start] & & B \arrow[dl,"\lambda"]\arrow[rr] \arrow[dd,"Z_B",near start] && C\arrow[r]\arrow[dd,"Z_C"]\arrow[dl] \arrow[dlllll, "\eta",swap, dashed, controls={+(-1.7,1.3) and +(0,2)}]  &0\\
0\arrow[r]&D \arrow[rr, crossing over] \arrow[dd,"Z_D"] & & E\arrow[rr,crossing over]\ &&F \\
& & \gr A \arrow[dl] \arrow[rr] & & \gr B \arrow[dl,"\gr \lambda", near start]\arrow[rr] && \gr C \arrow[r]\arrow[dl]\arrow[dlllll, "\gr \eta", dashed, out=-110,in=-50 , ,looseness=0.7]&0\\
0\arrow[r]& \gr D \arrow[rr] & & \gr E \arrow[rr]\arrow[from=uu, "Z_E",near start, crossing over]&& \gr F \arrow[from=uu, "Z_F",near start, crossing over]\\
\end{tikzcd}
\end{equation}


\begin{lem}\label{lem:Naturality} If, in the multi-cube~\eqref{eq:Cube} all vertical faces commute, then so does the square:
\begin{equation}\label{eq:HomExp}
\begin{tikzcd}
D \arrow[d, "Z_D"] && C \arrow[ll, dashed, "\eta"] \arrow[d, "Z_C"]  \\
\gr D && \gr C \arrow[ll,dashed, "\gr \eta"]
\end{tikzcd}
\end{equation}
\end{lem}
\begin{proof}
    Follows from the uniqueness of the induced maps.
\end{proof}

In Section \ref{sec:identifybracketinCON}, we will show how the Goldman bracket and Turaev cobracket each arise as induced maps $\eta$, where $\lambda=\lambda_1-\lambda_2$ is a difference of tangle operations.
Therefore the Kontsevich integral induces an expansion for the Goldman--Turaev operations, and 
the commutativity of the square \eqref{eq:HomExp} for each operation is -- by definition -- the homomorphicity property of the expansion. This homomorphicity is our main result. The non-trivial vertical face of the multi-cube is the one containing $\lambda$, and the commutativity of this for each Goldman-Turaev operation will follow from homomorphicity properties of the Kontsevich integral. Namely, the Kontsevich integral (standing in for $Z_B$ and $Z_E$) intertwines the appropriate tangle operations $\lambda_0$ and $\lambda_1$ with their associated graded counterparts. This is the idea behind the approach of this paper. 

\section{Preliminaries: Homomorphic expansions and the Goldman-Turaev Lie bialgebra}\label{sec:Prelims}

\subsection{Homomorphic expansions and the framed Kontsevich integral}\label{subsec:FramedKon}
The Kontsevich integral is the knot theoretic prototype of a {\em homomorphic expansion}. Homomorphic expansions (a.k.a. formality isomorphisms, well-behaved universal finite type invariants) provide a connection between knot theory and quantum algebra/Lie theory.
We begin with a short review of homomorphic expansions from an algebraic perspective, which is outlined -- in a slightly different, finitely presented case -- in \cite[Section 2]{WKO2}.

Kontsevich's original construction gives an invariant of unframed links; for a detailed introduction, we recommend \cite[Section 8]{CDM_2012}, or \cite{Kon, BN1, Da}. In this paper we work primarily with framed links and tangles, thus we briefly review the framed versions of the Vassiliev filtration and Kontsevich integral; for more detail see \cite[Sections 3.5 and 9.1]{CDM_2012} and \cite{LM96}. We also note that this paper uses Kontsevich's analytic construction of the Kontsevich integral, rather than the subsequent combinatorial construction from Drinfel'd associators. The reason is one of convenience: we limit the technical complexity of our arguments by avoiding parenthetisations. Nonetheless, our construction could be re-stated in terms of associators with limited effort.

\subsubsection{Homomorphic expansions.}\label{sec:hom_exp} Let $\glosm{caLK}{\calK}$ denote a given set of knots, links or tangles in $\R^3$ (e.g., oriented knots), and allow formal linear combinations with coefficients in $\C$. For links and tangles, allow only linear combinations of embeddings of the same skeleton\footnote{The {\em skeleton} of a knotted object is the underlying combinatorial object. For example: the skeleton of a link is the number of components; the skeleton of a braid is the underlying permutation; the skeleton of a tangle is the number of strands, connectivity, and number of circle components. In these contexts $\C\calK$ is a disjoint union of vector spaces, rather than a single vector space.}. The {\em Vassiliev filtration} (defined in terms of resolutions of double points $\doublepoint= \overcrossing - \undercrossing$) is a decreasing filtration on this linear extension: 
$$\C\calK = \calK_0 \supseteq \calK_1 \supseteq \calK_2 \supseteq ...$$

The degree completed associated graded space of $\C\calK$ with respect to the Vassiliev filtration is
$$\calA:= \prod_{n \geq 0} \calK_n/\calK_{n+1}.$$

An {\em expansion} is a filtered linear map  $Z: \C\calK \to \calA$, such that the associated graded map of $Z$ is the identity
$ \gr Z= id_\calA$. 

Usually, $\calK$ is equipped with additional operations: examples are knot connected sum, tangle composition, strand orientation reversal, etc. Homomorphic expansions are compatible with these operations, and thus allow for a study of $\calK$ via the more tractable associated graded spaces.

Specifically, an expansion is {\em homomorphic} with respect to an operation $m$, if it intertwines $m$ with its associated graded operation on $\calA$. That is, $Z \circ m = \gr m \circ Z$. 
A crucial step towards making effective use of this machinery is to get a handle on the space $\calA$ in concrete terms: for example, in classical knot theory, $\calA$ has a combinatorial description as a space of {\em chord diagrams} \cite[Chapter 4]{CDM_2012}. 

\begin{figure}
\centering
\begin{picture}(200,60)
\put(-10,0){\includegraphics[scale=.7]{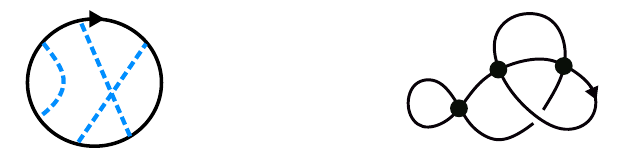}}
\put(80,20){\Large{$\mapsto$}}
\put(84,30){{$\psi$}}
\end{picture}
\caption{Example of $\psi$ mapping a chord diagram to a knot with double points by contracting the chords. The right-hand side represents a well-defined element in $\calK_3/\calK_{4}$.}
\label{fig:psionchord}
\end{figure}

There is a natural map $\psi$ from chord diagrams with $i$ chords to $\calK_i/\calK_{i+1}$, defined by ``contracting chords'' as in Figure~\ref{fig:psionchord}. It is not difficult to establish that $\psi$ is surjective. 
In the case of classical (oriented, unframed) knots, there are two relations in the kernel of $\psi$: the 4-Term (4T) and Framing Independence (FI) relations, shown in Figure~\ref{fig:4TFI}.   In fact, these two relations generate the kernel, and $\psi$ descends to an isomorphism on the quotient; this, however, is significantly harder to prove.

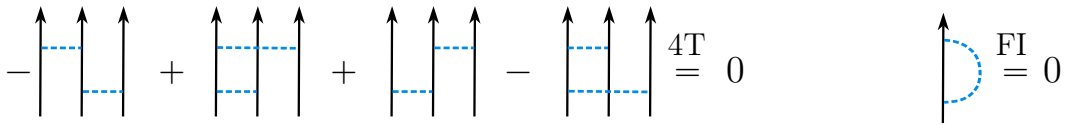
\begin{figure}[ht]
\centering
\begin{picture}(300,70)
\put(-45,0){\includegraphics[scale=.7]{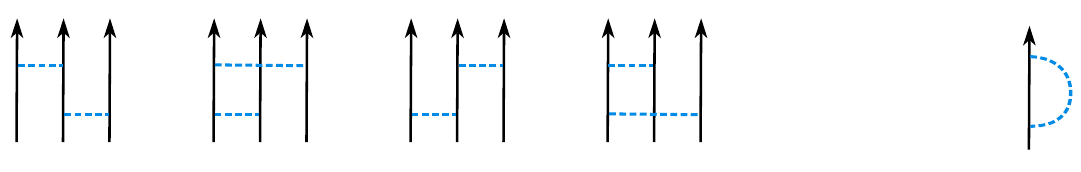}}
\put(340,23){\Large{$0$}}
\put(322,33){\large{FI}}
\put(324,23){\Large{$=$}}
\put(220,23){\Large{$0$}}
\put(198,33){\large{4T}}
\put(200,23){\Large{$=$}}
\put(136,23){\Large{$-$}}
\put(70,23){\Large{$+$}}
\put(5,23){\Large{$+$}}
\put(-53,23){\Large{$-$}}
\end{picture}
\caption{The 4T and FI relations, understood as local relations: the strand(s) are part(s) of the skeleton circle, and the skeleton may support additional chords outside the picture shown.}
\label{fig:4TFI}
\end{figure}

The key technique is to construct an expansion as in the following Lemma, \cite[Proposition 2.7]{WKO2}: 

\begin{lem}\cite{WKO2} \label{lem:assocgradyoga}
Let $\C\calK$ be a filtered vector space (or union of vector spaces), and $\calA$ the associated graded space of $\C\calK$. Let $\calC$ be a ``candidate model'' for $\calA$: a graded linear space equipped with a surjective homogeneous map $\psi: \calC \to \calA$. If there exists a filtered map $Z: \C\calK \to \calC$, such that $\gr Z \circ \psi = id_{\calC}$, then $\psi$ is an isomorphism and $\psi\circ Z$ is an expansion for $\calK$.

\[ \begin{tikzcd}
\C\calK \arrow{r}{Z} & \calC \arrow[dd,"\psi"] &&&   & \calC \arrow[dd,"\psi"]\arrow[ddrr,dashed, "\gr Z\circ \psi=id_{\calC}" right]\\
&&\text{ }\arrow[r,Mapsto, "\gr"] &\text{ }&&\\
& \calA&&&& \calA\arrow[rr,"\gr Z"] &&\calC
\end{tikzcd}
\]

\end{lem}

In other words, once one finds a candidate model $\calC$ for $\calA$, finding an {\em expansion valued in} $\calC$ also implies that $\psi$ is an isomorphism.
In classical Vassiliev theory, $\calK$ is the space of oriented knots, $\calC$ is the space of chord diagrams, and a $\calC$-valued expansion is the Kontsevich integral \cite{Kon}.

\subsubsection{Framed theory.}\label{subsubsec:Framing} In this paper we work with {\em framed} links and tangles, so we give a brief introduction to the framed version of the general theory summarised in the previous section. For simplicity, we consider links for now.

Let $\glosm{claK}{\tcalK}$ denote the set of {\em framed} links in $\mathbb{R}^3$: that is, links along with a non-zero section of the normal bundle. A link diagram is interpreted as a framed link using the blackboard framing. The Reidemeister move R1 move changes the blackboard framing, and by ommitting it, one obtains a Reidemeister theory for framed links. In analogy with a double point, a {\em framing change} is defined to be the difference
$$\frchange:= \poskink -\arrowup.$$
The framed Vassiliev filtration is the descending filtration
$$\tcalK=\tcalK_0 \supseteq \tcalK_1 \supseteq \tcalK_2\supseteq ...$$
where $\glosm{tcalKi}{\tcalK_i}$ is linearly generated by knots with at least $i$ double points {\em or framing changes}. The degree completed associated graded space of $\tcalK$ with respect to the framed Vassiliev filtration is
$$\tcalA:= \prod_{n \geq 0} \tcalK_n/\tcalK_{n+1}.$$

A natural first guess for a combinatorial description of $\tilde{\calA}$ is in terms of chord diagrams with ``framing change markings'' $\circarrow$ on the skeleton, graded by the number of chords and markings. There is a natural surjective graded map $\tilde{\psi}$ from marked chord diagrams onto $\tilde{\calA},$ which contracts chords as in the classical case, and which replaces each marking $\circarrow$ with a framing change $\frchange$. The kernel of $\tilde{\psi}$
includes the $4T$ relation as before. 

In place of the $FI$ relation (\shortchord=0), a weaker relation arises from the equality $\poskink -\negkink = \doublekink$ in $\tcalK$. In fact, $\doublekink=\poskink-\negkink=(\poskink-\arrowup)+(\arrowup-\negkink)$, and $\arrowup-\negkink = \poskink -\arrowup$ modulo $\tcalK_2$. In other words, the following relation is in the kernel of $\tilde{\psi}$:
$$\shortchord=2\circarrow.$$
Therefore, it is not necessary to have dedicated notation for the framing change markings, since $\circarrow=\frac{1}{2}\shortchord$. The candidate model for the associated graded space is simply chord diagrams modulo the $4T$ relation, and no $FI$ relation. We denote this space by $\tcalC$.

To show that $\tilde{\psi}: \tcalC \to \tcalA$ is an isomorphism, it is enough to construct a $\tcalC$-valued expansion and use Lemma~\ref{lem:assocgradyoga}. This $\tcalC$-valued expansion is the framed version $\tZ$ of the Kontsevich integral. For details of this construction see \cite[Section 9.1]{CDM_2012}, or \cite{LM96,Gor}.

\subsection{The Goldman-Turaev Lie bialgebra}\label{subsec:IntroGT} In order to define the Goldman-Turaev Lie bialgebra, we need to recall some basic definitions and notation.

Let $\glosm{Dp}{D_p}$ denote $p$-punctured disc, with $p+1$ circle boundary components $\partial _0, \partial_1,...,\partial_p$, embedded in the complex plane so that $\partial_0$ is the outer boundary, as in Figure~\ref{fig:DP}. In particular, the plane-embedding specifies a framing (trivialisation of the tangent bundle) on $D_p$, and thus immersed loops in $D_p$ are equipped with a notion of {\em rotation number}.

Let $\glosm{pi}{\pi=\pi_1(D_p,*)}$ denote the fundamental group of $D_p$ with basepoint $*\in \partial_0$. We denote by $\glosm{Cpi}{\C\pi}$ the group algebra of $\pi$. 

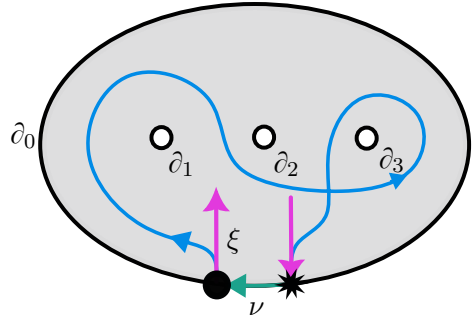
\begin{figure}
    \centering
    \begin{tikzpicture}[scale=1]
    \draw (0, 0) node {\includegraphics[width=6cm]{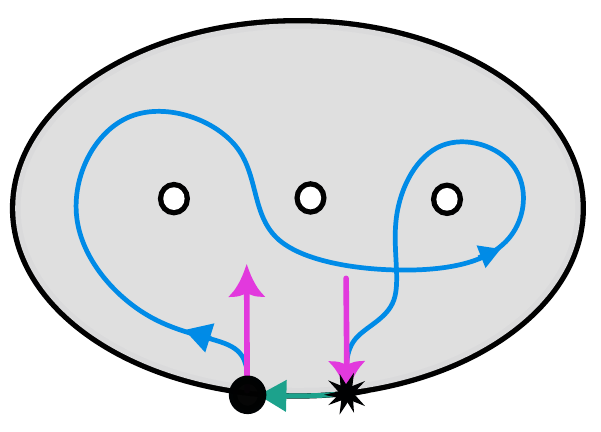}};
    \draw (-3.1, 0.2) node {$\partial_0$};
    \draw (-1,-.185) node {$\partial_1$};
    \draw (.4,-.185) node {$\partial_2$};
    \draw (1.8, -.13) node {$\partial_3$};
    \draw (-.3, -1.2) node {$\xi$};
    \draw (0,-2.1) node {$\nu$};
    \end{tikzpicture}
    \caption{$D_3$ with an immersed loop from $\bullet$ to $*$ with initial tangent vector $\xi$ and terminal tangent vector $-\xi$. The path along the boundary from $*$ to $\bullet$ is $\nu$.}
    \label{fig:DP}
\end{figure}

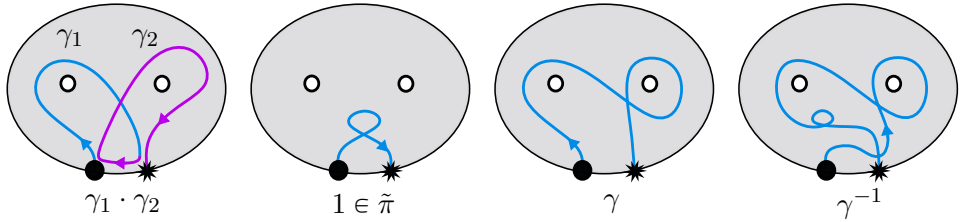
\begin{figure}
\begin{tikzpicture}
\draw (0, 0) node {\includegraphics[width=13cm]{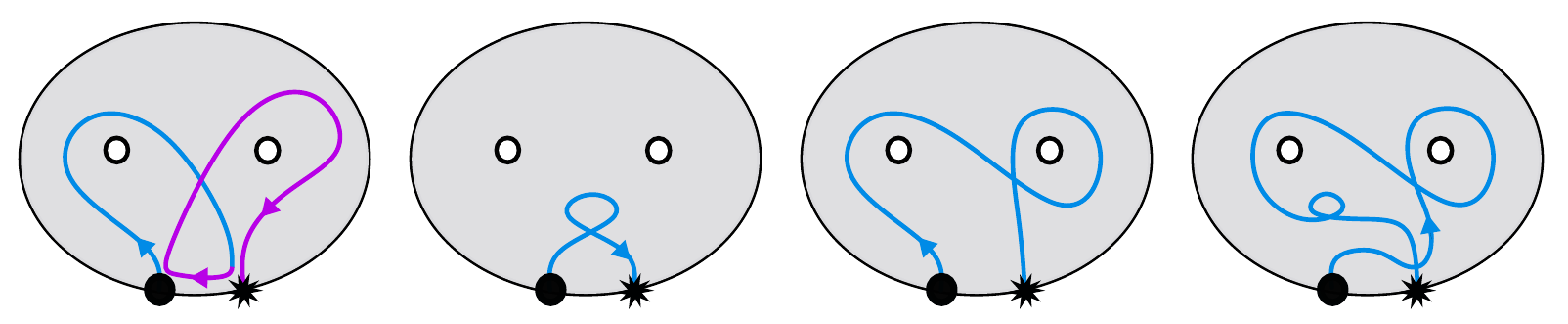}};
\draw(-4.8,-1.5) node {$\gamma_1\cdot \gamma_2$};
\draw(-5.5,.7) node {$\gamma_1$};
\draw(-4.5,.7) node {$\gamma_2$};
\draw(-1.6,-1.5) node {$1\in \tpi$};
\draw(1.7,-1.5) node {$\gamma$};
\draw(5,-1.5) node {$\gamma^{-1}$};
\end{tikzpicture}
\caption{The group structure on $\tpi$.}\label{fig:DPGroup}
\end{figure}

We also need to consider based paths.
Let $\bullet$ and $*$ be two ``nearby'' basepoints on $\partial_0$ and $\glosm{xi}{\xi}$ be the direction of the inward pointing normal vector to $\partial_0$ at $\bullet$ and $*$. Let $\tpi=\tpi_{\bullet*}$ denote the set of regular homotopy classes of immersed curves $\gamma: ([0,1],0,1) \to (D_p, \bullet,*)$, so that $\dot{\gamma}(0)=\xi$, and $\dot{\gamma}(1)=-\xi$, as shown in Figure~\ref{fig:DP}. Note that the rotation number is invariant under regular homotopy.  Recall that $\glosm{tpi}{\tpi}$ is in fact a group, illustrated in Figure~\ref{fig:DPGroup} and defined as follows: 
\begin{enumerate} 
\item Let $\glosm{nu}{\nu}$ denote the path from $*$ to $\bullet$ along $\partial_0$ (in the negative direction). The group product $\gamma_1 \cdot \gamma_2$ is the smooth concatenation of $\gamma_1$ with $\nu$ followed by $\gamma_2$. 
\item The group identity is the class of paths which, when composed with $\nu$, become contractible loops of rotation number zero.
\item The inverse of $\gamma$ is the concatenation $\overline{\nu} \, \overline{\gamma} \,\nu^*$ where the overline denotes the reverse path, and $\nu^*$ includes a negative twist (to ensure that the rotation number of $\gamma \cdot \gamma^{-1}$ is 0). The beginning and end of the path is adjusted in an epsilon neighbourhood of the base points to have inward and outward pointing tangent vectors, as in Figure~\ref{fig:DPGroup}.
\end{enumerate}
Denote by $\glosm{Ctpi}{\C\tpi}$ the group algebra of $\tpi$. There is a forgetful map $\tpi \to \pi$ which maps $\gamma$ to the (non-regular) homotopy class of $\gamma \, \nu$. This linearly extends to a forgetful map $\C\tpi \to \C \pi$.

For an algebra $A$ we denote by $|A|$ the \emph{linear}\footnote{Not to be confused with the abelianisation of $A$. In particular, $|A|$ does not inherit an algebra structure from $A$.} quotient $A/[A,A]$, where $[A,A]$ denotes the subspace spanned by commutators $[x,y]=xy-yx$ for $x,y \in A$. 
We denote the quotient (trace) map by $|\cdot|:A\to |A|$. In our context,  $\glosm{aCpi}{|\Cp|}$ has an explicit description as the $\C$-vector space generated by homotopy classes of free loops in $D_p$. In a similar but more subtle fashion, $\glosm{abstCpi}{|\tCp|}$ is spanned by {\em regular} homotopy classes of immersed free loops, where $|\gamma|$ denotes the class of $\gamma\nu$ as a free immersed loop.

The Goldman--Turaev Lie bialgebra comes in two flavours: {\em original} and {\em enhanced}. The original construction of the Goldman bracket is a Lie bracket on $|\Cp|$. However, the original Turaev cobracket is only well-defined on $\glosm{Cpba}{\Cpba=|\Cp|/\C\cl}$, the linear quotient by the homotopy class of the constant loop. The space $\Cpba$ is a Lie bialgebra with this cobracket and the Goldman bracket, which descends from $|\Cp|$.
There is an enhancement of the cobracket (\cite{akkn_g0, Mas} following earlier work of Turaev \cite{Tur78}), which promotes the cobracket to $|\Cp|$, thereby making $|\Cp|$ a Lie bialgebra under the Goldman bracket and the enhanced cobracket. In \cite{akkn_g0} this enhancement is necessary in order to establish the relationship between the Goldman-Turaev Lie bialgebra and Kashiwara--Vergne theory. To define the enhanced cobracket, a curve in $|\Cp|$ is lifted to an immersed curve with a fixed rotation number. Below we review the definitions of the Goldman bracket and the enhanced version of the Turaev cobracket.


The Goldman Bracket sums over smoothing intersections between two free loops.
For a free loop $\alpha$ in $\Cpa$ and a point $q$ on $\alpha$, denote by $\alpha_q$ the loop $\alpha$ based at $q$.

\begin{definition}[The Goldman bracket]\label{def:bracket}
 Let $\alpha, \beta \in |\C\pi|$ be free loops with homotopy representatives chosen so that there are only finitely many transverse double intersections between $\alpha$ and $\beta$. The Goldman bracket 
 $\glosm{Gbrack}{[\cdot,\cdot]_G}:|\C\pi|\otimes |\C\pi|\rightarrow |\C\pi|$ is given by
$$
    [\alpha,\beta]_G \coloneqq -\sum_{q \in \alpha \cap \beta} \varepsilon_q |\alpha_q\beta_q|,
$$
where $\varepsilon_q=\varepsilon(\dot\alpha_q,\dot\beta_q)\in\{\pm 1\}$ is the local intersection number of $\alpha$ and $\beta$ at $q$, $\alpha_q\beta_q$ is the concatenation of $\alpha_q$ and $\beta_q$, and the extension to $|\Cp|$ is linear. Then one easily checks that $[\cdot, \cdot]_G$ is a Lie bracket on $|\Cp|$. \end{definition}

The original definition of the Turaev cobracket is similar, but uses self intersections of a curve in place of the intersections between two curves. Unfortunately, it is not well-defined with respect to the Reidemeister 1 relation for free homotopy curves, hence the need for the enhancement. We construct the (enhanced) cobracket via a self-intersection map for \emph{based} curves: this map was first described by Turaev in \cite{Tur78}, and is explained in \cite[Section 3.2]{Mas} and in \cite[Section 5.2]{akkn_g0}; this definition lends itself well to direct comparison with the three-dimensional operations of Section~\ref{sec:IdentifyingGTinCON}. For a based curve $\gamma$ in $\Cp$, the idea is to ``snip off'' portions of $\gamma$ at self intersection points to get two curves, one of which is based and the other free.
Figure \ref{fig:defmu} shows an example.

\begin{figure}
 \centering
    \begin{tikzpicture}
    \draw (-2, 0) node {\includegraphics[scale=.4]{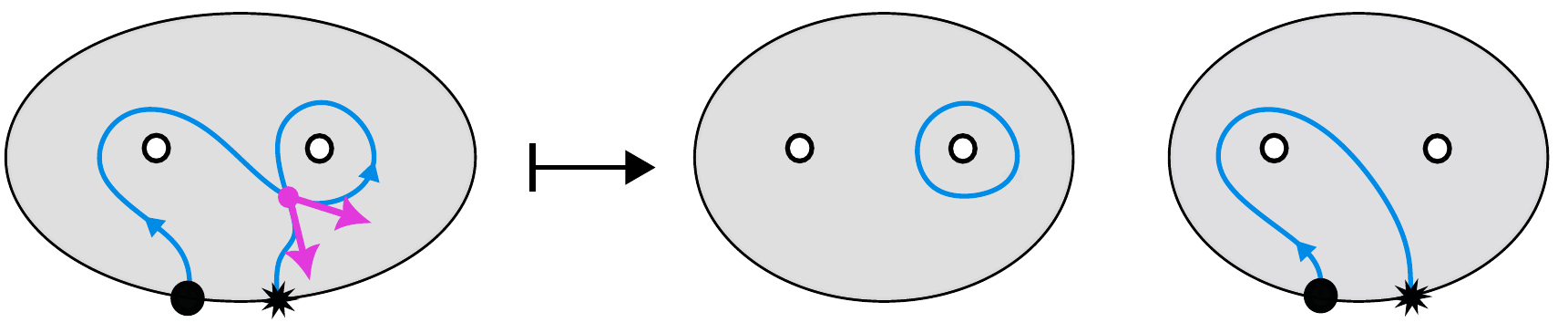}};
    \draw (-5.9, 0) node {$p$};
    \draw (-4.84, -.4) node {$\dot{\gamma}_1^p$};
    \draw (-6,-.6) node {$\dot{\gamma}_2^p$};
    \draw (-3.5,.4) node {$\mu$};
    \draw (.5, 0) node {\Large{$\otimes$}};
    \end{tikzpicture}
    \caption{Example of the self intersection map $\mu$ where $\epsilon_p=-1$.}
    \label{fig:defmu}
\end{figure}
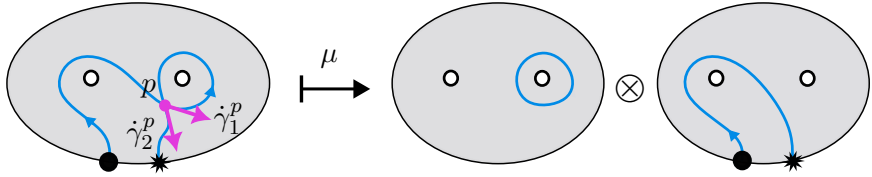

\begin{definition}[The self-intersection map]\label{def:mu} For $\gamma \in \Cp$, let $\tilde{\gamma}\in \Ctp$ denote a path such that $\tilde{\gamma}\nu$ is homotopic to $\gamma$; and such that $\tilde{\gamma}$ has only transverse double points, and $\rot(\tilde{\gamma})=1/2$ (hence, $\rot(\tilde\gamma\nu)=0$). Let $\tilde{\gamma} \cap \tilde{\gamma}$ denote the set of double points. 
The self intersection map $\glosm{mu}{\mu}$ is defined as follows:  
$$\mu:{\C\pi}\to \Cpa\otimes {\C\pi}$$ 
\[\mu(\gamma)=-\sum_{p\in \tilde\gamma\cap\tilde\gamma} \varepsilon_p|\tilde\gamma_{t^p_1t_2^p}|\otimes \tilde\gamma_{0t_1^p}\tilde\gamma_{t_2^p1},\]
where $t_1^p$ and $t_2^p$ are the first and second time parameter in $[0,1]$  where $\tilde \gamma$ goes through $p$; where $\tilde\gamma_{rs}$ denotes the path traced by $\tilde\gamma$ from $t=r$ to $t=s$; the sign $\varepsilon_p=\varepsilon\left(\dot{\tilde\gamma}(t_1^p),\dot{\tilde\gamma}(t_2^p)\right)\in \{\pm 1\}$ is the local self-intersection number; and the formula extends to $\Cp$ linearly.
\end{definition}

The Turaev cobracket is obtained from $\mu$ by closing off the path component and making the tensor product alternating: this descends to a map on $\Cpa$, as follows.

\begin{definition}\label{def:cobrac}(The Turaev co-bracket)
The Turaev cobracket $\glosm{del}{\delta}$ is the unique linear map which makes the following diagram commute, where $\glosm{alt}{\operatorname{Alt}}(x\otimes y)= x\otimes y - y \otimes x=x\wedge y$, and $\glosm{deltilde}{\tilde{\delta}}$ denotes the composition of $\mu$ with closure and alternation, as shown:
\[
\begin{tikzcd}[ every label/.append style = {font = \normalsize}]
 & 
{\C\pi}
\arrow[rr, "\mu"]
\arrow[dd, "|\cdot|"] 
\arrow[ddrrrr, "\tilde{\delta}"]
&&
\protect{\Cpa \otimes {\C\pi}} \arrow[rr, "1\otimes |\cdot|"]
&&
\protect{\Cpa \otimes \Cpa} \arrow[dd, "\operatorname{Alt}"] 
\\
&&&&& 
\\
&\protect{\Cpa} \arrow[rrrr, dashed, "\delta"] 
& & & &
\protect{\Cpa \wedge \Cpa}
\end{tikzcd}
\]

\end{definition}

\begin{remark}\label{rem:MuFraming}
    The definition of the self-intersection map $\mu$ in \cite[Section 5.2]{akkn_g0} is slightly different from ours in that it requires the curve $\tilde{\gamma}$ to have rotation number $-1/2$. Hence, in order to obtain a rotation number 0 curve when closing $\tilde{\gamma}$, a positive bigon needs to be inserted, which results in a framing correction term $1\wedge |\gamma|$ when $\delta$ is constructed from $\mu$. Since our definition of $\mu$ uses a different choice of rotation number, this correction term doesn't appear.
\end{remark}

\subsection{Associated graded Goldman-Turaev Lie bialgebra}\label{sec:gr_bialgebra}

The I-adic filtration on $\Cp$ is the filtration by powers of the augmentation ideal $\glosm{I}{\calI}=\langle \{\alpha - 1\}_{\alpha \in \pi} \rangle$:
$$\C\pi = \calI^0 \supseteq \calI \supseteq \calI^2 \supseteq...$$ 
By the 1930's work of Magnus \cite{Mag}, the associated graded algebra of $\Cp$ with respect to this filtration is the degree completed free algebra $\glosm{FA}{\As }= \As\langle x_1, \cdots, x_p\rangle$:

\begin{prop}
    Given the set of standard generators $\{\gamma_i\}_{i=1}^p$ for $\pi$, there is an isomorphism of algebras $\gr\Cp\to \As$ and the exponential expansion $\varphi(\gamma_i^{\pm 1})=e^{\pm x_i}$ is a homomorphic expansion.
\end{prop}

The I-adic filtration of $\Cp$ descends to a filtration on $|\Cp|$: $$|\C\pi|=|\calI^0|\supseteq |\calI| \supseteq |\calI^2| \supseteq...$$ 
The completed associated graded vector space for $|\C\pi|$ with respect to this filtration is, by definition
$$\yellowm{\gr |\C\pi|}=\prod_{n=0}^\infty |\calI^n|/|\calI^{n+1}|.$$ 
There is an isomorphism 
$\gr|\Cp| \cong |\As|,$ where $\glosm{absFA}{|\As|}$ denotes the linear quotient $|\As|=\As/[\As,\As]$, and the exponential expansion descends to a homomorphic expansion for $|\Cp|$. The vector space $|\As|$ is spanned by cyclic words in letters $x_1, \cdots, x_p$, that is, words modulo cyclic permutations of the letters. 

Therefore, $|\As|$ carries the structure of a Lie bialgebra under $\gr[\cdot,\cdot]_G$ and $\gr \delta$ \cite[Section 3]{AKKN_highergenus}. This is also known as the {\em necklace} Lie bialgebra; the cobracket was originally found by \cite{Schedler}; the fact that the two coincide was observed by Massuyeau \cite[Sections 5.4 and 7]{Mas}.

Note that the Goldman bracket and the Turaev co-bracket are not strictly filtered maps, as they both shift filtered degree down by one\footnote{In \cite[Sections 3.3, 3.4]{AKKN_highergenus} the down-shifts are by up to two filtered degrees, as the generating curves around genera and those around boundary components carry different weights. In our genus zero setting this translates to a degree shift of $-1$.}. For example, if $x\in |\calI^r|$ and $y\in |\calI^s|$, then $[x,y]_G \in |\calI^{r+s-1}|$. Correspondingly, the associated graded operations are maps of degree $-1$.

Figure~\ref{fig:grbracket} shows a schematic calculation of the graded Goldman bracket, with cyclic words represented diagrammatically as letters along a circle. The graded Goldman bracket sums over matching pairs of letters in the two words $z$ and $w$ which the bracket is being applied to, joins the circles at the matching letter, and takes the difference of the two ways of including only one copy of the letter in the new cyclic word. Stated algebraically, this is summarised as follows:

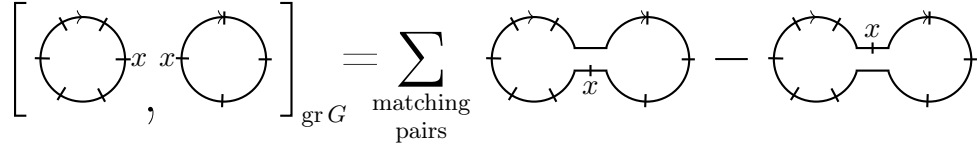
\begin{figure}
\begin{tikzpicture}[scale=.75]
\begin{scope}[xshift=.5cm]
\draw[thick] (0,0) circle (.75cm);
\draw[->] (0,.75)--(.01,.75);
\draw[thick, decorate,decoration= {border,amplitude=0.2cm, angle=90,segment length=pi*6}] (0,0) circle (.85);
\draw node at (1,0) {$x$};
\draw[thick] (-1,-1)--(-1.2,-1)--(-1.2,1)--(-1,1);
\end{scope}

\draw[thick] (3,0) circle (.75cm);
\draw[->] (3,.75)--(3.01,.75);
\draw[thick, decorate,decoration= {border,amplitude=0.175cm, angle=90,segment length=pi*9}] (3,0) circle (.85);
\draw node at (2,0) {$x$};
\draw[thick] (4,-1)--(4.2,-1)--(4.2,1)--(4,1);
\draw node[right] at (4.2,-1){\Small $\operatorname{gr}G$};
\draw node at (1.7,-1){\Huge ,};
\draw node at (6,0){\huge = $\sum$};
\draw node at (6.5, -.8){\Small matching };
\draw node at (6.5, -1.3){ \Small pairs  };
\draw node at (12,0){\huge  $-$};

\begin{scope}[xshift=8.5cm]
\draw[thick] (0,0) circle (.75cm);
\draw[->] (0,.75)--(.01,.75);
\draw[thick, decorate,decoration= {border,amplitude=0.2cm, angle=90,segment length=pi*6}] (0,0) circle (.85);
\draw[thick] (2,0) circle (.75cm);
\draw[->] (2,.75)--(2.01,.75);
\draw[thick, decorate,decoration= {border,amplitude=0.175cm, angle=90,segment length=pi*9}] (2,0) circle (.85);
\draw[line width=9, white] (0,0)--(2.25,0);
\draw[thick] (.7,.2)--(1.3,.2);
\draw[thick] (.7,-.2)--(1.3,-.2);
\draw[thick] (1,-.3)--node[below]{$x$}(1,-.1);
\end{scope}

\begin{scope}[xshift=13.5cm]
\draw[thick] (0,0) circle (.75cm);
\draw[->] (0,.75)--(.01,.75);
\draw[thick, decorate,decoration= {border,amplitude=0.2cm, angle=90,segment length=pi*6}] (0,0) circle (.85);
\draw[thick] (2,0) circle (.75cm);
\draw[->] (2,.75)--(2.01,.75);
\draw[thick, decorate,decoration= {border,amplitude=0.175cm, angle=90,segment length=pi*9}] (2,0) circle (.85);
\draw[line width=9, white] (0,0)--(2.25,0);
\draw[thick] (.7,.2)--(1.3,.2);
\draw[thick] (.7,-.2)--(1.3,-.2);
\draw[thick] (1,.3)--node[above]{$x$}(1,.1);
\end{scope}

\end{tikzpicture}
\caption{A schematic diagrammatic example of the graded Goldman bracket.}\label{fig:grbracket}
\end{figure}

\begin{prop}\cite[Section 3.3]{AKKN_highergenus}\label{prop:gr_bracket_def}
 Let $z = |z_1\cdots z_l|$ and $w = |w_1\cdots w_m|$ be two cyclic words in $|\As|$. The graded Goldman bracket $$\gr\left([-,-]_G\right)=\glosm{grGbrac}{[-,-]_{\operatorname{gr}G}}:|\As|\otimes |\As|\to |\As|$$ is given by:
\begin{align*}
        [z,w]_{\operatorname{gr}G} = \sum_{j,k} \delta_{z_j,w_k} (|w_1&\dots w_{k-1}z_{j+1}\dots z_l z_1\dots z_j w_{k+1}\dots w_m| -\\& |w_1\dots w_{k-1}z_j \dots z_lz_1\dots z_{j-1}w_{k+1} \dots w_m|),
\end{align*} 
where $\delta_{z_j,w_k}$ is the Kronecker delta.
\end{prop}


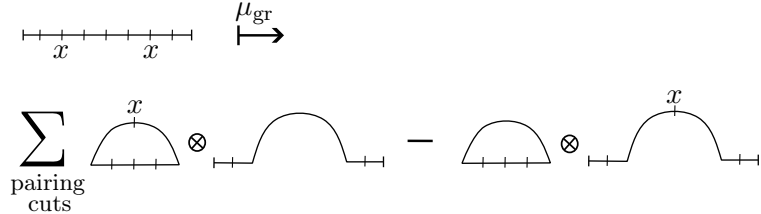
\begin{figure}
    \centering
    \begin{picture}(300,70)
    \put(0,10){\includegraphics[scale=.5]{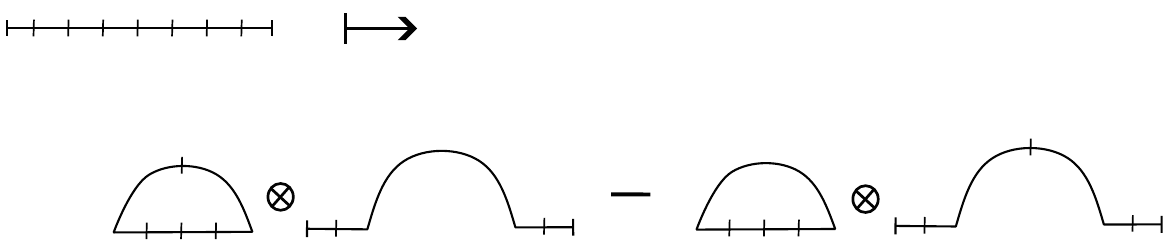}}
    \put(13,54){$x$}
    \put(47,54){$x$}
    \put(41,33){$x$}
    \put(245,37){$x$}
    \put(82,70){$\mu_{\gr}$}
    \put(0,18){\huge  $\sum$}
    \put(-3, 4){\Small pairing }
    \put(0, -4){ \Small cuts  }
    \end{picture}
    \caption{A schematic diagrammatic example of the graded Self-intersection map, $\gr\mu$.}
    \label{fig:grmu}
\end{figure}

Figure~\ref{fig:grmu} shows a schematic diagrammatic calculation of the graded self-intersection map $\mu_{\gr}$, as a sum over \emph{pairing cuts}. A pairing cut identifies two matching letters in a word, and splits the word along a chord connecting these matching letters. The graded self-intersection map outputs the tensor product of the resulting cyclic word and the remainder of the associative word. The choice of rotation number in the definition of $\mu$ (Remark~\ref{rem:MuFraming}) makes no difference in the associated graded $\mu_{\gr}$, since the framing correction term is in filtration degree 0. In summary:

\begin{prop}\label{prop:gr_mu}\cite[Section 3.4]{AKKN_highergenus}
Let $w=w_1\dots w_m \in FA$. The graded self-intersection map
$$\gr(\mu)=\glosm{grmu}{\mu_{\gr}}: \As\rightarrow |\As|\otimes \As$$ 
 is given by:
\begin{align*}
    \mu_{\operatorname{gr}}(w) = \sum_{j<k} \delta_{w_j,w_k}( |w_j&\dots w_{k-1}|\otimes  w_1 \dots w_{j-1} w_{k+1}...w_m - \\& |w_{j+1}\dots w_{k-1}| \otimes w_1 \dots w_jw_{k+1}\dots w_m), 
\end{align*}
where $\delta_{w_j,w_k}$ denotes the Kronecker delta.
\end{prop}

Figure~\ref{fig:paircut}(A.) shows a schematic diagrammatic definition of the graded Turaev co-bracket, again as a sum over \emph{pairing cuts}. A pairing cut in a cyclic word identifies a pair of coinciding letters, and cuts the cycle into two cycles along the chord connecting the matching letters. To obtain the cobracket, one takes a sum of wedge products of the resulting split cyclic words, adding one copy of the coinciding letter to either side, as shown in Figure~\ref{fig:paircut}(B.) and expressed in formulas below:

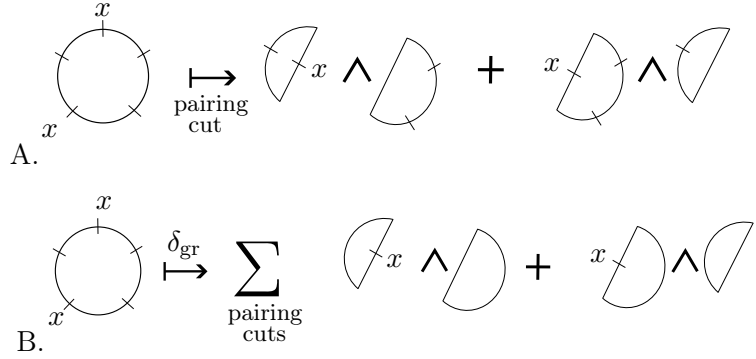
\begin{figure}
\begin{picture}(275,70)
    \put(0,10){\includegraphics[scale=.5]{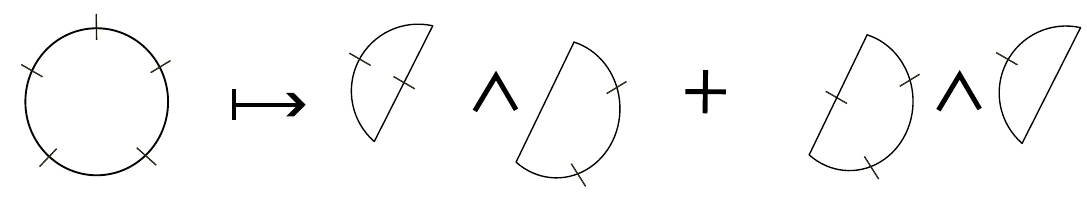}}
    \put(-12,0){A.}
    \put(20,56){$x$}
    \put(0,10){$x$}
    \put(50, 20){\Small pairing }
    \put(52, 12){ \Small cut  }
    \put(102,33){$x$}
    \put(190,35){$x$}
    \end{picture}
    \begin{picture}(270,70)
    \put(0,10){\includegraphics[scale=.47]{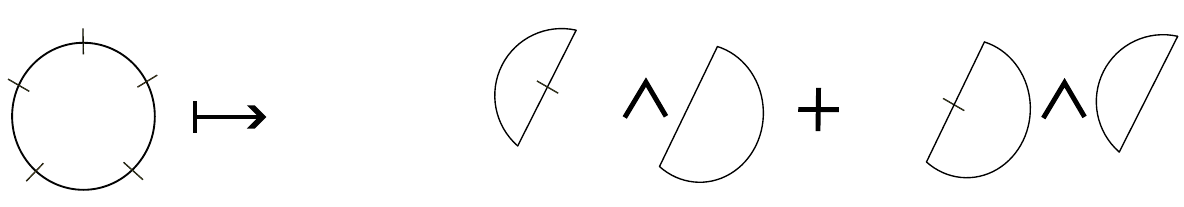}}
    \put(-12,0){B.}
    \put(18,54){$x$}
    \put(0,10){$x$}
    \put(45,38){$\delta_{\gr}$}
    \put(70,26){\huge  $\sum$}
    \put(68, 12){\Small pairing }
    \put(70, 4){ \Small cuts  }
    \put(128,32){$x$}
    \put(205,34){$x$}
    \end{picture}
\caption{(A.) An example pairing cut of a cyclic word. (B.) An example of the graded Turaev cobracket as a sum over pairing cuts of a cyclic word.}\label{fig:paircut}
\end{figure}

\begin{prop}\label{prop:gr_del} \cite[Section 3.4]{AKKN_highergenus}
Let $w=w_1\dots w_m \in |FA|$. The graded Turaev cobracket  $$\gr(\delta)=\glosm{grdel}{\delta_{\operatorname{gr}}}:
|\As|\rightarrow |\As|\wedge |\As|$$ 
is given by
\begin{align*}
    \delta_{\operatorname{gr}}(w) = \sum_{j<k} \delta_{w_j,w_k}( |w_j&\dots w_{k-1}|\wedge  |w_{k+1}\dots w_m w_1 \dots w_{j-1}| +\\ &|w_k\dots w_m w_1 \dots w_{j-1}|\wedge |w_{j+1}\dots w_{k-1}|), 
\end{align*}
where $\delta_{w_j,w_k}$ denotes the Kronecker delta\footnote{Apologies for the notation clash.}.
\end{prop}


\section{Expansions for tangles in handlebodies}\label{sec:TangleSetUp}

\subsection{Framed oriented tangles}\label{sec:framed_tangles}
This section introduces the space $\glosm{ctT}{\ctT}$ of framed, oriented tangles in a genus $p$ handlebody, with formal linear combinations. Our main result -- proven in Section~\ref{sec:IdentifyingGTinCON} -- is that homomorphic expansions on $\ctT$ induce  homomorphic expansions on the Goldman-Turaev Lie biagebra.

Let $\glosm{Mp}{M_p}$ denote the manifold $ D_p\times I$ where $D_p$ is a disc in the complex plane with $p$ points removed.  While $M_p$ is not a compact manifold, knot theory in $M_p$ is equivalent to knot theory in a genus $p$ handlebody. For the purpose of the Kontsevich integral, we identify $D_p$ with a unit square $[0,1]+[0,i]$ in the complex plane with $p$ points removed, so $M_p$ can be drawn as a cube with $p$ vertical lines removed; we call these lines {\em poles}, as shown in the middle in Figure \ref{fig:polestudio}. We refer to $D_p \times \{0\}$ as the ``floor'' or ``bottom'', and $D_p\times \{1\}$ as the ``ceiling'' or ``top''. The ``back wall'' is the face $[i, i+1]\times [0,1]$. We refer to the $i\in \C$ direction as North. 

\begin{figure}
    \centering
    \includegraphics[scale=.7]{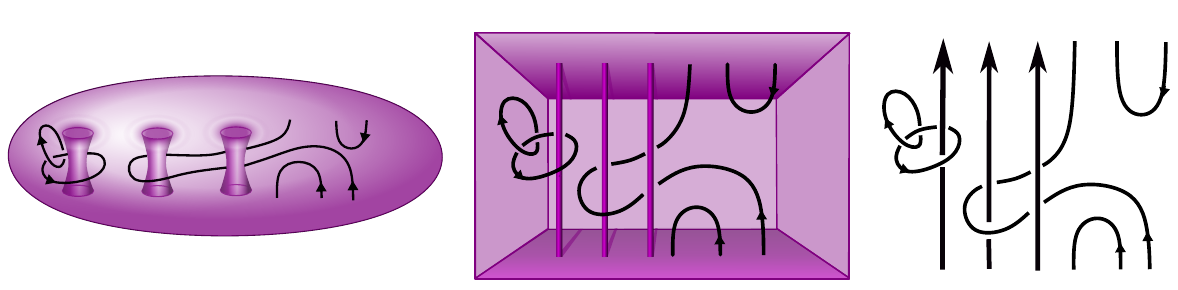}
    \caption{An example of a tangle in $M_3$, drawn first in a handlebody, then in a cube with poles, and lastly as a tangle diagram projected to the back wall of the cube.}
    \label{fig:polestudio}
\end{figure} 

\begin{definition}\label{def:tangle}
An oriented tangle $T$ in $M_p$ is an embedding of an oriented compact 1-manifold 
$$(S, \partial S) \hookrightarrow (M_p, D_p\times\{0\}\cup D_p\times\{1\}).$$
The interior of $S$ lies in the interior of $M_p$, and the boundary points of $S$ are mapped to the top or bottom. 
Oriented tangles in $M_p$ are considered up to ambient isotopy fixing the boundary. We denote the set of isotopy classes by $\calT$. An example is shown in Figure~\ref{fig:polestudio}.
\end{definition}

\begin{definition}\label{def:framed_tangles}
A \emph{framing} for an oriented tangle $T$ in $M_p$ is a continuous choice of unit normal vector at each point of $T$, which is fixed pointing North at the boundary points.
{\em Framed oriented tangles} in $M_p$ are also considered up to ambient isotopy fixing the boundary. We denote the set of isotopy classes of framed oriented tangles by $\glosm{tT}{\tT}$.
\end{definition}

Henceforth, any tangle is assumed to be framed and oriented unless otherwise stated. The skeleton of a tangle is the underlying combinatorial information with the topology forgotten:

\begin{definition}\label{def:skeleton}
 The {\em skeleton} $\sigma(T)$ of a tangle $T= (S \hookrightarrow M_p)$ -- see Figure~\ref{fig:skeleton} -- is the set of tangle endpoints $P_{bot}\subseteq D_p\times\{0\}$ and $P_{top} \subseteq D_p\times\{1\}$, along with 
 \begin{enumerate}
     \item A partition of $P_{bot}\cup P_{top}$ into ordered pairs given by the oriented intervals of $S$.
     \item A non-negative integer $k$: the number of circles in $S$. 
 \end{enumerate}
\end{definition}

\begin{figure}
    \centering
    \begin{picture}(300,90)
\put(0,10){\includegraphics[scale=.7]{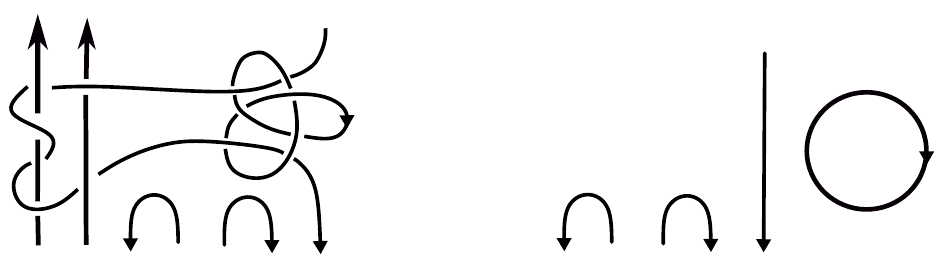}}
\put(38,5){\Small{$(1,0)$}}
\put(99,5){\Small{$(5,0)$}}
\put(180,5){\Small{$(1,0)$}}
\put(248,5){\Small{$(5,0)$}}
\end{picture}
    \caption{The left tangle is in $M_2$, and on the right is a schematic diagram of the skeleton of the tangle. The skeleton of the tangle is the combinatorial data given by 
    the following set of order pairs and the integer 1: $\{[((2,0),0),((1,0),0))],[((3,0),0),(4,0),0))],[((5,0),1),(5,0),0))], 1\}$}.
    \label{fig:skeleton}
\end{figure}
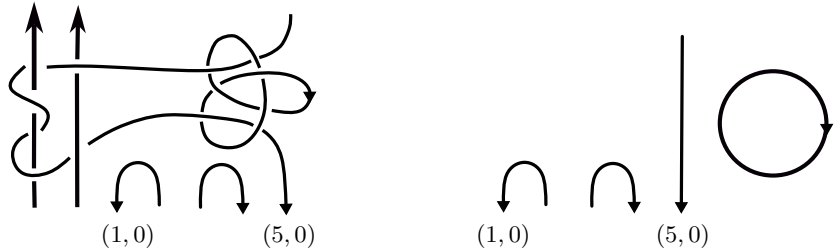

The skeleton of a framed tangle is the same as the skeleton of the underlying unframed tangle. The set of framed tangles in $M_p$ with skeleton $S$ is denoted $\tT(S)$. For example, $\tT(\circleSkel)$ is the set of framed knots in $M_p$.

The linear extension of $\tT(S)$, denoted $\ctT(S)$, is the vector space of $\mathbb{C}$-linear combinations of tangles in $\tT(S)$. We denote by $\ctT$ the disjoint union $\sqcup_S \tT(S)$ over all skeleta $S$. Tangles with different skeleta cannot be linearly combined.

One may represent tangles in $M_p$ using tangle diagrams in (at least) two different ways: by projecting to the back wall of $M_p$ or to the floor.

Projecting to the back wall, an $\ell$-component tangle in $M_p$ can be diagrammatically represented as a tangle diagram with $p$ straight vertical {\em poles}, and $\ell$ {\em tangle strands} of circle and interval components. The strands pass over (in front of) and under (behind) the poles and other strands, as shown on the right in Figure~\ref{fig:polestudio}. The poles are oriented upwards. By Reidemeister's theorem, $\tT$ is in bijection with such diagrams modulo the Reidemeister moves R2 and R3, and the framed version of R1.

By projecting instead to the floor $D_p \times \{0\}$, a tangle in $M_p$ is represented by a tangle diagram in $D_p$. The R2 and R3 moves continue to apply. The endpoints of the tangle are fixed: bottom endpoints are denoted by dots, top endpoints are denoted by stars. Strands of the tangle diagram can pass over bottom endpoints, or under top endpoints, as shown in Figure~\ref{fig:BottomDiagram}. However, the strands cannot pass across the punctures in $D_p$. 

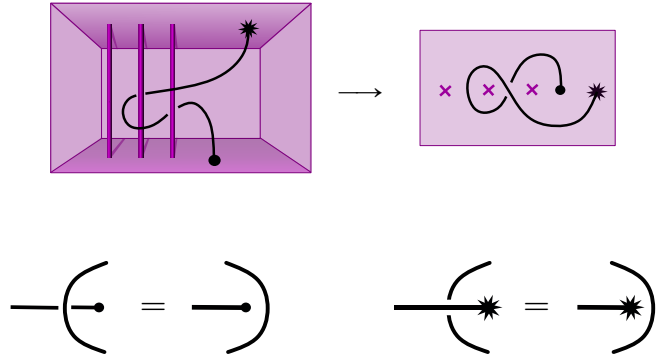
\begin{figure}
    \centering
     \begin{picture}(250,140)
    \put(0, 0){\includegraphics[scale=.60]{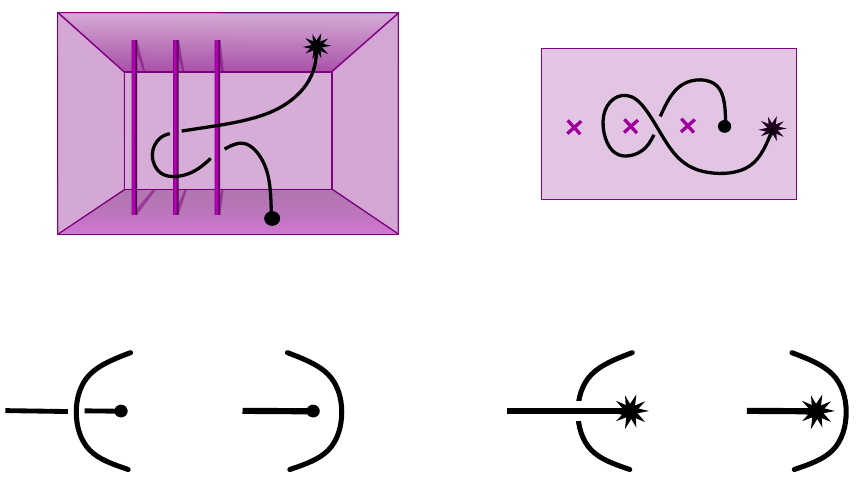}}
    \put(125,100){$\longrightarrow$} 
    \put(50,18){\Large $=$}
    \put(195,18){\Large $=$}
    \end{picture}
    \caption{An example of a tangle in $M_3$ projected to the bottom floor of the cube. Strands of a tangle diagram can pass over bottom endpoints (dot) or under top endpoints (star).}\label{fig:BottomDiagram}
\end{figure}

\subsection{Operations on $\tT$}\label{sec:opsonT} There are several useful operations defined on $\tT$. These operations extend linearly to $\ctT$, and are used in Section~\ref{sec:IdentifyingGTinCON} to relate quotients of $\ctT$ to the Goldman-Turaev Lie bialgebra.
\begin{itemize}
\item \emph{Stacking product:} Given tangles $T_1, T_2 \in M_p$, if the top endpoints of $\sigma(T_1)$ coincide with the bottom endpoints of $\sigma(T_2)$ in $D_p$, and the orientations on the strands of $T_1$ and $T_2$ agree, then the product $T_1T_2$ is the tangle obtained by stacking $T_2$ on top of $T_1$.

\item \emph{Strand addition:} The \textit{strand addition} operation adds a non-interacting additional strand to a tangle $T$ at a point $q \in D_p$ to get a new tangle $\Tql$. More precisely, pick a contractible $U\subseteq D_p$ such that $T$ is contained entirely in $U\times [0, 1]$ and a point $q \in D_p$ outside of $U$. The tangle $\Tql$ is $T$ together with an upward-oriented vertical strand $q\times I$ at $q$.

\item \emph{Strand orientation switch:} This operation reverses the orientation of a given strand of the tangle.

\item \emph{Flip:} Given a tangle $T$ in $M_p$, the flip of a tangle $T$ in $M_p$, denoted $\glosm{fli}{T^\sharp}$, is the mirror image of $T$ with respect to the ceiling, as shown in Figure \ref{fig:flip}. When $T$ is flipped, each top boundary point $(q, 1)$ becomes a bottom boundary point $(q, 0)$, and vice versa. The orientations and framing of the strands of $T$ are reflected along with the strands. However, the orientations of the poles remain ascending. Equivalently, the flip operation can be defined as reversing the parametrisation of $I$ in $M_p \cong D_p \times I$. This, in effect, flips the orientation of the poles but changes nothing else. 

\end{itemize}

\begin{figure}
    \centering
     \begin{tikzpicture}
    \draw (0, 0) node {\includegraphics[scale=.3]{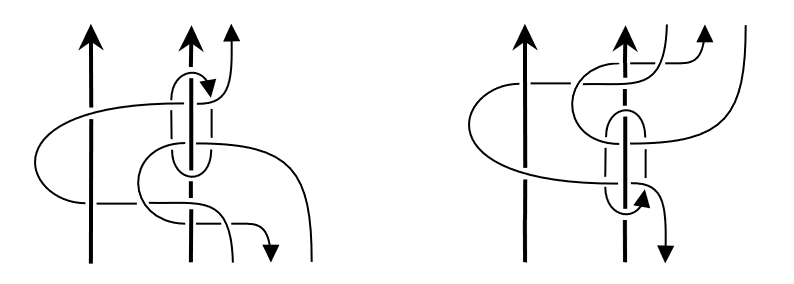}};
    \draw (0,0) node {\Large $\xrightarrow{\sharp}$};
    \end{tikzpicture}
    \caption{A tangle in $M_2$ and its flip}
    \label{fig:flip}
\end{figure}

In Section~\ref{sec:identifybracketinCON}, we show that the stacking commutator of tangles, given by $[T_1, T_2]=T_1T_2-T_2T_1$, induces to the Goldman bracket in the sense of Section~\ref{sec:conceptsum}. In Section~\ref{sec:cobracketinCON} a similar but more subtle argument relates the flip operation to the Turaev cobracket.

\subsection{The $t$-filtration on $\tT$ and the associated graded $\tA$}\label{sec:t-filtration}
In setting up a theory of Vassiliev invariants for $\tT$, there are different filtrations to consider. 
In line with classical notation of Vassiliev invariants, we denote by a double point the difference between an over-crossing and an under-crossing: $$\dpcross=\pcross - \ncross.$$ Double points, however, come in two varieties: \emph{pole-strand}, if the crossing occurs between a pole and a tangle strand, and \emph{strand-strand}, if the crossing occurs between two tangle strands. As the poles are fixed, they never cross each other, hence, there are no pole-pole double points.

The main filtration we consider on $\ctT$ is the filtration by the total number of double points of either type, as well as strand framing changes (as in Section~\ref{subsec:FramedKon}). We call this the \emph{\underline{t}otal filtration}, or $t$-filtration for short, and write it as
$$\ctT=\tT_0\supseteq \tT_1\supseteq \tT_2 \supseteq \tT_3\supseteq \cdots $$ where $\glosm{tTt}{\tT_t}$ is the set of linear combinations of framed tangle diagrams with at least $t$ total double points and strand framing changes. In spirit, this filtration comes from the diagrammatic view of projecting to the back wall of the cube.

The associated graded space of $\ctT$ with respect to the total filtration is
$$\glosm{tA}{\tA}:= \gr\ctT=\prod_{t\geq 0} \tT_t/\tT_{t+1}.$$ The degree $t$ component of $\tA$ is $\glosm{tAt}{\tA_t} :=\tT_t/\tT_{t+1}$.

As in classical Vassiliev theory (cf. section~\ref{subsec:FramedKon}), the associated graded space $\tA$ has a combinatorial description in terms of {\em chord diagrams}. 
\begin{definition}
A {\em chord diagram} on a tangle skeleton is an even number of marked points on the poles and skeleton strands, up to orientation preserving diffeomorphism, along with a perfect matching on the marked points -- that is, a partition of marked points into unordered pairs. In diagrams, the pairs are connected by a {\em chord}, indicated by a dashed line, as in Figure ~\ref{fig:AdmissibleNonAdmissible}.
\end{definition}

\begin{definition}\label{def:admissible}
A chord diagram is {\em admissible} if all chords connect strands to strands, or strands to poles, but there are no pole-pole chords. See Figure~\ref{fig:AdmissibleNonAdmissible} for examples.
\end{definition}

\begin{figure}
\centering
\includegraphics[scale=.6]{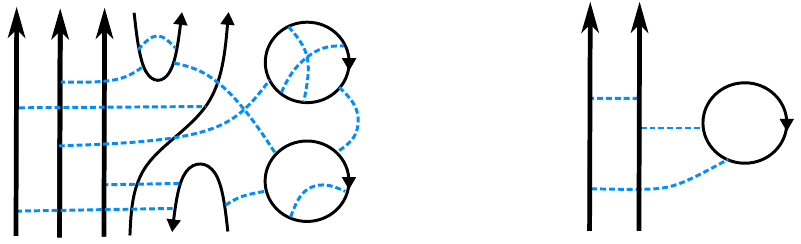}
\caption{Two chord diagrams: an admissible one (left), and a non-admissible one (right) that does contain a pole-pole chord.}
\label{fig:AdmissibleNonAdmissible}
\end{figure}

\begin{figure}
\centering
\begin{picture}(200,50)
\put(0,0){\includegraphics[scale=.4]{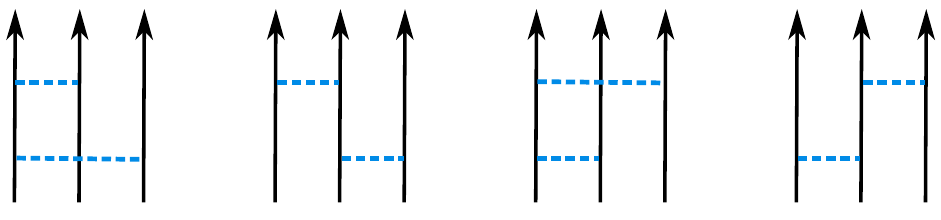}}
\put(190,15){\Large{$=0$}}
\put(35,15){\Large{+}}
\put(85,15){\Large{$-$}}
\put(135,15){\Large{$-$}}
\end{picture}
\caption{The 4T relation, which is admissible if at most one of the three skeleton components is a pole.}
\label{fig:Admissible 4T}
\end{figure}

\begin{definition}\label{def:cdspace}
The space $\glosm{Ds}{\D(S)}$ of \emph{admissible chord diagrams on a skeleton $S$} is the space of $\C$-linear combinations of admissible chord diagrams on the skeleton $S$, modulo {\em admissible} $4T$ relations, shown in Figure~\ref{fig:Admissible 4T}. Admissible $4T$ relations are $4T$ relations where all four terms are admissible\footnote{Equivalently, a 4T relation is admissible if at most one of the three skeleton components involved is a pole.}. 
That is, $$\D(S)=\frac{\C\langle\text{admissible chord diagrams on $S$}\big\rangle}{\big\{\text{admissible $4T$ relations}\}}$$
The space $\D(S)$ is a graded vector space, where the degree is given by the number of chords. Denote the degree $t$ component of $\D(S)$ by $\glosm{Dst}{D_t(S)}$. Let $\glosm{D}{\D}$ denote the disjoint union $\sqcup_S\D(S)$, and denote the degree $t$ component of $\D$ by $\glosm{Dt}{\D_t}=\sqcup_S\D_t(S)$.
\end{definition}

The well-known map $\glosm{psi}{\psi}:\D\rightarrow \tA$ from classical Vassiliev theory is defined as follows.
In degree $t$, $\psi_t: \D_t \to \tT_t/\tT_{t+1}$, ``contracts'' the $t$ chords to double points, as shown in Figure~\ref{fig:psi}. This may create other crossings, but modulo $\tT_{t+1}$ the over/under information at these crossings does not matter.

\begin{lem}\label{lem:psi}
The map $\psi$ is well-defined and surjective.
\end{lem}

\begin{proof} 
To show $\psi$ is well-defined, it suffices to show that admissible $4T$ relations in $\calD_t$ are in the kernel of $\psi$. This is the standard ``lasso trick'' recalled in Figure~\ref{fig:psicomputation}.
For surjectivity, recall from Section~\ref{subsubsec:Framing} that a framing change in $\tA$ is half of chord.
So, both framing changes and double points are in the image of $\psi$, and thus $\psi$ is surjective.
\end{proof}

\begin{figure}
\centering
\begin{picture}(200,100)
\put(-25,0){\includegraphics[scale=.7]{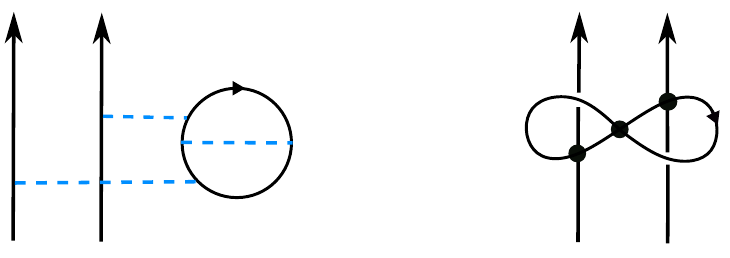}}
\put(100,40){\Large{$\mapsto$}}
\put(103,50){{$\psi$}}
\end{picture}
\caption{Example of $\psi$ with the right hand side viewed as an element of $\tT_3/\tT_4$. Different choices of over or under crossings with the poles only differ by elements of $\tT_4$.}
\label{fig:psi}
\end{figure}

\begin{figure}
\centering
\begin{picture}(290,100)
\put(-40,0){\includegraphics[scale=.8]{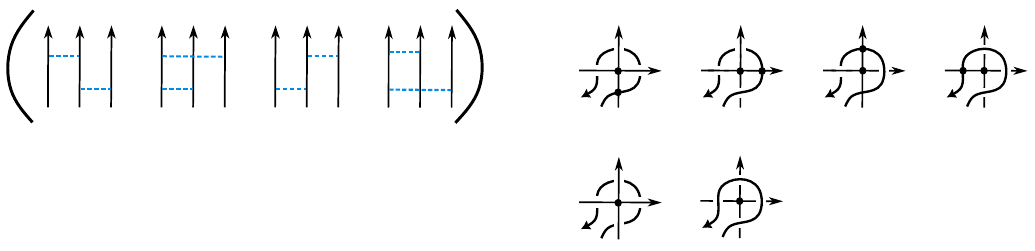}}
\put(-57,65){\Huge{$\psi$}}
\put(-35,65){\Large{$-$}}
\put(8,65){\Large{$+$}}
\put(50,65){\Large{$+$}}
\put(95,65){\Large{$-$}}
\put(153,65){\Large{$=$}}
\put(170,65){\Large{$-$}}
\put(215,65){\Large{$+$}}
\put(265,65){\Large{$+$}}
\put(310,65){\Large{$-$}}
\put(153,15){\Large{$=$}}
\put(215,15){\Large{$-$}}
\put(265,15){\Large{$=$}\Large{ 0}}
\end{picture}
\caption{Showing that $\psi: \D \to \tA$ is well defined. The figure is understood locally: in degree $t$ the chord diagrams have $t-2$ other chords elsewhere, and correspondingly the tangles have $t-2$ other double points elsewhere. }
\label{fig:psicomputation}
\end{figure}




According to Lemma ~\ref{lem:assocgradyoga}, in order to show that $\psi$ is an isomorphism, one needs to find an expansion valued in $\D$.


\begin{lem}\label{thm:Zwelldefined}
The framed Kontsevich integral $Z: \ctT \to \D$ satisfies the conditions of Lemma~\ref{lem:assocgradyoga}: it is filtered, and $\gr Z\circ \psi = \id_{\D}$.
\end{lem}

\begin{proof} 
This is a variant of a standard fact \cite{Kon}; one detailed explanation is in \cite[Section 4.3]{BN1}. A small point to verify is that the image of $Z$ on an element of $\ctT$ is an admissible chord diagram. This is immediate from the definition of the Kontsevich integral: the poles are parallel, hence the coefficient of a chord diagram with a pole-pole chord is computed by integrating zero. The main part, that $\gr Z\circ \psi=id_{\D}$, is done as in \cite[Section 4.4.2, Thm 1 part (3)]{BN1}.
\end{proof}

The next corollary is then immediate from Lemma~\ref{lem:assocgradyoga}:

\begin{cor}\label{cor:grcd}
The map $\psi: \D \to \tA$ is an isomorphism, and $Z$ is an expansion for $\tT$.
\end{cor} 
After identifying $\tA$ with $\D$, the degree $t$ component of $\tA$, $\tA_t =\tT_t/\tT_{t+1}$,  consists of all admissible chord diagrams in $\tA$ with exactly $t$ chords.

For a skeleton $S$, we denote by $\glosm{tAS}{\tA(S)}$ the space of admissible chord diagrams on the skeleton $S$, so $\tA(S)$ is the associated graded vector space of $\ctT(S)$. For example, $\tA(\circleSkel)$ is the associated graded vector space of the space of knots in $M_p$.

\subsection{Operations on $\tA$}

The tangle operations \textit{stacking}, {\em strand addition}, {\em strand orientation switch}, and \textit{flip} on $\tT$ induce associated graded operations denoted by the same names on $\tA$. In view of Corollary \ref{cor:grcd}, we give descriptions of these operations using chord diagrams. 

The operation \textit{stacking} is given by concatenating the skeleta of two chord diagrams (as long as they have the same number of poles, and the top endpoints of one match the bottom endpoints of the other, including orientations). 

The associated graded {\em strand addition} operation adds a vertical skeleton strand to a chord diagram. The new strand has no chord endings.

The associated graded {\em strand orientation switch} for strand $e$ switches the orientation of the strand $e$, and multiplies each chord diagram with $(-1)$ to the power of the number of chord endings on $e$. The sign arises from the fact that reversing the orientation of $e$ changes the signs of double points between $e$ and any other distinct strand or pole.

\begin{figure}
    \begin{picture}(330,115)
    \put(0,0){\includegraphics[width=0.75\textwidth]{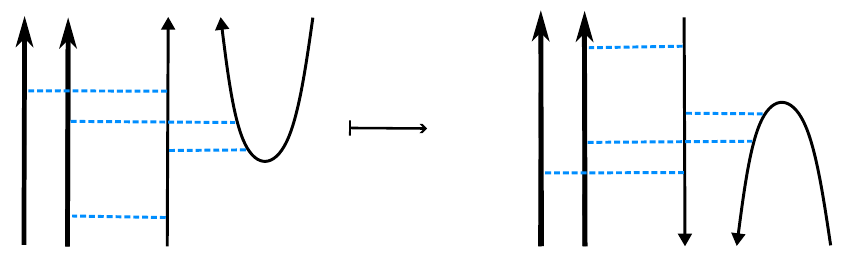}}
    \put(160,65){$\sharp$}
    \put(185,53){$(-1)^3$}
    \end{picture}
    \caption{An example chord diagram and its flip.}
    \label{fig:chorddiagoperations}
\end{figure}

The associated graded operation \textit{flip}, denoted by $\sharp$, reflects a chord diagram with respect to a ``mirror on the ceiling''; then reverses the orientations of the poles so that they are oriented upwards, as in Figure~\ref{fig:chorddiagoperations}; and multiplies by a factor of $(-1)^m$, where $m$ is the total number of chord endings on the poles. The factor of $(-1)^m$ arises from the pole orientation reversals, as this changes the signs of any pole-strand double points.

The following proposition is straightforward from the definition of $Z$.

\begin{prop}\label{prop:Zhomom}
The Kontsevich integral $Z$ intertwines stacking, strand additions, orientation reversals and flips with their associated graded operations. \qed
\end{prop}

\subsection{The $s$-filtration on $\tT$ and $\tA$}\label{sec:s-sfiltration}

Recall from Section \ref{sec:t-filtration} that the total filtration on $\ctT$ is given by strand framing changes and double points between strands with poles and strands with strands. In this section we introduce a second filtration on $\ctT$, given by strand framing changes, and {\em only strand-strand} double points. We call this the \emph{\underline{s}trand filtration}, or simply $s$-filtration. 

We use superscripts for the $s$-filtration:
$$\ctT=\tT^0\supseteq \tT^1\supseteq \tT^2 \supseteq \tT^3\supseteq \cdots, $$ where $\glosm{tTs}{\tT^s}\subseteq \ctT$ is spanned by tangles with at least $s$ strand framing changes or strand double points.

\begin{remark}
The associated graded structure of $\ctT$ with respect to the $s$-filtration was studied by Habiro and Massuyeau in \cite{HM}, as part of their work on {\em bottom tangles}\footnote{Projecting to the bottom of the cube rather than the back wall makes the strand filtration the natural Vassiliev filtration to consider.}. Yet we do not apply the associated graded functor to the $s$-filtration, but rather, quotient only by $\tT^1$ and $\tT^2$ to identify the Goldman-Turaev spaces and operations in Section ~\ref{sec:IdentifyingGTinCON}. 
\end{remark}

\begin{remark}
    We recover homotopy classes of curves in the quotient  $\tT/\tT^1$,  as $\pcross-\ncross$=0, or $\dpcross=0$.   Emergent knots and tangles arise in the next quotient of $\tT$ by $\tT^2$.
Emergent knots are situated in between homotopy classes of curves (where there is no notion of over or under strands) and classical knots with the usual Reidemeister theory.
Emergent knots satisfy a familiar relation in finite type invariant theory  $\dpcross \dpcross =0$.
More concretely,  
modding out $\tT$ by $\dpcross \dpcross =0$ declares two tangles the same if they differ by two crossing changes. 
This quotient removes most, but not all, knotted information of the tangles and the slightest of knot theory \emph{emerges}.
(See \cite{Kuno25} Appendix: Emergent knotted objects).
\end{remark}

In turn, the $s$-filtration induces a filtration on $\tA$, as follows. 
Let $\glosm{tTst}{\tT_t^s}$ denote $\tT_t \cap \tT^s$: that is, the linear span of tangles in $\ctT$, which that have at least $t$ double points or framing changes, at least $s$ of which are strand-strand double points or framing changes.

\begin{definition}\label{def:filtrationQuotientNotation}
    Denote by $\glosm{tAtgeqs}{\tA^{\geq s}}$ the $s$-filtered component of $\tA$: $$\tA^{\geq s }:=\prod \tT_t^s/\tT_{t+1}^s.$$ Explicitly, $\tA^{\geq s}$ is spanned by chord diagrams with at least $s$ strand-strand chords.
\end{definition}

For strand-strand chords we will use the shorthand word $s$-chords. Note that the number of $s$-chords is only a filtration, not itself grading on $\tA$, as the 4T relation is not homogeneous with respect to the number of $s$-chords.

\begin{prop}\label{prop:ZrespectsS}
    The Kontsevich integral $Z$ is a filtered map with respect to the $s$-filtration. 
\end{prop} 
\begin{proof}
    This is a close analogue of Theorem \ref{thm:Zwelldefined}. As strand-strand double points correspond to strand-strand chords via the identification $\psi$ of the associated graded space with chord diagrams, the proof translates verbatim from \cite[Section 4.3]{BN1}. 
\end{proof}

\subsection{The Conway quotient}\label{sec:Conway}
In this section we introduce the last necessary ingredient for identifying the Goldman-Turaev operations: the Conway quotient of $\ctT$. This is essentially the Conway skein module of tangles in $M_p$, but without fixing the value of the unknot. We show that the Kontsevich integral descends to the Conway quotient.

\begin{definition}\label{def:conway}
The Conway quotient of $\ctT$, denoted $\glosm{Ctnab}{\C\tTnab}$, is given by 
$$\C\tTnab := \bigslant{\ctT\llbracket a \rrbracket }{\pcross-\ncross = (e^{\frac{a}{2}}-e^{-\frac{a}{2}})\smooth},$$ 
where $\glosm{a}{a}$ is a formal variable of $t$ and $s$ degree 1, and the skein relation is restricted to  strand-strand crossings. We use the shorthand $\glosm{b}{b}:=e^{\frac{a}{2}}-e^{-\frac{a}{2}}$.
\end{definition}

The $t$ and $s$ filtrations on $\ctT$ induce filtrations on $\ctT_\nab$. 
Let $\glosm{tTnabs}\tT_{\nab}^s$ denote the $s$-filtered component in the $s$-filtration of $\ctT_\nab$. 
Let $\glosm{tTnabt}\tT_{\nab, t}$ denote the $t$-filtered component in the total filtration of $\ctT_\nab$,
and $\glosm{tAnab}{\tA_\nab}:= \gr_t\ctT_\nab=\prod \tT_{\nab,t}/\tT_{\nab,t+1}$ denote the associated graded algebra of $\ctT_\nab$ with respect the total filtration. 
Let $\glosm{tanabt}{\tA_{\nab,t}}$ denote the degree $t$ component of $\tA_{\nab}$, $\glosm{tAnabs}{\tA_\nab^s}$ be the $s$-th filtered component, and $\glosm{tAnabts}{\tA_{\nab,t}^s}=\tA_{\nab,t}\cap \tA_\nab^s$.
We now show that $\tA_\nab$ has a chord diagrammatic description similar to Corollary \ref{cor:grcd}. Recall that $\D$ is the space of chord diagrams on tangle skeleta, modulo admissible 4T relations.

\begin{definition}\label{def:D_con} The conway quotient of $\D$ is given by
 $$\glosm{Dnab}{\D_\nab} := \bigslant{\D\llbracket a \rrbracket}{\sschordop = a \smoothop,\,\,\, \sschordsame = a\smoothsame}$$ where the new relations are restricted to chords on strand skeleton components (not poles). 
 \end{definition}

Note that the two new relations in $\D_\nab$ are equivalent, shown in both combinations of orientations for convenience.
Furthermore, the relations are homogeneous (respect the $t$-grading) on $\D$, and therefore $\D_\nab$ is also graded by the sum of the total number of chords and the exponent of $a$. The next theorem shows that $\tA_\nab\cong \D_\nab$: this essentially follows from the results of \cite{LM95}. For completeness we present a direct proof.

 \begin{thm}\label{thm:Z_conway}
 The isomorphism $\psi$ descends to an isomorphism $\psi_{\nab}:\tA_\nab\cong \D_\nab$, and the Kontsevich integral descends to an expansion $\glosm{Znab}{Z_\nab}: \ctT_\nab \to \D_\nab$. 
\end{thm}

 \begin{proof}
 First we show that $\psi$ descends to a surjective graded map  $\psi_\nab:\D_\nab\rightarrow \tA_\nab$. 
 To show that $\psi_\nab$ is well-defined, we need to show that the Conway relations in $\D_\nab$ are in the kernel. We verify one of the two equivalent relations:
 $$\psi\left(\sschordsame - a\smoothsame\right)=\dpointwithxarrowsup -  a\pcross=a\pcross-a\pcross=0.$$


Next, we verify that the Kontsevich integral $Z$ descends to a map $\ctT_\nab\to \tA_\nab$ by verifying the relations in $\ctT_\nab$. We do this first at the level of tangles with two bottom and two top endpoints (directly above). Recall that
 the Kontsevich integral is invariant under both total horizontal and total vertical rescaling, and hence well-defined for such two-two tangles without specifying the distance between the endpoints. 

Recall that 
\[Z(\pcross) = \left(e^{\frac{C}{2}}\right)\cdot\swap, \quad \text{ and } \quad Z(\ncross) = \left(e^{-\frac{C}{2}}\right)\cdot\swap,\] where $C$ denotes a chord, the exponential is interpreted formally as a power series with the stacking multiplication, as shown in the first equality below. Using the Conway relation, we compute:
$$C^k = \powerchord\stackrel{\nab}{=}a^k \powerswap= a^k(\swap)^k= \begin{cases}
  a^k \smoothchord , & \text{ if } k \text{ is even} \\
  \\
  a^k \swap ,& \text{ if } k \text{ is odd}
\end{cases}$$

Now applying $Z$ to the left hand side of the Conway relation, we obtain 
\begin{align*}
    Z(\pcross)-Z(\ncross) &= (e^{\frac{C}{2}}-e^{\frac{-C}{2}})\swap\\
    &= \sum_{k = 0}^\infty \left(\frac{C^k}{2^kk!}-\frac{(-1)^kC^k}{2^kk!}\right)\swap =\sum_{k=0}^\infty \frac{C^{2k+1}}{2^{2k}(2k+1)!}\swap\\
    &\stackrel{\nab}{=} \sum_{k=0}^\infty \frac{a^{2k+1}\swap}{2^{2k}(2k+1)!}\swap = \sum_{k=0}^\infty \frac{a^{2k+1}}{2^{2k}(2k+1)!}\smoothchord\\
    &= (e^{\frac{a}{2}}-e^{-\frac{a}{2}})\smoothchord\\
    &= Z\left((e^{\frac{a}{2}}-e^{-\frac{a}{2}})\smooth\right).
\end{align*}

To see that the local verification above is sufficient, one needs to recall more about the Kontsevich integral. Namely, $Z$ is multiplicative with respect to the stacking composition of tangles (with fixed endpoints), and asympototically commutes with ``distant disjoint unions'', and these two facts imply the global equality (in fact, they lead to a combinatorial construction of $Z$ for {\em parenthesised} tangles). For details see \cite[Chapter 8]{CDM_2012}.

Therefore, by Lemma~\ref{lem:assocgradyoga}, $Z$ is a (homomorphic) expansion for $\ctT_\nab$ and $\psi: \calD_\nab \to \tA_\nab$ is an isomorphism.
\end{proof}

\medskip
Let $\glosm{i}{\iota}$ denote the composition of the natural embedding with the Conway quotient map 
$$\iota: \ctT \to \ctT\llbracket a \rrbracket \to \ctT_\nab.$$
The map $\iota$ is not injective, see for example Figure~\ref{fig:ConwaySkel}. However, it is surjective: all $\C$-linear combinations of tangles are in the image, and given a tangle $T$, $b^kT$ is equal in $\C\tTnab$ to a tangle with $k$ double points, which is, in turn, a $\C$-linear combination of tangles.

\begin{definition} \label{def:conway_skel}
    For skeleton $S$, let $\glosm{CtnabS}{\C\tTnab(S)}$ denote the image $\iota(\ctT(S))$.
\end{definition}

\begin{figure}
\begin{picture}(270,50)
    \put(0,20){\includegraphics[width=10cm]{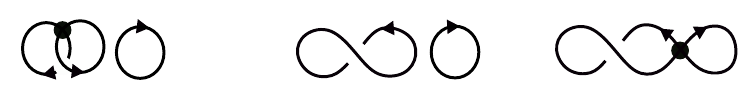}}
    \put(70,33){\Large $= b\cdot ($}
    \put(183,33){\Large $)=$}
    \put(30,15){$\upsidein$}
    \put(10,0){$\ctT_\nab(\circleSkel \, \circleSkel \, \circleSkel)$}
    \put(240,15){$\upsidein$}
    \put(220,0){$\in\ctT_\nab(\circleSkel)$}
\end{picture}

\caption{The map $\iota$ is not injective: The left hand side and the right hand side are both elements of $\ctT$, and equal in $\ctT_{\nab}$. }
\label{fig:ConwaySkel}
\end{figure}

Note that while the skeleton fibration of $\ctT$ is a partition into disjoint fibers $\ctT(S)$, this is no longer true in $\C\tTnab$ due to the non-injectivity of $\iota$.
For example, the middle term of the equality in Figure~\ref{fig:ConwaySkel} lies in both $\C\tTnab(\circleSkel)$ and $\C\tTnab(\circleSkel\,\circleSkel\,\circleSkel)$.

We will identify the Goldman-Turaev Lie bialgebra in low-degree quotients of the $s$-filtration of $\C\tTnab$. The next few propositions establish the necessary understanding of these quotients. Denote by $\glosm{tTn}{\tT^{/n}}$ the quotient $\C\tT/\tT^n$, and similarly for the Conway quotients, $\glosm{tTnabn}{\tTnab^{/n}}$ denotes $\C\tTnab/\tTnab^{n}$.

\begin{prop}\label{prop:nonabneeded}
    The map $\iota$ descends to a canonical isomorphism $\tTnab^{/1}\cong \tT^{/1}$.
\end{prop}
\begin{proof}
The Conway relation applies only in $s$-filtered degree one and higher, and hence has no effect on $\tT^{/1}$.
\end{proof}

In light of this, we write only $\tT^{/1}$, rather than $\tTnab^{/1}$. Now let $\glosm{tT1/2}{\tT^{1/2}}$ denote $\tT^1/\tT^2$, and $\glosm{tTnab1/2}{\tTnab^{1/2}}$ denote $\tTnab^1/\tTnab^2$. 

Finally, we establish a key technical result about low $s$-degree quotients of the Conway quotient:

\begin{prop}\label{prop:divbybexists} 
The $\C$-linear map given by post-composing $\iota$ with multiplication\footnote{In physics, multiplication by a variable $b$ is denoted $\widehat{b}$.} by $b$, $$\glosm{bhat}{\widehat{b}}:\tT^{/1} \to \tTnab^{/2}$$ is injective, and its image is $\tTnab^{1/2}$. 
\end{prop}

\begin{proof}
We first prove that the image of $\widehat{b}$ is $\tTnab^{1/2}$. The quotient $\tT^{/1}$ is spanned by cosets of tangles $T$. 
It is immediate that the image of $\widehat{b}$ is contained in $\tT^{1/2}$, as $\widehat{b}(T)=bT$ represents an element in $\tT^{1}$. 

Conversely, any element $y\in\tT^{1/2}$ is (non-uniquely) represented as a sum of the form $\sum_{i=1}^k T_i + b\sum_{j=1}^l T_j'$, where $T_i$ are tangles with one double point each, and $T_j'$ are arbitrary tangles. Then, by the Conway relation, each $T_i=b\cdot T_i^C$, where $T_i^C$ denotes the tangle where the double point in $T_i$ has been smoothed. Thus, $y=b\left(\sum_{i=1}^k T^C_i + \sum_{j=1}^l T_j'\right)$, and therefore $y$ is in the image of $\widehat{b}$, and $\widehat{b}$ is surjective onto $\tTnab^{1/2}$.

To prove the injectivity of $\widehat{b}$, we construct a one-sided inverse: a ``division by $b$'' map $\glosm{bcheck}{\widecheck{b}}$ on $\tTnab^{/2}$, as follows.

For a tangle diagram $D_T$ (representing a tangle $T$) and a crossing $x$ of $D_T$, let $\epsilon(x) \in \{\pm 1\}$ be the sign of $x$, and $D_T\vert_{x \to \upupsmoothing}$ be the diagram $D_T$ with $x$ replaced by a smoothing. We first define a map $\widecheck{b}$ from the free $\C[b]$-module spanned by tangle diagrams, to $\tT^{/1}$, as the linear extension of the following: \begin{align*}b^k D_T &\stackrel{\widecheck{b}}{\mapsto} 0 \text{ if } k\geq 2, \\
b D_T &\stackrel{\widecheck{b}}{\mapsto} D_T, \\
D_T &\stackrel{\widecheck{b}}{\mapsto} \frac{1}{2}\sum_{x \text{ crossing of $T$}} \epsilon(x) D_T \vert_{x \to \upupsmoothing}.
\end{align*} 

 We claim that this descends to a well defined map $\widecheck{b}: \tTnab^{/2} \to \tT^{/1}$. It is straightforward to check that the Reidemeister moves are in the kernel of $\widecheck{b}$. We also need to verify that $\tT^2_\nab$ and the Conway relation are in the kernel.
 
 An element of $\tT^2_\nab$ can be represented as a sum of terms $b^kD_T\in \tT^2_{\nab}$, where $D_T$ is a tangle diagram with or without double points. If $k \geq 2$ then $\widecheck{b}(b^kD_T)=0$. If $k = 1$, then $D_T$ has a double point, so $\widecheck{b}(bD_T) =D_T$ is zero in $\tT^{/1}$. If $k=0$, then $D_T$ has at least two double points. Smoothing a crossing interferes with at most one of the double points, so
$\widecheck{b}(D_T)$ can be written as a sum of terms with at least one double point each. Hence $\widecheck{b}(D_T)\in \tT^1$ as well.

To show that the Conway relation vanishes, note that $\widecheck{b}(\doublepoint)=\widecheck{b}(\overcrossing - \undercrossing)$ is a sum with two types of terms: those which smooth a crossing that is a part of the double point, and those which smooth a crossing that is not.
In the latter case, the double points are unchanged, so these terms are in $\tT^1_\nab$. From the terms where the crossings forming the double point are smoothed, we get 
\begin{align*}\widecheck{b}\left(\pcross - \ncross\right)= \frac{1}{2}\smooth - (-1)\frac{1}{2}\smooth =\smooth = \widecheck{b}\left(b\smooth\right),\end{align*} as the Conway relation requires. Thus, $\widecheck{b}$ is well-defined on $\tT_\nab^{/2}$.

Finally, $\widecheck{b}$ is clearly a one-sided inverse for $\widehat{b}$, and therefore, $\widehat{b}$ is injective.
\end{proof} 

\begin{cor}\label{cor:divbyb}
The map $\widehat{b}:\tT^{/1} \to \tTnab^{1/2}$ is a $\C$-linear isomorphism with inverse $\widecheck{b}: \tTnab^{1/2} \to \tT^{/1}$.
\end{cor}

Notice that both $\widehat{b}$ and $\widecheck{b}$ shift the filtered degrees. The Goldman-bracket and Turaev cobracket are also degree-shifting, and these shifts will be realised by $\widehat{b}$ and $\widecheck{b}$. The following fact in particular will be important in the construction of the Goldman bracket:

\begin{lem}\label{lem:mbOnCircle}
The map $\widehat{b}$ restricts to an injective $\C$-linear map $$\widehat{b}: \tT^{/1}(\circleSkel) \to \tTnab^{1/2}(\circleSkel\,\circleSkel).$$
\end{lem}
\begin{proof}
Elements of $\tT^{/1}$ are linearly generated by the cosets of knots. Given a knot $K$, $bK$ is equal in $\tT^{1/2}$ to a difference of two two-component links, by a single use of the Conway relation. Hence, the codomain is $\tT^{1/2}(\circleSkel\,\circleSkel)$. Injectivity is inherited from $\widehat{b}$ on $\tT^{/1}$.
\end{proof}

Note that this restriction of $\widehat{b}$ is not surjective to $\tT^{1/2}(\circleSkel\,\circleSkel)$, for example, two-component links with a double point involving only one component are not in the image.

We introduce the same notation on the associated graded side:
\begin{definition}\label{def:Anot} 
    $\tA_\nab$ is the quotient of $\tA\llbracket a \rrbracket$ by the chord diagram Conway relation. 
    For a skeleton $S$, let $\glosm{tAnabS}{\tA_{\nab}(S)}$ denote the image $\gr \iota (\tA(S))$. Also, let 
     $\glosm{Aslashs}{\tA^{/s}}:=\tA/\tA^{\geq s}$ and $\glosm{Anabslashs}{\tA_{\nab}^{/s}}:=\tA_{\nab}/\tA_{\nab}^{\geq s}$.

\end{definition}

By a straightforward calculation of the degree shifting associated graded maps we obtain:

\begin{prop}\label{rem:grdivbyb}
    The associated graded map of $\widehat{b}$ is an isomorphism $\gr \widehat{b}=\widehat{a}:\tA^{/1}\to \tA^{1/2}_\nab$, which multiplies chord diagrams by $a$. The inverse is the ``division by $a$'' isomorphism $\gr \widecheck{b}=\widecheck{a}:\tA^{1/2}_\nab\rightarrow \tA^{/1}$.  Furthermore, $\widehat{a}$ restricts to an injective map $\widehat{a}: \tA^{/1}(\circleSkel)\to \tA^{1/2}_\nab(\circleSkel\, \circleSkel)$. \qed
\end{prop}

\section{Identifying the Goldman-Turaev Lie bialgebra}\label{sec:IdentifyingGTinCON}

 In this section we establish our main results: we identify the Goldman-Turaev Lie bialgebra in the low $s$-filtered degree quotients of $\ctT$, and show that the Kontsevich integral induces a homomorphic expansion.
The arguments follow the outline summarized in Section \ref{sec:conceptsum}, and obtain the Goldman bracket and the self-intersection map $\mu$ as induced operations.  
In turn, the homomorphicity of the Kontsevich integral follows from the naturality of the construction.

\subsection{The Goldman bracket}\label{sec:identifybracketinCON} Recall from Section~\ref{subsec:IntroGT} that $D_p$ denotes the $p$-punctured disc,  $\pi$ is its  fundamental group, and $|\Cp|$ is the linear quotient $|\Cp|:=\Cp/[\Cp,\Cp]$, which is linearly generated by homotopy classes of free loops in $D_p$. 
The Goldman bracket (Definition \ref{def:bracket}) is a lie bracket
 $[\cdot,\cdot]_G:|\C\pi|\otimes |\C\pi|\rightarrow |\C\pi|$.
We start by identifying $\Cpa$ in a low degree quotient of $\ctT(\circleSkel)$ through a map $\beta$ induced by the bottom projection. 

\begin{figure}
\centering
\begin{picture}(100,100)
\put(0,0){\includegraphics[width=4.5cm]{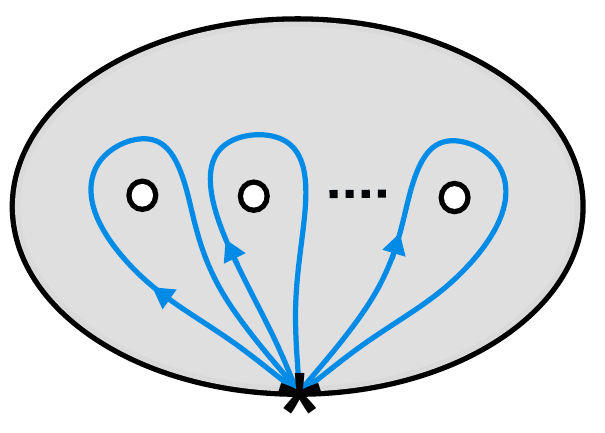}}
\put(20,68){\Large $\gamma_1$}
\put(47,68){\Large $\gamma_2$}
\put(95,68){\Large $\gamma_p$}
\end{picture}
    \caption{The standard generating curves of $\pi$.}
    \label{fig:GenCurves}
\end{figure}
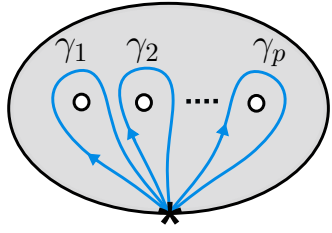

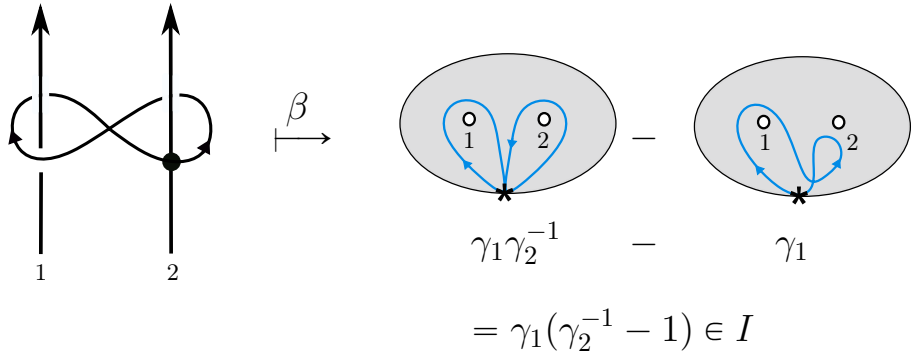
\begin{figure}
\centering
\begin{picture}(350,150)
\put(0,50){\includegraphics[width=12cm]{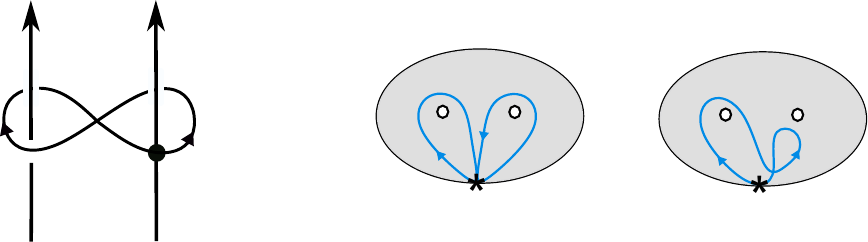}}
\put(10,40){\Small 1}
\put(60,40){\Small 2}
\put(172,90){\Small 1}
\put(200,90){\Small 2}
\put(284,89.5){\Small 1}
\put(317,89.5){\Small 2}
\put(105,101){\Large $\beta$}
\put(99.9,91){\Small |} 
\put(100,90){\Large $\longrightarrow$}
\put(235,90){\Large $-$}
\put(235,50){\Large $-$}
\put(175,50){\Large $\gamma_1\gamma_2^{-1}$}
\put(290,50){\Large $\gamma_1$}
\put(175,17){\Large $=\gamma_1(\gamma_2^{-1}-1)\in I$}
\end{picture}
    \caption{Example calculation demonstrating that $\beta$ is a filtered map.}
    \label{fig:BetaFiltered}
\end{figure}

\begin{prop}\label{prop:BotProj}
Consider the $t$-filtration of $\ctT(\circleSkel)$. With respect to this filtration, the bottom projection $M_p \to D_p\times\{0\}$ induces a surjective filtered map 
$$\glosm{beta}{\beta}: \ctT(\circleSkel) \to \Cpa.$$
\end{prop}

\begin{proof}
By the framed Reidemeister Theorem, framed knots in $\ctT(\circleSkel)$ are faithfully represented by knot diagrams in $D_p\times \{0\}$ -- regular projections to the bottom with over/under information -- modulo the framed Reidemeister moves (weak R1, R2, and R3). Diagrammatically, the bottom projection forgets the over/under information, in other words, imposes the relation $\overcrossing = \undercrossing$. The images of the Reidemeister moves follow from the corresponding moves generating homotopies of immersed free loops, hence $\beta$ is well-defined. The projection is clearly surjective as any loop can be lifted to a framed knot by introducing arbitrary under/over information at the crossings and imposing the blackboard framing.

The statement that $\beta$ is filtered means that step $k$ of the the Vassiliev $t$-filtration in $\ctT(\circleSkel)$  projects to step $k$ of the filtration on $|\Cp|$ induced by the I-adic filtration of $\pi$.  Note that strand-strand double points and framing changes are in the kernel of $\beta$, thus, we only have something to prove for knots with $k$ strand-pole double points.

Let $\gamma_1, ..., \gamma_p$ denote the standard generators of $\pi$ as in Figure~\ref{fig:GenCurves}. A knot $K\in \ctT(\circleSkel)$ maps to a free loop in $\Cpa$, whose conjugacy class in $\pi$ is represented as a word in the generators $\gamma_i$. 
A pole-strand double point on pole $j$ maps to a difference between two curves passing on either side of the $j$'th puncture (as in Figure~\ref{fig:BetaFiltered}). Therefore, one of the words in $\C\pi$ representing these curves can be obtained from the other by inserting a single letter $\gamma_j^{\pm 1}$. The double point, which represents the difference, thus maps to a product with a factor of $(\gamma_j^{\pm 1} -1)$, and a knot with $k$ pole-strand double points maps to a product with $k$ factors of the form $(\gamma_j^{\pm 1}-1)$. This is by definition an element in $\calI^k$.  
\end{proof}


\begin{prop}\label{prop:kerbeta} The kernel of $\beta$ is $\tT^1(\circleSkel)$, and thus, $\beta$ descends to a filtered linear isomorphism $\beta: \tT^{/1}(\circleSkel) \to \Cpa$.
\end{prop} 
\begin{proof}
Two framed knots in $\ctT(\circleSkel)$ project to the same loop in $\Cpa$ if and only if they differ by framing changes and (strand-strand) crossing changes, which generate exactly the step 1 of the $s$-filtration, that is, $\tT^1(\circleSkel)$. 
\end{proof}

Recall that $\tA$ denotes the (degree completed) associated graded space of $\ctT$ with respect to the $t$-filtration.
described as the space of admissible chord diagrams on a circle skeleton, as in Definition~\ref{def:cdspace}. The $s$-filtration on $\C\tT$ induces a corresponding $s$-filtration on $\tA$, and
$\tA^{\geq i}(\circleSkel)$ denotes the $i$-th $s$-filtered component of $\tA(\circleSkel)$. We also denote $\tA^{/i}(\circleSkel)=\tA(\circleSkel)/\tA^{\geq i}(\circleSkel)$.

Recall from Section \ref{subsec:IntroGT} that the associated graded  vector space of $\Cpa$ is $|\As|$, where $\As = \As\langle x_1, \cdots, x_p\rangle$ denotes the free associative algebra over $\C$, and the linear quotient $|\As| = \As/[\As, \As]$
is the graded $\C$-vector space generated by cyclic words in the letters $x_1,...,x_p$. The graded Goldman bracket is a map $[-,-]_{\operatorname{gr}G}:|\As|\otimes |\As|\to |\As|$, as defined in Proposition \ref{prop:gr_bracket_def}. Denote the degree completions of $\As$ and $|\As|$ by $\glosm{hatAS}{\hAs}$ and $\glosm{hatasbAs}{|\hAs|}$.
By applying the associated graded functor to $\beta$, we identify $|\hAs|$ as follows:

\begin{prop}\label{prop:gr_beta}The associated graded map $\gr\beta: \tA(\circleSkel) \to |\hAs|$ has kernel $\tA^{\geq 1}(\circleSkel)$. Hence, $\gr\beta$ descends to an isomorphism $\gr\beta: \tA^{/1}(\circleSkel) \to |\hAs|$.
\end{prop}

\begin{proof} The statement follows from applying the associated graded functor to the filtered short exact sequence
\begin{center}
\begin{tikzcd}
    0 \arrow{r}& \tT^1(\circleSkel) \arrow{r} & \tT(\circleSkel) \arrow[r,"\beta"] & \Cpa  \arrow{r} & 0.
\end{tikzcd}
\end{center}
The filtrations on $\tT^{1}(\circleSkel)$ and $\Cpa$ are induced from the filtration on $\tT(\circleSkel)$, as in Lemma \ref{lem:ExaxtGr}, therefore the associated graded sequence is also exact.
\end{proof}

\begin{remark}\label{rem:ChorsOnPoles}
In $ \tA^{/1}(\circleSkel)$ chord diagrams with any strand-strand chords are zero. Thus, $\tA^{/1}(\circleSkel)$ is spanned by chord diagrams on poles and a single circle strand, with strand-pole chords only: for an example see the left of Figure~\ref{fig:CycWord}. Note also that all admissible 4T relations involve a strand-strand chord, and are zero in $\tA^{/1}$. This means that chord endings on the poles commute, and there are no further relations. Such a chord diagram corresponds naturally to a cyclic word by labeling the poles with $x_1,...,x_p$ and reading along the circle skeleton, as on the right of Figure \ref{fig:CycWord}. Indeed, this is an isomorphism, and gives the explicit description of $\gr \beta$.
\end{remark}

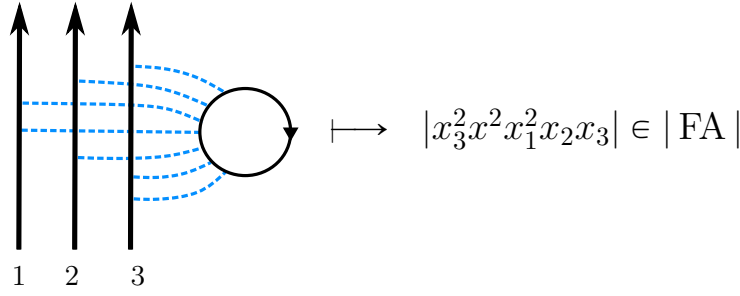
\begin{figure}
\centering
\begin{picture}(300,120)
\put(0,0){\includegraphics[width=12cm]{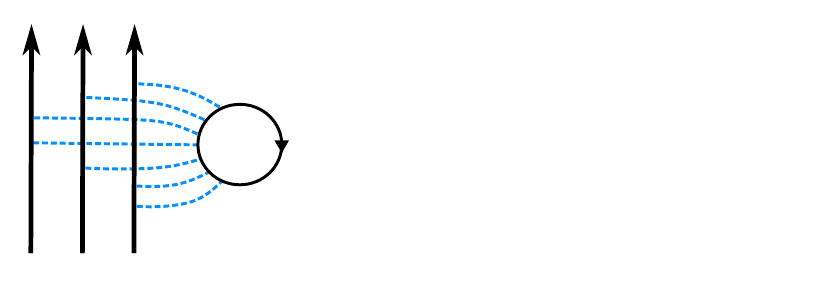}}
\put(10,5){1}
\put(30,5){2}
\put(55,5){3}
\put(130,61){\Small |} 
\put(130,60){\Large $\longrightarrow$}
\put(165, 60){\Large $|x_3^2x_2x_1^2x_2x_3|\in |\As| $}
\end{picture}
    \caption{An example demonstrating how chord diagrams with no strand-strand chords can be read as cyclic words in $|\As|$.}
    \label{fig:CycWord}
\end{figure}

Having identified the domain of the Goldman Bracket, $\Cpa \otimes \Cpa$, as $\tT^{/1}(\circleSkel)\otimes \tT^{/1}(\circleSkel)$ through the identification $\beta$, we can now show that the Goldman bracket is induced -- in the sense of Section~\ref{sec:conceptsum} -- by the stacking commutator on $\ctT$.

\begin{thm}\label{thm:bracketsnake}
Let $\glosm{lambda1}{\lambda_1}:\tTnab^{/2}(\circleSkel)\otimes \tTnab^{/2}(\circleSkel)\rightarrow \tTnab^{/2}(\circleSkel \; \circleSkel)$ denote the stacking product. 
Let $\glosm{lambda2}{\lambda_2}$ denote the opposite product given by $\lambda_2(K_1,K_2)=K_2K_1$. 
 The stacking commutator $\glosm{lambda}{\lambda}=\lambda_1-\lambda_2$ induces a unique map $\hat{\eta}: \tT^{/1}(\circleSkel) \otimes \tT^{/1}(\circleSkel) \to \tT^{/1}(\circleSkel)$, in the sense of the commutative diagram in Figure~\ref{fig:Snakeforbracket}. The map $\glosm{etahat}{\hat{\eta}}$ coincides with the Goldman bracket on $\Cpa$ via the identification $\beta:\tT^{/1}(\circleSkel)\xrightarrow{\cong}\Cpa$, that is, 
$$[-,-]_G=\beta \circ \hat{\eta} \circ (\beta^{-1}\otimes \beta^{-1}).$$
\end{thm}

\begin{figure}
\begin{tikzcd}[ column sep=small]
0\arrow[r]&\operatorname{Ker} \arrow[r]\arrow[d, "0",swap]  & \tTnab^{/2}(\circleSkel)\otimes \tTnab^{/2}(\circleSkel) \arrow[r] &\tT^{/1}(\circleSkel)\otimes \tT^{/1}(\circleSkel)\arrow[d,"0"] \arrow[r] \arrow[ddll, dashed, out=-140, in=0, looseness=.6 ,"\hat{\eta}"]
 \arrow[dll, swap, dashed, out=110, in=70, looseness=.7 ,"\eta"]
&0\\
0\arrow[r]& \tTnab^{1/2}(\circleSkel \text{\hspace{.75mm}} \circleSkel) \arrow[r]
& \tTnab^{/2}(\circleSkel \text{\hspace{.75mm}}\circleSkel)\arrow[from=u,"\lambda",swap,  near start, crossing over] \arrow[r]
& \tT^{/1}(\circleSkel \text{\hspace{.75mm}}\circleSkel)\arrow[r]&0\\
&\tT^{/1}(\circleSkel)\arrow[u,hook, "\widehat{b}"]
\end{tikzcd}
\caption{Recovering the Goldman bracket. The horizontal maps are the natural quotient and inclusion maps, and $\operatorname{Ker}$ denotes the kernel of the consecutive projection. The map $\widehat{b}$ denotes multiplication by $b$ (Lemma~\ref{lem:mbOnCircle}). 
}
\label{fig:Snakeforbracket} 
\end{figure}

\begin{proof}
    The vector space $\tTnab^{/2}(\circleSkel)$ is generated by the equivalence classes of knots in $M_p$. 
    For $K_1, K_2\in \tT$, we abuse notation and denote by $K_1 \otimes K_2$ the class of $K_1\otimes K_2$ in $\tTnab^{/2}(\circleSkel)\otimes \tTnab^{/2}(\circleSkel)$. The stacking commutator $\lambda(K_1\otimes K_2)=K_1K_2-K_2K_1$ is the difference between placing $K_2$ above or below $K_1$ in $D_p\times I$.
    
    We first show that the right hand square of Figure~\ref{fig:Snakeforbracket} commutes. Regularly project $K_1$, $K_2$ and their stacking products to the bottom $D_p$ to obtain knot diagrams $D_1$ and $D_2$, and link diagrams $D_1D_2$ and $D_2D_1$. A \textit{mixed crossing} of a link diagram is a crossing where the two strands belong to separate components. Notice that $D_2D_1$ is precisely $D_1D_2$ with all mixed crossings flipped. 

    Number the mixed crossings of $D_1D_2$ from 1 to $r$, and let $L_i$ denote the link diagram where the first $i$ mixed crossings have been flipped. Specifically, $L_0=D_1D_2$ and $L_r=D_2D_1$ Then $L_0-L_r=D_1D_2-D_2D_1$ can be written as a telescopic sum:
    \begin{equation}\label{eq:Telescope}
    D_1D_2-D_2D_1=(L_0-L_1)+(L_1-L_2)+...+(L_{r-1}- L_r).
    \end{equation}
    In the sum, each term in parenthesis is a two-component link with a single mixed double point, with a sign (the crossing sign of the $i$-th mixed crossing). Thus, $\lambda(K_1,K_2)\in \tTnab^{1}$, and maps to zero in $\tTnab^{/1}$. Hence, the right hand square commutes. 

    We now turn to the left square of the diagram. The kernel of the projection map 
    $$\tTnab^{/2}(\circleSkel)\otimes\tTnab^{/2}(\circleSkel)\rightarrow  \tTnab^{/1}(\circleSkel)\otimes\tTnab^{/1}(\circleSkel)$$ 
    is generated by $ \tTnab^{1/2}(\circleSkel)\otimes \tTnab^{/2}(\circleSkel)$ 
    and $ \tTnab^{/2}(\circleSkel)\otimes \tTnab^{1/2}(\circleSkel) $. Suppose that $K_1\otimes K_2$ is in $ \tTnab^{1/2}(\circleSkel)\otimes \tTnab^{/2}(\circleSkel)$; without loss of generality assume that $K_1$ is a knot with one double point.
    Then, by Equation~\ref{eq:Telescope}, $\lambda(K_1\otimes K_2)$ can be written as a telescopic sum of links with two double points each, hence it is zero in $\tTnab^{/2}(\circleSkel \, \circleSkel)$. Therefore, the left square commutes.

    As in Section~\ref{sec:conceptsum}, then $\lambda$ induces a unique map 
    $$\eta:\tT^{/1}(\circleSkel)\otimes \tT^{/1}(\circleSkel)\rightarrow \tTnab^{1/2}(\circleSkel \text{\hspace{.75mm}} \circleSkel).$$ 
    We can now identify $\eta$ as the Goldman bracket.
    The isomorphism $\beta$ gives $\tT^{/1}(\circleSkel)\cong \Cpa$ (Proposition \ref{prop:kerbeta}), and so, identifies the domain of $\eta$ with the domain of the Goldman bracket. We will argue that the image of $\eta$ also lies in $\tT^{/1}(\circleSkel)\cong \Cpa$.
    
    By Equation~\eqref{eq:Telescope}, $\lambda(K_1,K_2)$ can be written a sum of $r$ terms, each with one mixed double point. Applying the Conway relation to each of the $r$ terms of the telescopic sum by smoothing the mixed double points changes the skeletons from two circles to one circle, and introduces a factor of $b$: 
    \begin{equation}\label{eq:ConwayTel}
    \lambda(K_1\otimes K_2)=D_1D_2-D_2D_1\stackrel{\nab}=b(\epsilon_1 K_{s_1}+\epsilon_2 K_{s_2}+...+\epsilon_r K_{s_r}).
    \end{equation}
    Here $K_{s_i}$ denotes the knot obtained from $L_{i-1}-L_i$ by smoothing the mixed double point (which is obtained from the $i$-th mixed crossing), and $\epsilon_i$ is the sign of the $i$-th mixed crossing.
    That is, $\lambda(K_1,K_2) \in b\tTnab^{/2}(\circleSkel)$. In other words, $\eta$ factors through $\tT^{/1}(\circleSkel)$, which embeds in $\tTnab^{1/2}(\circleSkel \, , \, \circleSkel)$ via the multiplication by $b$ map $\widehat{b}$, by Lemma~\ref{lem:mbOnCircle}. Hence, we obtain the map $\hat{\eta}: \tT^{/1}(\circleSkel) \otimes \tT^{/1}(\circleSkel) \to \tT^{/1}(\circleSkel)$, as needed.  

    Finally, we check that $\hat{\eta}$ coincides with the Goldman bracket via the identification $\beta$. For curves $\gamma_1\otimes \gamma_2\in\tT^{/1}(\circleSkel)\otimes \tT^{/1}(\circleSkel)$, let $K_1\otimes K_2 \in \tTnab^{/2}(\circleSkel)\otimes \tTnab^{/2}(\circleSkel)$ be an arbitrary pre-image (vertical lift) of $\gamma_1\otimes \gamma_2$.
    Then 
    \[\eta(\gamma_1\otimes \gamma_2)=\frac{\lambda(K_1\otimes K_2)}{b} \in \tT^{/1}(\circleSkel),\]
    where we use the notation $\frac{1}{b}$ to mean composition with $\widecheck{b}$. The Goldman bracket (Definition~\ref{def:bracket}) is precisely a sum of smoothings of the mixed crossings of $\gamma_1$ and $\gamma_2$. The only thing to check is that the crossing signs coincide with the negative signs of the local coordinate systems in the Goldman bracket definition, which is indeed the case. See Figure~\ref{fig:combracket} for an example.
\end{proof}

    \begin{figure}
        \begin{picture}(375,80)
   \put(0,0){ \includegraphics[scale=.5]{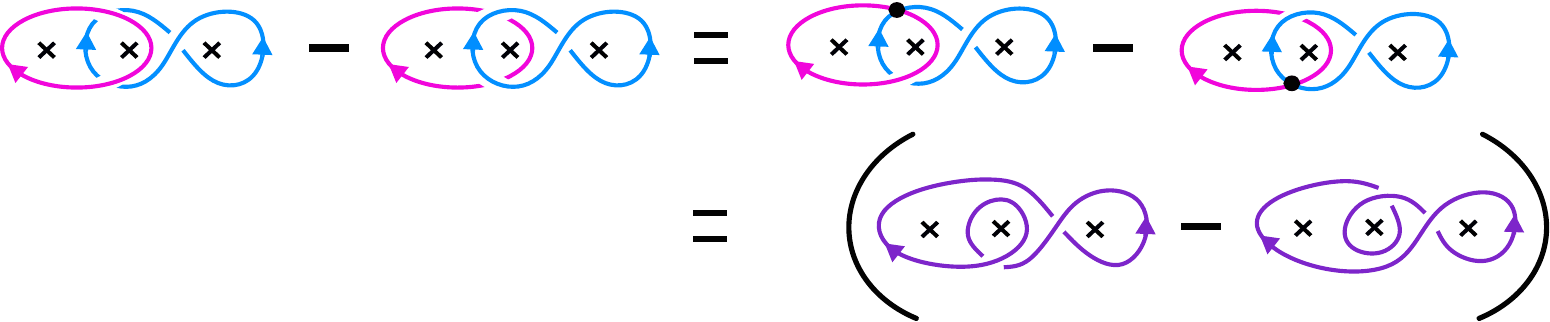}}
   \put(168,32){\Large{$\nabla$}}
   \put(190,20){\Large{$b$}}
\end{picture}
        
        \caption{Example commutator bracket computation. The first equality is true after canceling terms in a telescoping expansion of the double points. }
        \label{fig:combracket}
    \end{figure}

Recall that the graded Goldman bracket (Proposition \ref{prop:gr_bracket_def}) is a linear map $[-,-]_{\gr G}:|\As|\otimes |\As|\to |\As|$, and by Proposition \ref{prop:gr_beta} we have an identification $\gr\beta: |\As|\xrightarrow{\cong}\tA^{/1}(\circleSkel)$. Applying the associated graded functor -- with respect to the total filtration -- to the diagram in Figure \ref{fig:Snakeforbracket}, we obtain the commutative diagram in Figure \ref{fig:Snakefor_gr_bracket} and recover the graded Goldman bracket:

\begin{figure}
\begin{tikzcd}[column sep=small]
0 \arrow[r] &\operatorname{Ker} \arrow[r]\arrow[d, "0",swap]  & \tA_{\nab}^{/2}(\circleSkel)\otimes \tA_{\nab}^{/2}(\circleSkel) \arrow[r] &\tA^{/1}(\circleSkel)\otimes \tA^{/1}(\circleSkel)\arrow[d,"0"] \arrow[r] \arrow[ddll, dashed, out=-140, in=0, looseness=.6 ,"\gr \hat{\eta}"]
\arrow[dll, swap, dashed, out=110, in=70, looseness=.7 ,"\gr\eta"]
&0\\
0\arrow[r]& \tA_{\nab}^{1/2}(\circleSkel \text{\hspace{.75mm}} \circleSkel) \arrow[r]
& \tA_{\nab}^{/2}(\circleSkel \text{\hspace{.75mm}}\circleSkel)\arrow[from=u,"\gr\lambda",swap,  near start, crossing over] \arrow[r]
& \tA_{\nab}^{/1}(\circleSkel \text{\hspace{.75mm}}\circleSkel) \arrow[r] & 0\\
&\tA^{/1}(\circleSkel)\arrow[u,hook]
\end{tikzcd}
\caption{Recovering the graded Goldman bracket by applying the associated graded functor to the commutative diagram of Figure \ref{fig:Snakeforbracket}.}\label{fig:Snakefor_gr_bracket}
\end{figure}

\begin{cor}\label{cor:snakefor_gr_bracket}
    The diagram in Figure~\ref{fig:Snakefor_gr_bracket} commutes, the rows are exact, $\gr \eta$ is the induced connecting homomorphism, and $\gr\hat{\eta}$ agrees with the associated graded Goldman bracket via the identification $\gr\beta: \tA^{/1}(\circleSkel)\xrightarrow{\cong} |\As|$. In other words,
\[\gr[\cdot,\cdot]_G=\gr \beta \circ \gr \hat{\eta}\circ(\gr \beta^{-1}\otimes \gr \beta^{-1}).\]
\end{cor}

\begin{proof}
    All arrows in the diagram in Figure~\ref{fig:Snakeforbracket} are filtered maps with respect to the total filtration; the rows are exact; and the total filtrations on the left and right hand sides are induced from the total filtration in the middle. Hence, Corollary \ref{cor:gr_induced_is_unique} applies, and hence the $\gr$ functor gives a commutative diagram with exact rows, as in Figure  \ref{fig:Snakefor_gr_bracket}. By the uniqueness of the connecting homomorphism, we know that it is $\gr \eta$. By the functoriality of the associated graded, the graded Goldman bracket is given by
    \[\gr[\cdot,\cdot]_G=\gr \beta \circ \gr \hat{\eta}\circ(\gr \beta^{-1}\otimes \gr \beta^{-1}).\] 
    \end{proof}

\begin{example}\label{ex:grGoldman}
    While the Corollary~\ref{cor:snakefor_gr_bracket} follows from abstract considerations, let us demonstrate the on an example the explicit calculation of the graded bracket. Recall from Remark~\ref{rem:ChorsOnPoles} that in $\tA^{/1}$ chord endings on the poles commute. The identification $\gr \beta$ works by assigning a letter $x_1,...,x_p$ to each pole, and reading off the cyclic word given by the chords along the circle skeleton component, as in Figure~\ref{fig:CycWord}.

    We compute the graded bracket of the words |$x_1x_2^2|$ and $|x_2x_3^2|$, via $\gr \beta$. The two cyclic words correspond to chord diagrams in $\tA^{/1}(\circleSkel)$, which we then consider in (lift to) $\tA^{/2}_\nab(\circleSkel)$. The map $\gr\lambda$ is the stacking commutator of these diagrams, as shown in Figure~\ref{fig:GradedBracket}. This lies in $\tA^{1/2}_\nab (\circleSkel\, \circleSkel)$, which is easiest to see via applying a 4T relation for the letter coincidence $x_2$, as shown in Figure~\ref{fig:GradedBracket}. In turn, via an application of the Conway relation, it is easy to see that the element of $\tA^{/1}(\circleSkel)$ which maps to this via multiplication by $b$ is $|x_1^2x_2x_3^2|-|x_1^2x_3^2x_2|$. This is precisely the value of the graded Goldman bracket: compare also with Figure~\ref{fig:grbracket}.   
\end{example}

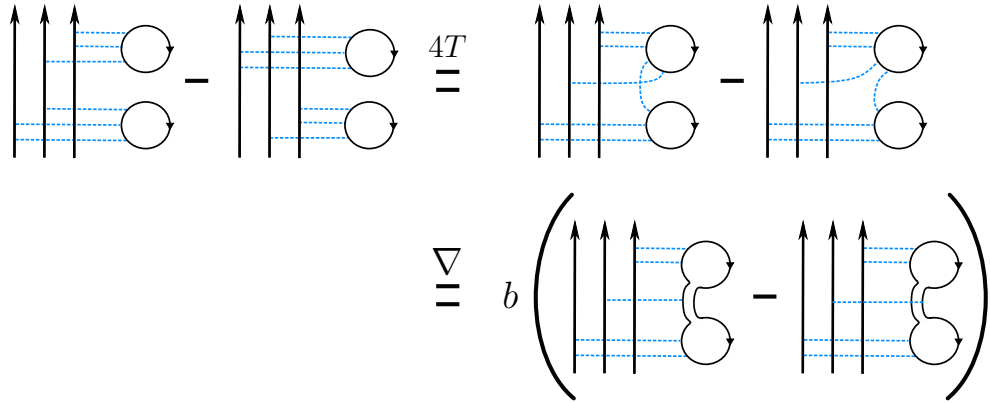
\begin{figure}
\begin{picture}(375,180)
   \put(0,0){ \includegraphics[width=13cm]{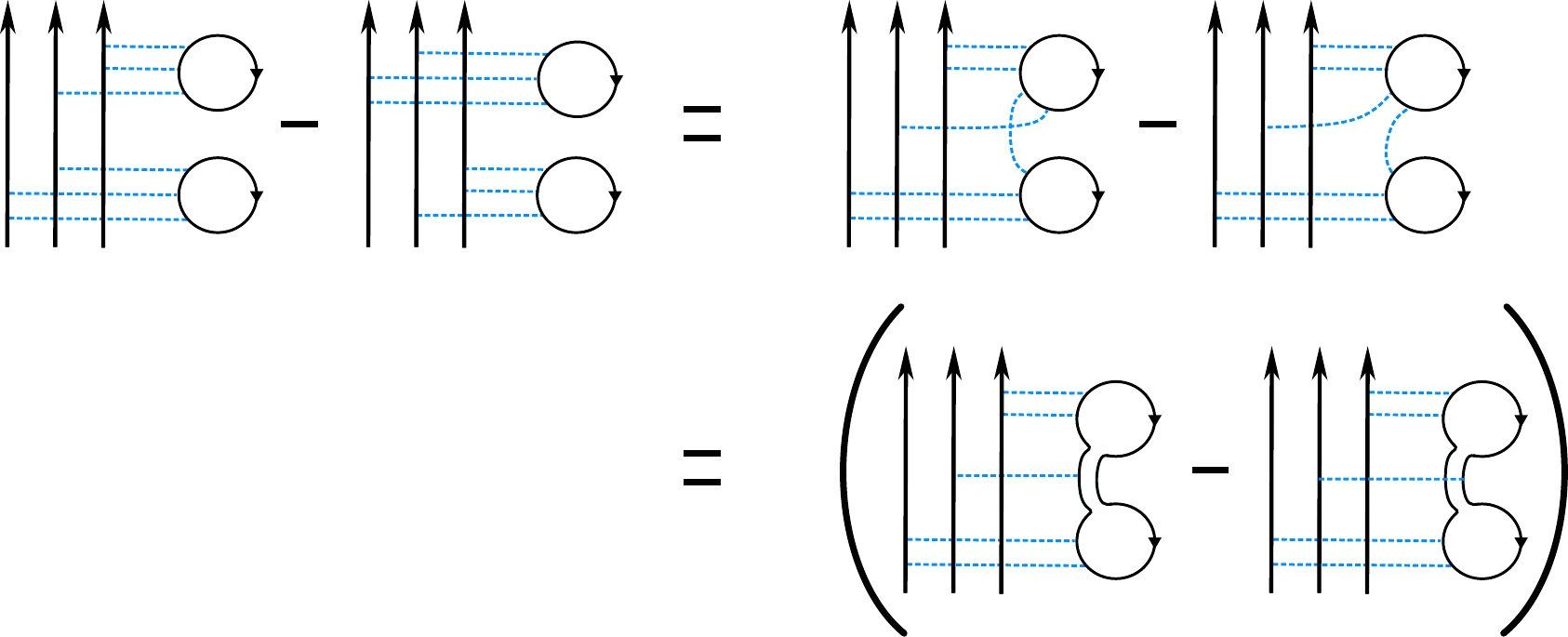}}
   \put(162,130){\large{$4T$}}
   \put(163,48){\Large{$\nabla$}}
   \put(190,35){\Large{$b$}}
\end{picture}
\caption{Example calculation for the diagrammatic realisation of the graded Goldman bracket.}\label{fig:GradedBracket}
\end{figure}

\begin{thm}\label{thm:Cube_for_bracket}
    The Kontsevich integral descends to a homomorphic expansion for the Goldman bracket,
    that is, the following diagram commutes:
   \[ \begin{tikzcd}
\Cpa\arrow[d, "Z^{/1}"] && \Cpa \otimes \Cpa \arrow[swap]{ll}{[\cdot,\cdot]_G} \arrow[d, "Z^{/1}\otimes Z^{/1}"]\\
\hAsa  &&\hAsa \otimes \hAsa  \arrow[swap]{ll}{\gr [\cdot,\cdot]_G}
\end{tikzcd}\]
    \end{thm}

\begin{proof}
The Kontsevich integral is homomorphic with respect to the stacking product (Proposition \ref{prop:Zhomom}). Since $\lambda$, the key ingredient in our construction of $[\cdot,\cdot]_G$, is the difference between the stacking product and its opposite product, $Z$ is homomorphic with respect to $\lambda$, thus the following square commutes:
\[
    \begin{tikzcd}[column sep=tiny, row sep=tiny]
       & \tTnab^{/2}(\circleSkel)\otimes \tTnab^{/2}(\circleSkel)\arrow[dl,"\lambda"]\arrow[dd,"Z^{/2}\otimes Z^{/2}"]\\
       \tTnab^{/2}(\circleSkel \text{\hspace{.75mm}} \circleSkel)\arrow[dd, "Z^{/2}",swap]\\
       &\tA_{\nab}^{/2}(\circleSkel)\otimes \tA_{\nab}^{/2}(\circleSkel)\arrow[dl,"\gr \lambda"]\\
       \tA_{\nab}^{/2}(\circleSkel \text{\hspace{.75mm}} \circleSkel)
    \end{tikzcd}
    \]

Hence, we know that that the entire multi-cube \eqref{eq:BracketMultiCube} is commutative: all other faces follow from Theorem \ref{thm:bracketsnake}, Corollary \ref{cor:snakefor_gr_bracket}, the fact that $Z$ is a filtered map with respect to the $s$-filtration (Proposition~\ref{prop:ZrespectsS}):

\begin{equation}\label{eq:BracketMultiCube}
\begin{tikzcd}[scale cd=.8,row sep=small, column sep=small]
 & \operatorname{Ker} \arrow[dl,"0"] \arrow[rr] \arrow[dd,"Z^{1/2}\otimes Z^{1/2}",near start] & & \tTnab^{/2}(\circleSkel)\otimes \tTnab^{/2}(\circleSkel) \arrow[dl,"\lambda"]\arrow[rr] \arrow[dd,"Z^{/2} \otimes Z^{/2}",near start] && \tT^{/1}(\circleSkel)\otimes \tT^{/1}(\circleSkel)\arrow[dd,"Z^{/1}\otimes Z^{/1}"]\arrow[dl,"0"]   \\
\tTnab^{1/2}(\circleSkel\text{\hspace{.75mm}} \circleSkel) \arrow[rr, crossing over] \arrow[dd,"Z^{1/2}"] & & \tTnab^{/2}(\circleSkel \text{\hspace{.75mm}} \circleSkel)\arrow[rr,crossing over]\ &&\tT^{/1}(\circleSkel\text{\hspace{.75mm}} \circleSkel) \\
 & \operatorname{Ker} \arrow[dl,"0"] \arrow[rr] & & \tA_{\nab}^{/2}(\circleSkel) \otimes \tA_{\nab}^{/2}(\circleSkel) \arrow[dl,"\gr \lambda", near start]\arrow[rr] && \tA^{/1}(\circleSkel)\otimes \tA^{/1}(\circleSkel) \arrow[dl,"0"]\\
 \tA_{\nab}^{1/2}(\circleSkel \text{\hspace{.75mm}} \circleSkel) \arrow[rr] & & \tA_{\nab}^{/2}(\circleSkel \text{\hspace{.75mm}} \circleSkel) \arrow[rr]\arrow[from=uu, "Z^{/2}",near start, crossing over]&& \tA^{/1}(\circleSkel \text{\hspace{.75mm}} \circleSkel) \arrow[from=uu, "Z^{/1}",near start, crossing over]\\
\end{tikzcd}
\end{equation}

Hence, using the naturality of the induced map construction (Lemma~\ref{lem:Naturality} and the diagram \eqref{eq:Cube}), we then know that the middle square of \eqref{eq:EtaSquare} commutes:

\begin{equation}\label{eq:EtaSquare}
    \begin{tikzcd}[scale cd=.8,row sep=small, column sep=small]
     \tT^{/1}(\circleSkel)\arrow[d,hook]\arrow[dddd,"Z^{/1}", out=200, in=160]\\
      \tTnab^{1/2}(\circleSkel\text{\hspace{.75mm}} \circleSkel) \arrow[dd,"Z^{1/2}"] 
    & &&& &&& \tT^{/1}(\circleSkel)\otimes \tT^{/1}(\circleSkel)\arrow[dd,"Z^{/1}\otimes Z^{/1}"] \arrow[lllllll,"\eta"]  \arrow[ulllllll,"\hat{\eta}",swap]
    \\ \\
     \tA_{\nab}^{1/2}(\circleSkel \text{\hspace{.75mm}} \circleSkel)
    & &&&  &&& \tA^{/1}(\circleSkel)\otimes \tA^{/1}(\circleSkel)\arrow[lllllll,"\gr\eta",swap]\arrow[dlllllll,"\gr\hat{\eta}"] \\
      \tA^{/1}(\circleSkel)\arrow[u,hook]
     \end{tikzcd}
\end{equation}
Since all other faces of the diagram~\eqref{eq:EtaSquare} are commutative by definition, the outside square also commutes. In turn, this is the middle square of the diagram~\eqref{eq:KIntBracket}:
\begin{equation}\label{eq:KIntBracket}    
\begin{tikzcd}
\Cpa\arrow[d, "Z^{/1}"] &\tT^{/1}(\circleSkel) \arrow[l,"\cong"',"\beta"]\arrow[d, "Z^{/1}"] && \tT^{/1}(\circleSkel)\otimes \tT^{/1}(\circleSkel) \arrow{ll}{\hat{\eta}} \arrow[d, "Z^{/1}\otimes Z^{/1}"] &&\Cpa \otimes \Cpa \arrow[ll,"\cong"',"\beta^{-1}\otimes \beta^{-1}"]\arrow[bend right=20,swap]{lllll}{[\cdot,\cdot]_G} \arrow[d, "Z^{/1}\otimes Z^{/1}"]\\
\hAsa &\tA^{/1}(\circleSkel) \arrow[l,"\cong"',"\gr\beta"]&& \tA^{/1}(\circleSkel)\otimes \tA^{/1}(\circleSkel) \arrow[ll, "\gr \hat{\eta}"] &&\hAsa \otimes \hAsa \arrow[ll,"\cong"',"\gr\beta^{-1}\otimes \gr\beta^{-1}"] \arrow[bend left=20,swap]{lllll}{\gr [\cdot,\cdot]_G}
\end{tikzcd}
\end{equation}
Once again, all other faces of \eqref{eq:KIntBracket} are commutative: by Theorem~\ref{thm:bracketsnake} and Corollary~\ref{cor:snakefor_gr_bracket} at the top and bottom; and otherwise by definition. Hence, the outside square commutes, and this is the statement of the theorem.
\end{proof}

\subsection{The Turaev co-bracket}\label{sec:cobracketinCON}
In  Section~\ref{subsec:IntroGT} we reviewed the definition of the Turaev cobracket on $\Cpa$ via the map $\mu: \tCp \to \Cpa \otimes \Cp$, which required choosing a rotation number $1/2$ representative for curves in $\tCp$. The knot-theoretic version for the cobracket lifts this construction. 

We start by interpreting $\tCp$ in the context of tangles.  
Let $\glosm{botskel}{\botSkel}$ denote an interval skeleton component where both endpoints are on the bottom $D_p \times\{0\}$. Inspired by \cite{H,HM} we call a tangle with skeleton $\botSkel$ a \emph{bottom tangle}.  We mark the endpoints of the interval by $\bullet$ and $*$, as in Figure ~\ref{fig:ascending}. Furthermore, we denote by $\tT(\circleSkel^k \, \botSkel^\ell)$ tangles with $k$ circle skeleton components, and $\ell$ bottom intervals.

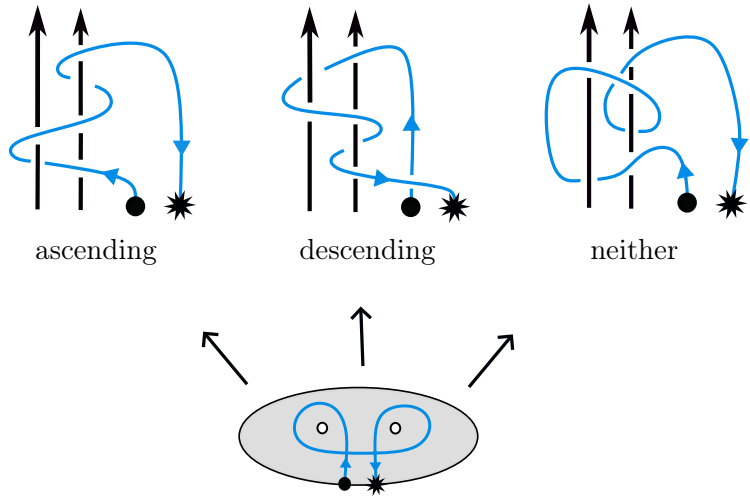
\begin{figure}
\centering
\begin{picture}(280,200)
    \put(0,0){\includegraphics[scale=.7]{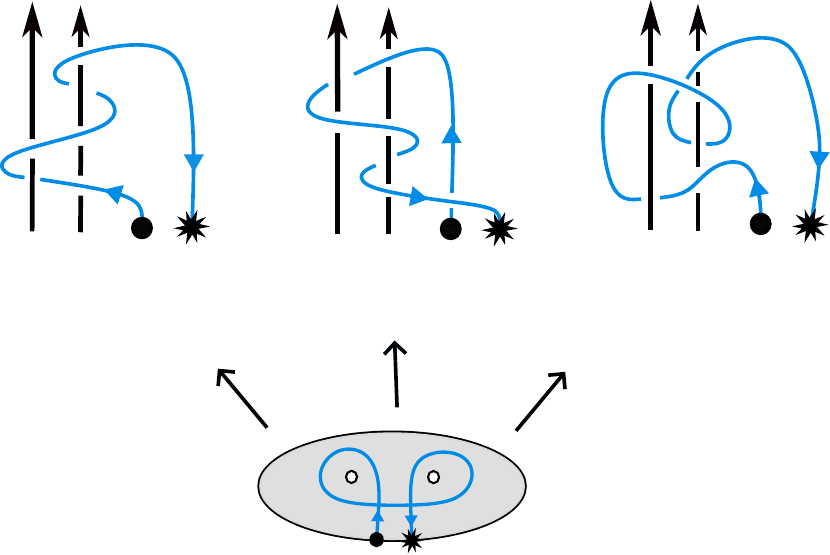}}
    \put(10,90){ascending}
    \put(110,90){descending}
    \put(220,90){neither}
\end{picture}
\caption{A curve in $\Cp$ lifted to ascending, descending, and neither ascending nor descending bottom tangles. The three tangles are equivalent in $\tT^{/1}$, but distinct in $\tT$.}\label{fig:ascending}
\end{figure}

We extend the projection map $\beta$ (Proposition ~\ref{prop:BotProj}) to such tangles to obtain an isomorphism similar to Corollary ~\ref{prop:BotProj}:

\begin{prop}\label{prop:ascispi} The natural bottom projection, post-composed with closing up open paths by concatenation with paths $\nu$ from the endpoint to the starting point along the boundary (as in Section~\ref{subsec:IntroGT}) gives a filtered linear map $$\beta: \C \tT(\circleSkel^k \, \botSkel^\ell) \to \Cpa^{\otimes k} \otimes \Cp^{\otimes \ell}$$ whose kernel is $\tT^{ 1}(\circleSkel^k \, \botSkel^\ell)$, hence it descends to a filtered isomorphism 
$$\beta: \tT^{/1}(\circleSkel^k \, \botSkel^\ell) \xrightarrow{\cong} \Cpa^{\otimes k} \otimes \Cp^{\otimes \ell}.$$
\end{prop}

\begin{proof} The proof is identical to the proof of Proposition~\ref{prop:BotProj}, aside from the minor issue of base points. In the bottom projection, a tangle strand from $\bullet$ to $*$ projects to a homotopy class of a path from $\bullet$ to $*$ in $D_p$. Such paths are in bijection with $\C\pi$ via composition with a path $\nu$ along the boundary from $*$ to $\bullet$.
\end{proof}

By inspection of the associated graded map, we obtain:

\begin{prop}\label{prop:gr_beta_bot_tangle} The associated graded map 
$$\gr\beta: \tA(\circleSkel^k \, \botSkel^\ell) \to |\As|^{\otimes k}\otimes \As^{\otimes \ell}$$ has kernel $\tA^{\geq 1}(\circleSkel^k \, \botSkel^\ell)$, hence, $\gr\beta$ descends to a graded isomorphism $$\gr\beta: \tA^{/1}(\circleSkel^{k} \, \botSkel^\ell) \xrightarrow{\cong} |\As|^{\otimes k}\otimes \As^{\otimes \ell}.$$ In particular, we have
$\gr\beta: \tA^{/1}(\botSkel) \xrightarrow{\cong} \As.$
\end{prop}

We also extend the statements about multiplication and division by $b$ to the context of mixed skeleta:

\begin{prop}\label{prop:qbonbottomtangles} The map $\widehat{b}$ descends to $\C$-linear isomorphism
\[\widehat{b}:\tT^{/1}(\circleSkel \, \botSkel)\xrightarrow{\cong} \tTnab^{1/2}(\botSkel),\]
with inverse map given by $\widecheck{b}$, division by $b$.
\end{prop}

    \begin{proof} 
 From Corollary \ref{cor:divbyb}, we know $\widehat{b}$ is a $\C$-linear isomorphism  $\tT^{/1}\xrightarrow{\cong} \tTnab^{1/2}$, so we only need to address  the change in skeleton. The quotient $\tT^{/1}(\circleSkel \, \botSkel)$ is generated by classes of tangle diagrams $D$ with skeleton consisting of one circle and one bottom-to-bottom interval component. After multiplication by $b$, $b\cdot D$ is equivalent via the Conway relation to a tangle with one double point in $\tT(\botSkel)$, as the un-smoothing combines the two skeleton components.
\end{proof}

Next, we recover the self intersection map $\mu$, in the context of tangles, as the connecting homomorphism induced  from the difference between two ways to lift a bottom tangle.

Let $\bullet$ and $*$ be two ``nearby'' points on the boundary of $D_p$, that is, $*$ is obtained by shifting $\bullet$ slightly forwards along the boundary orientation, as shown in Figure~\ref{fig:ascending}. We will obtain a homomorphic expansion for the Turaev cobracket by computing the Kontsevich integrals of one-strand tangles which start at $\bullet$ and end at $*$. For this purpose, we will assume that $\bullet$ and $*$ are {\em arbitrarily close}: technically, we take the limit of Kontsevich integrals as $*$ approaches $\bullet$ backwards along the boundary. We suppress this in the notation, and write simply $\glosm{ZT}{Z(T)}$ when we mean $\displaystyle{\lim_{\bullet \leftarrow *} Z(T)}$.

\begin{definition}\label{def:asc+desc}
An embedding \[T: (I, \{0,1\}) \hookrightarrow (M_p, \{\bullet, *\} )\] (representing a bottom tangle) is called \emph{ascending} if it first ascends monotonically from $\bullet$, and then goes \emph{straight} down to $*$. More precisely, if $(z, s)$ is a global coordinate system for $M_p = D_p \times I$, then $T$ is an ascending tangle if there exists $c \in (0,1)$ such that when $t\in (0,c)$, the $\frac{d}{ds}$ component of $\dot{T}$ is positive; when $t \in (c+\epsilon,1)$, $\dot{T}$ is a negative constant multiple of $\frac{d}{ds}$; and when $t\in (c, c+\epsilon)$, $T$ smoothly transitions through a maximum.

Likewise, an embedding $T$ is \emph{descending} if it first goes straight up from $\bullet$, then monotonically descends to $*$. This can also be made precise as above. 
See Figure ~\ref{fig:ascending} for examples.
\end{definition}

\begin{definition}
An \emph{ascending tangle} is a bottom tangle in $M_p$ whose ambient isotopy class has an ascending embedding. Similarly, a  \emph{descending tangle} is a bottom tangle in $M_p$ whose ambient isotopy class has a descending embedding. 
\end{definition}

In the bottom projection, an ascending embedding will traverse each of its crossings on the under strand first, and on the over strand later. A descending embedding will traverse each crossing on the over strand first.

\medskip

To recover $\mu$, we need to define framed versions of the ascending and descending lifts.
Given a curve $\gamma$ in the fundamental group $\pi=\pi_1(D_p,*)$, $\gamma$ has a unique lift $\tilde{\gamma}$ in $\tpi=\tpi_{\bullet *}$ with the property that $\tilde{\gamma}\cdot \nu$ has rotation number zero, where $\nu$ is a path along the boundary from $*$ to $\bullet$, as in Figure~\ref{fig:DP}.

Up to isotopy, the curve $\tilde{\gamma}$ has a unique framed ascending tangle with writhe $+1$. Denote this framed ascending lift by $\glosm{lambdaa}{\lambda_a(\gamma)}$. Similarly, up to isotopy there is a unique framed descending lift of $\tilde{\gamma}$ with writhe $-1$, denoted $\glosm{lamdad}{\lambda_d(\gamma)}$.

To construct a diagram of $\lambda_a(\gamma)$ in the bottom projection, turn each self intersection of $\gamma$ into a crossing so that each crossing is traversed on the under-strand first: see Figure~$\ref{fig:FramedAscDesc}$. We claim that this diagram always has the correct write, $+1$, as follows. Recall that the rotation number of $\tilde{\gamma}\nu$ is zero. By Whitney's formula \cite{Whit}, the algebraic self-intersection number of $\tilde{\gamma}$ is one more than the rotation number: in this case, it is 1. In turn, in constructing an ascending diagram we have turned positive (respectively, negative) self intersections to positive (respectively, negative) crossings, hence the writhe equals the algebraic self intersection number. Thus, the writhe is $+1$, as required.

Similarly, to construct a diagram of $\lambda_d(\gamma)$ in the bottom projection, turn each self intersection of $\gamma$ into a crossing so that each crossing is traversed on the over-strand first: see the bottom line of Figure~$\ref{fig:FramedAscDesc}$. By a similar Whitney argument, this diagram always has writhe $-1$, as required. Diagrammatically, in the bottom projection one obtains the descending lift from the ascending lift by changing all (strand-strand) crossings of the tangle diagram.

When using different projection planes such as the back wall, one must take care to preserve the framed isotopy class of the tangle: in other words, to maintain the same write. In the back projection of an ascending lift, this means that a positive kink is the only strand-strand crossing: see on the top right in Figure~\ref{fig:FramedAscDesc}. In the back wall projection of a descending lift, there is a single negative strand-strand crossing, as shown in the bottom right of the same figure.

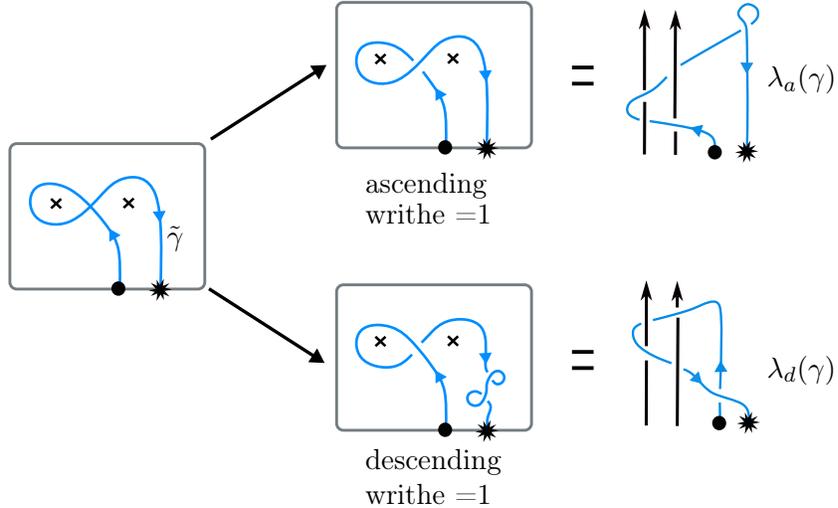
\begin{figure}
\begin{picture}(300,190)
    \put(0,15){\includegraphics[width=10cm]{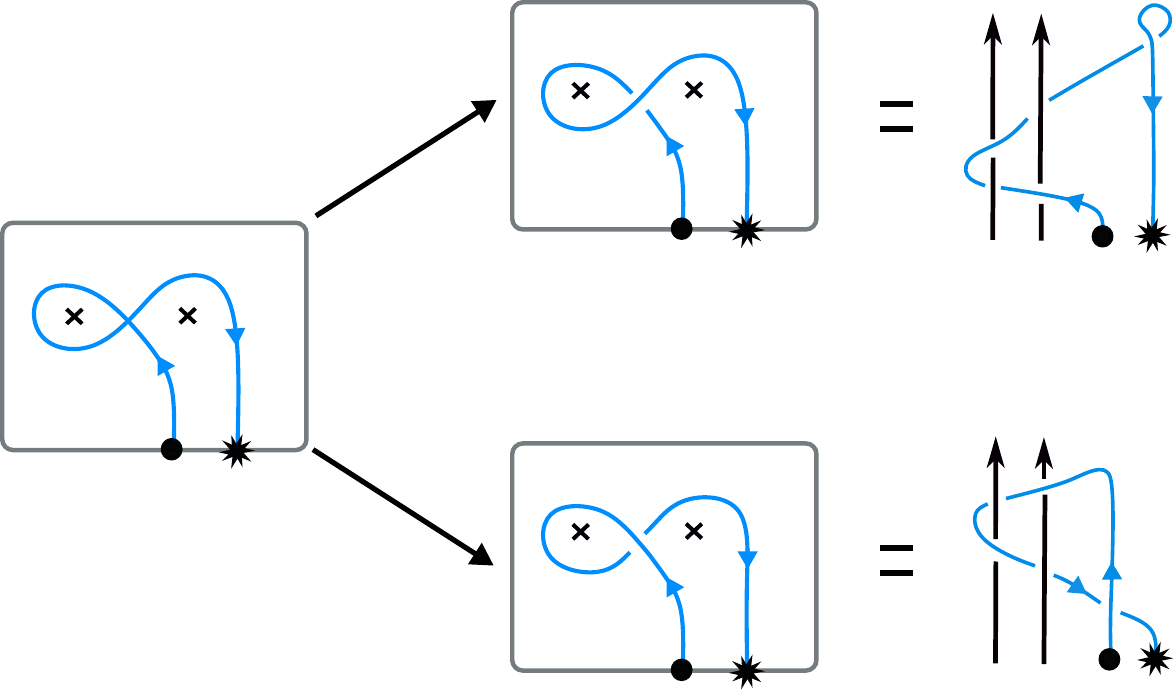}}
    \put(135,110){ascending}
    \put(135,98){writhe =1}
    \put(135,5){descending}
    \put(135,-8){writhe = -1}
    \put(60,90){$\tilde{\gamma}$}
    \put(287,40){$\lambda_d(\gamma)$}
    \put(287,150){$\lambda_a(\gamma)$}
\end{picture}
\caption{The framed ascending and descending lifts of a curve.}\label{fig:FramedAscDesc}
\end{figure}

Via the isomorphism $\beta$, $\lambda_a$ and $\lambda_d$ can be seen as maps $\tT^{/1}(\botSkel)\rightarrow \tT^{/2}(\botSkel)$.
Let $\glosm{lambdabar}{\bar{\lambda}}:\tT^{/1}(\botSkel)\rightarrow \tT^{/2}(\botSkel)$ denote the difference 
\[\bar{\lambda}(\gamma)=\lambda_a(\gamma)-\lambda_d(\gamma).\]

We are now ready to recover the self-intersection map $\mu:\Cp\to \Cpa\otimes \Cp$ (from Definition~\ref{def:mu}). Let $\glosm{q}{q}$ be the projection map from $\tT_\nab^{/2}(\botSkel)$ to $\tT^{/1}(\botSkel)$

\begin{thm}\label{thm:snake_for_mu}
    The map $\lambda=\bar{\lambda} \circ q$  induces a unique map 
    \[\hat{\eta}: \tT^{/1}(\botSkel) \to \tT^{/1}(\circleSkel\,\botSkel)\] 
    in the sense of the commutative diagram of Figure~\ref{fig:Snakeformu}. The map $\hat{\eta}$ agrees with the self-intersection map \[\mu:\Cp\to \Cpa\otimes \Cp\] under the identification $\beta: \tT^{/1}(\circleSkel \, \botSkel) \xrightarrow{\cong} \Cpa \otimes \Cp$. That is, \[\mu=\beta \circ\hat{\eta} \circ\beta^{-1}.\] 
\end{thm}

\begin{figure}
\begin{tikzcd}
&\tTnab^{1/2}(\botSkel) \arrow[r]\arrow[d, "0",swap]  & \tTnab^{/2}(\botSkel) \arrow[r,"q"] & \tT^{/1}(\botSkel)\arrow[d,"0"] \arrow[r]\arrow[dl,"\bar{\lambda}", swap] \arrow[ddll,near end, dashed, out=-120, in=0, looseness=.6 ,"\hat{\eta}"] \arrow[dll,swap, dashed, out=110, in=60, looseness=1 ,"\eta"]
&0\\
0\arrow[r]& \tTnab^{1/2}(\botSkel ) \arrow[r] \arrow[d,"\widecheck{b}"] 
& \tTnab^{/2}(\botSkel )\arrow[from=u,"\lambda=\bar{\lambda}\circ q",swap,  crossing over] \arrow[r]
& \tT^{/1}(\botSkel ) \\
&\tT^{/1}(\circleSkel \, \botSkel)
\end{tikzcd}
\caption{The nontrivial horizontal maps are the respective quotient maps, and $q$ is one such quotient map.}\label{fig:Snakeformu}
\end{figure}

\begin{proof}
We first show that the diagram of Figure~\ref{fig:Snakeformu} commutes. The commutativity of the left square is immediate from the exactness of the top row. The right side square is split into two triangles: of these the top one commutes by definition. The commutativity of the bottom triangle, that is, the fact that the post-composition of $\bar{\lambda}$ with the projection to $\tT^{/1}(\botSkel)$ is the zero map, is also straightforward, as follows. For a curve $\gamma$ in  $\tT^{/1}(\botSkel)\cong \Cp$, the map $\bar{\lambda}$ is the difference of two lifts of $\gamma$ to bottom tangles. When these lifts are subsequently projected back to $\tT^{/1}(\botSkel)$, they both project to $\gamma$, hence, their difference is 0. 

 Thus, as in Section~\ref{sec:conceptsum}, diagram \eqref{eq:inducedconnhom}, $\lambda$ induces a unique well defined  homomorphism $\eta:\tT^{/1}(\botSkel)\rightarrow \tTnab^{1/2}(\botSkel)$. By Proposition~\ref{prop:qbonbottomtangles}, the division by $b$ map $\widecheck{b}$ restricts to an isomorphism $\widecheck{b}: \tTnab^{1/2}(\botSkel) \to \tT^{/1}(\circleSkel \, \botSkel)$. The map $\hat{\eta}:\tT^{/1}(\botSkel)\rightarrow \tT^{/1}(\circleSkel \, \botSkel)$ is the composition $\hat{\eta}=\widecheck{b}\circ \eta$.

It remains to show that $\mu=\beta\circ \hat{\eta} \circ \beta^{-1}$. Given a curve $\gamma \in \Cpa$, the value of $\beta\circ \hat{\eta} \circ \beta^{-1}$ is calculated as follows: $\gamma$ is lifted to a curve of rotation number zero in $\tpi=\tpi_{\bullet *}$, and, in turn, interpreted as an element in $\tT^{/1}(\botSkel)$. The map $\bar{\lambda}$ is applied to this framed curve, to obtain a difference of tangles in $\tTnab^{1/2}(\botSkel)$. This value is divided by $b$, and interpreted as a loop and a curve in $\Cpa \otimes \Cp$.

Let $\gamma$ be a curve in $\tT^{/1}(\botSkel)\cong \Cp$, and let $\lambda_a(\gamma)=T_a$ be the framed ascending lift of $\gamma$  and $\lambda_d(\gamma)=T_d$ the framed descending lift. Then $\bar{\lambda}(\gamma)=T_a-T_d$. Denote the bottom projection of $T_a$ by $D_a$, and the bottom projection of $T_d$ by $D_d$. In particular, $D_d$ is obtained from $D_a$ by flipping all crossings.

As in the proof of Theorem~\ref{thm:bracketsnake}, number the crossings to be flipped from 1 to $r$. Let $D_i$ denote the link diagram where the first $i$ of the crossings have been flipped. Specifically, $D_0=D_a$ and $D_r=D_d$. Then $D_0-D_r$ can be written as a telescopic sum:
    \begin{equation}\label{eq:Telescope2}
    D_0-D_r=(D_0-D_1)+(D_1-D_2)+...+(D_{r-1}- D_r).
    \end{equation}

In the sum \eqref{eq:Telescope2} each term $(D_i-D_{i+1})$ contains one (signed) double point corresponding to a self-intersection of $\gamma$. We apply the Conway relation ($\doublepoint=b \cdot \upupsmoothing$) at these double points. A straightforward check shows that the sign arising from the crossing signs matches the sign $-\varepsilon_p$ in the definition of $\mu$ (Definition~\ref{def:mu}).  Thus, dividing by $b$ and reinterpreting via $\beta$ coincides with the value of $\mu(\gamma)$, as required. See Figure~\ref{fig:CobracketCalc} for an example. 
\end{proof}

\begin{figure}
    \centering
    \begin{picture}(380,150)
   \put(0,0){ \includegraphics[width=13cm]{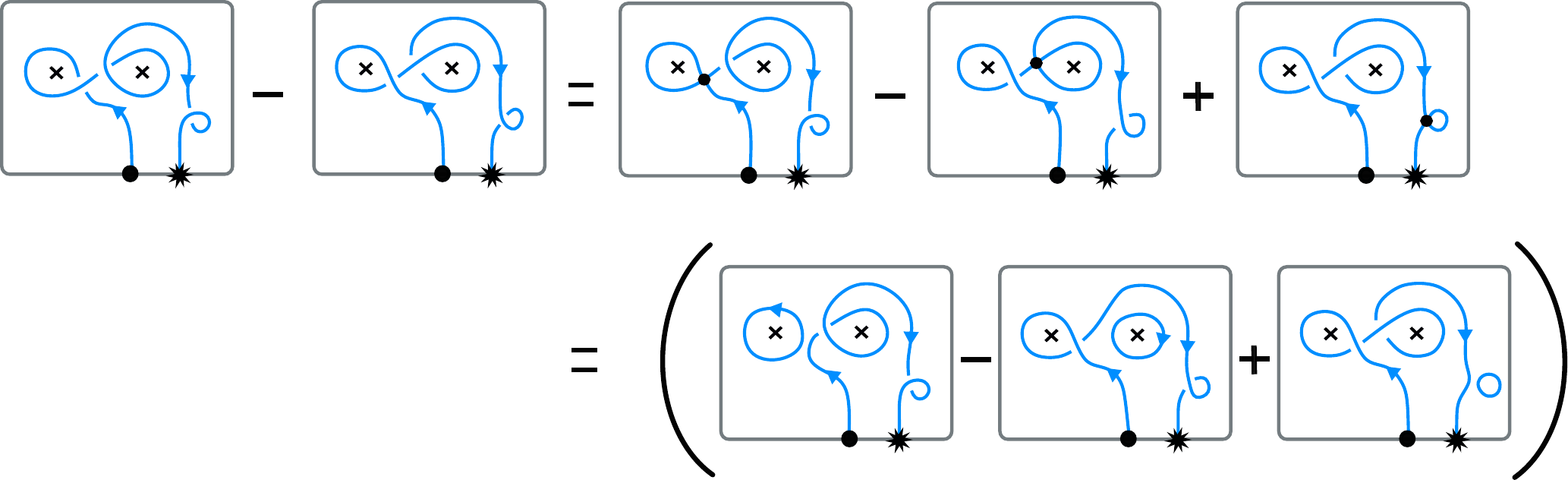}}
   \put(135,36){\Large{$\nabla$}}
   \put(150,25){\Large{$b$}}
\end{picture}
    \caption{An example computation of the map $\hat{\eta}$: at the top left is the difference of the writhe 1 ascending lift and the writhe -1 descending lift of a curve.}
    \label{fig:CobracketCalc}
\end{figure}

As with the Goldman bracket, the associated graded version of Theorem~\ref{thm:snake_for_mu} follows:

\begin{cor}\label{cor:grmu}
The diagram in Figure~\ref{fig:Snakefor_gr_cobracket} commutes, the rows are exact, $\gr \eta$ is the induced homomorphism, and $\gr \mu =\gr \beta \circ \gr \hat{\eta}\circ (\gr \beta)^{-1}$.
\end{cor}

\begin{proof}
The commutativity of the diagram (Figure~\ref{fig:Snakefor_gr_cobracket}) and the exactness of the rows follow from general principles in exactly the same way as Corollary~\ref{cor:snakefor_gr_bracket}. The rest is immediate.
\end{proof}

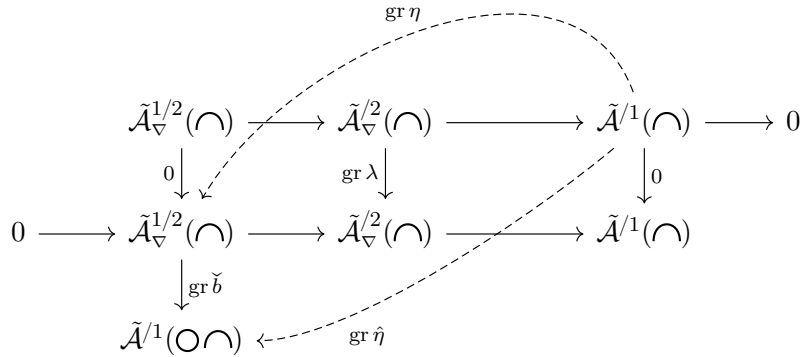
\begin{figure}
\begin{tikzcd}[]
&\tA_{\nab}^{1/2}(\botSkel) \arrow[r]\arrow[d, "0",swap]  & \tA_{\nab}^{/2}(\botSkel) \arrow[rr] & &\tA^{/1}(\botSkel)\arrow[d,"0"] \arrow[ddlll, dashed, out=-140, in=0, looseness=.6 ,"\gr\hat{\eta}", near end]\arrow[dlll,swap, dashed, out=110, in=60, looseness=1 ,"\gr\eta"] \arrow[r]& 0\\
0\arrow[r]& \tA_{\nab}^{1/2}(\botSkel ) \arrow[r] \arrow[d,"\gr \widecheck{b}"] 
& \tA_{\nab}^{/2}(\botSkel )\arrow[from=u,"\gr\lambda",swap,  crossing over] \arrow[rr]
&& \tA^{/1}(\botSkel )\\
&\tA^{/1}(\circleSkel \, \botSkel) \\
\end{tikzcd}
\caption{Associated graded diagram constructing the graded self-intersection map.}\label{fig:Snakefor_gr_cobracket}
\end{figure}

It is necessary for proving the formality statement -- that is, the compatibility of the Kontsevich integral with the bracket and cobracket -- to also have a concrete understanding of the associated graded map of $\lambda$. Recall from Remark  \ref{rem:ChorsOnPoles} and Proposition~\ref{prop:gr_beta_bot_tangle} that in $\tA^{/1}$ chord endings commute on the poles: this gives the isomorphism $\gr \beta: \tA^{/1}(\botSkel)\to \As$.

\begin{lem}\label{lem:GradedAscDesc}
Given a chord diagram $D\in \tA^{/1}(\botSkel)$, the map $\gr \lambda_a: \tA^{/1} \to \tA^{/2}_\nab$ orders the chord endings of $D$ in an ascending order along the poles, that is, the ordering along the poles match the ordering along the strand. Similarly, $\gr \lambda_d: \tA^{/1} \to \tA^{/2}_\nab$ orders the chord endings of $D$ along the poles in a descending order, that is, opposite to the ordering along the strand.
\end{lem}

\begin{proof}
    This is immediate from the definition of the associated graded map, by choosing a singular tangle $T_D$ representing $D$, and inspecting the chord diagram representing $\lambda_a(T_D)$ (respectively, $\lambda_d(T_D)$).
\end{proof}

Recall from Section~\ref{subsec:IntroGT} that the Turaev cobracket $\delta: \Cpa \to \Cpa \otimes \Cpa$ is constructed from $\mu: \Cp \to \Cpa \otimes \Cp$ by post-composing $\mu$ with the trace map $\Cp \to \Cpa$ in the second component and antisymmetrising (using $Alt(x\otimes y)=x\otimes y - y\otimes x$). The composition \[\tilde{\delta}=Alt\circ (1 \otimes |\cdot|) \circ \mu :\Cp \to \Cpa \otimes \Cpa\] descends to the Turaev cobracket $\delta: \Cpa \to \Cpa \otimes \Cpa $, as in Definition~\ref{def:cobrac}.

We mimic this construction in the context of tangle diagrams  by post-composing $\hat{\eta}$ with the closure map $\glosm{cl}{cl}:\tT(\botSkel)\rightarrow \tT(\circleSkel)$ on the open component, followed by anti-symmetrising, as shown in the diagram \eqref{eq:CobracketSteps}. 
\begin{equation}\label{eq:CobracketSteps}
\begin{tikzcd}[]
\tT^{/1}(\botSkel) \arrow[rr, "\hat{\eta}"]\arrow[rrd,"\hat{\zeta}"]\arrow[rrdd,"\tilde{\zeta}",swap]\arrow[dd, "cl",swap]  & & \tT^{/1}(\circleSkel \, \botSkel) \arrow[d, "cl"]\\
& &\tT^{/1}(\circleSkel)\otimes \tT^{/1}(\circleSkel)\arrow[d,"Alt"] \\
\tT^{/1}(\circleSkel)\arrow[rr,dashed, "\zeta"]& & \tT^{/1}(\circleSkel)\otimes \tT^{/1}(\circleSkel)\\
\end{tikzcd}
\end{equation}
The closure map connects the endpoints of the bottom tangle, $\bullet$ and $*$, by the path $\nu$ connecting $*$ to $\bullet$ along the bottom boundary $\partial D_p$, following the orientation (as in Figure~\ref{fig:DP}). We denote the map $cl \circ \hat{\eta}=:\glosm{zetahat}{\hat{\zeta}}$, and after antisymmetrisation $Alt \circ \hat{\zeta}=:\glosm{zetatilde}{\tilde{\zeta}}$. We will show that $\tilde{\zeta}$ descends to a map $\glosm{zeta}{\zeta}: \tT^{/1}(\circleSkel) \to \tT^{/1}(\circleSkel) \otimes \tT^{/1}(\circleSkel)$, which realises the Turaev cobracket via the identification $\beta$.

\begin{prop}\label{prop:zeta_to_del} The map $\zeta$ realises the Turaev cobracket $\delta$ via the identifications $\beta$, in the sense that the diagram in Figure~\ref{fig:Identifications} commutes.
\end{prop}

\begin{figure}
\begin{tikzcd}[]
\tT^{/1}(\circleSkel)\arrow[rrrrd,dashed, "\zeta", out=0, in=110, looseness=.6 ]\arrow[ddd,"\beta","="']\arrow[rdd, phantom,"(1)" description, out=-50, in=210, looseness=.6]&&&&\\
& \tT^{/1}(\botSkel)\arrow[lu,"cl"]\arrow[d,"\beta","="']\arrow[r,"\hat{\eta}"]\arrow[rrr,"\tilde{\zeta}",out=45, in=135, looseness=.3] \arrow[rd,phantom,"(2)" description]& \tT^{/1}(\circleSkel \, \botSkel)\arrow[r,"cl"]\arrow[d,"\beta","="']\arrow[rd,phantom,"(3)" description]&\tT^{/1}(\circleSkel)\otimes \tT^{/1}(\circleSkel)\arrow[r,"Alt"]\arrow[d,"\beta\otimes \beta","="']\arrow[rd,phantom,"(4)" description]&\tT^{/1}(\circleSkel)\otimes \tT^{/1}(\circleSkel)\arrow[d,"\beta\otimes \beta","="']\\
&\Cp \arrow[ld,"|\cdot |"]\arrow[r,swap, "\mu"]\arrow[rrr,"\tilde{\delta}",out=-45, in=-135, looseness=.3] &\Cpa\otimes \Cp \arrow[r,swap,"|\cdot |"]&\Cpa\otimes \Cpa \arrow[r,swap,"Alt"]&\Cpa\otimes \Cpa\\
\Cpa\arrow[rrrru,dashed,swap, "\delta",out=0, in=-110, looseness=.6 ]&&&&\\
\end{tikzcd}
\caption{The map $\zeta$ realises the Turaev cobracket $\delta$.}\label{fig:Identifications}
\end{figure}

\begin{proof}
 The only substantial statement is the commutativity of the square (2): this is Theorem~\ref{thm:snake_for_mu}. Squares (1) and (3) are the same: the closure map corresponds to the trace $\Cp\to \Cpa$. Square (4) is tautological. The maps $\tilde{\zeta}$ and $\tilde{\delta}$ are defined as the compositions shown: in the case of $\tilde{\delta}$ this is Definition~\ref{def:cobrac}. In particular, we have
 $\tilde{\delta}=(\beta\otimes \beta)\circ\tilde{\zeta}\circ\beta^{-1}.$ 
 The fact that $\tilde{\delta}$ descends to $\delta$ is Proposition~5.10 of \cite{akkn_g0}. The fact that $\tilde{\zeta}$ descends to $\zeta$ is immediate from the canonical identifications.
\end{proof}

\begin{cor}
The corresponding statement is true for the associated graded cobracket:
$$\gr\delta=(\gr\beta\otimes \gr\beta)\circ\gr{\zeta}\circ\gr\beta^{-1}.$$ 
\end{cor}

The key result left to prove is that $\zeta$ is homomorphic with respect to the Kontsevich integral $Z$: $\gr \zeta \circ Z= Z \circ \zeta$, and hence the Kontsevich integral descends to a homomorphic expansion for $\delta$. The subtlety involved is that $Z$ does {\em not} respect the map $\hat{\eta}$ -- there is an error term -- but it cancels after applying the closure and alternation.

The proof is based on the naturality of the induced homomorphisms, as outlined in Section~\ref{sec:conceptsum} and demonstrated in Section~\ref{sec:identifybracketinCON} for the Goldman bracket. The naive version of this idea would be to prove that all faces of a multi-cube similar to \eqref{eq:BracketMultiCube} commute. As before, the only non-trivial part of this statement is the commutativity of the middle square involving the map $\lambda$; unfortunately, in the case of the self-intersection map, this square fails to commute:   

\begin{equation}\label{eq:FailToCommute}
    \begin{tikzcd}[column sep=tiny, row sep=tiny]
       & \tTnab^{/2}(\botSkel)\arrow[dl,"\lambda"]\arrow[dd,"Z^{/2}"]\\
       \tTnab^{/2}(\botSkel)\arrow[dd, "Z^{/2}",swap]\arrow[rd, phantom,"\dncom" ]\\
       &\tA_{\nab}^{/2}(\botSkel)\arrow[dl,"\gr \lambda"]\\
       \tA_{\nab}^{/2}(\botSkel )
    \end{tikzcd}
 \end{equation}
This failure is mirrored in the setting of the Goldman--Turaev Lie bialgebra by the fact the the self-intersection map $\mu$ is {\em not formal}, only the Turaev cobracket obtained from it is. This is also observed in \cite[Theorem 7.1, Corollary 7.3]{Mas} -- where the correction term (there denoted $q$) is expressed in terms of the Gamma function of the associator underlying 
the Kontsevich integral -- as well as \cite{akkn_g0}. 

The resolution of this issue comes down to two observations:
\begin{enumerate}
    \item The square \eqref{eq:FailToCommute} fails to commute by a controlled error; and
    \item after applying the closure map and alternating to pass to the Turaev cobracket, this error vanishes.
\end{enumerate}

In order to proceed we need to define an operation on $\tT$, which will help relate the ascending and descending lifts. The {\em vertical flip}, or {\em flip} for short. This is a composition of a vertical mirror image (mirror image to the ceiling), with orientation reversal of each pole. In other words, the flip  of a tangle is its vertical mirror image but with poles still  ascending. The flip of a tangle $T$ is denoted $T^\sharp$. The flip operation is also well-defined on the Conway quotient $\tTnab$ by setting $b^\sharp=-b$ for the variable $b$.

The associated graded {\em vertical flip}, or simply {\em flip}, of a chord diagram $D\in \tA$, denoted $D^{\sharp}$, is the vertical mirror image of $D$ with ascending poles, multiplied by $(-1)^s$, where $s$ is the $s$-degree of $D$. This is because the mirror image reverses the signs of all crossings, then pole reversals reverse back the signs of all pole-strand crossings. Thus, only the signs of strand-strand crossings are reversed by the composite. On the Conway quotient $\tA_\nab$, the associated graded flip is given by setting $a^\sharp=-a$.

\begin{lem}\label{lem:CDflip}
  The Kontsevich integral respects flips:  
  $$Z(T^\sharp)= (Z(T))^\sharp$$
  for any $T\in \tT$ or $T\in \tTnab$.
\end{lem}

\begin{proof}
    The Kontsevich integral is well known to respect mirror images and orientations switches, hence, it respects the composition. 
\end{proof}

The next lemma addresses the first of the two steps outlined above by modifying the bottom arrow of \eqref{eq:FailToCommute} to correct the error:

\begin{lem}\label{lem:LambdaAlg}
    There exists a map $\glosm{lambdaalg}{\lambda^{alg}}: \tA_{\nab}^{/2}(\botSkel) \to \tA_{\nab}^{/2}(\botSkel)$ so that the diagram \eqref{eq:FixedSquare} commutes\footnote{Of course, $\lambda^{alg}\neq \gr \lambda$.}.
\end{lem}

\begin{equation}\label{eq:FixedSquare}
    \begin{tikzcd}[]
       \tTnab^{/2}(\botSkel)\arrow[d, "Z^{/2}",swap] & \tTnab^{/2}(\botSkel)\arrow[l,"\lambda"]\arrow[d,"Z^{/2}"]
      \\
       \tA_{\nab}^{/2}(\botSkel ) &\tA_{\nab}^{/2}(\botSkel)\arrow[l,"\lambda^{alg}"]
    \end{tikzcd}
 \end{equation}

\begin{proof}
    Recall that by definition, $\lambda=(\lambda_a-\lambda_d)\circ q$, where $q$ is the projection $\tT^{/2}_\nab(\botSkel)\to\tT^{/1}(\botSkel)$, and $\lambda_a$ and $\lambda_d$ are the ascending and descending lifts. Since the Kontsevich integral is compatible with the $s$-filtration and hence with $q$, it is enough to show that the analogous statements are true for $\lambda_a$ and $\lambda_d$ separately. Namely, we show that there exist maps $\lambda_a^{alg}$ and $\lambda_d^{alg}$ making the following squares commute:
    
\begin{equation}\label{eq:twoFixedSquare}
    \begin{tikzcd}[]
       \tTnab^{/2}(\botSkel)\arrow[d, "Z^{/2}",swap] & \tT^{/1}(\botSkel)\arrow[l,"\lambda_a"]\arrow[d,"Z^{/1}"] &&
       \tTnab^{/2}(\botSkel)\arrow[d, "Z^{/2}",swap] & \tT^{/1}(\botSkel)\arrow[l,"\lambda_d"]\arrow[d,"Z^{/1}"]
      \\
       \tA_{\nab}^{/2}(\botSkel ) &\tA^{/1}(\botSkel)\arrow[l,"\lambda_a^{alg}"] &&
        \tA_{\nab}^{/2}(\botSkel ) &\tA^{/1}(\botSkel)\arrow[l,"\lambda_d^{alg}"]
    \end{tikzcd}
 \end{equation}

 Let $\gamma$ be a curve in $\tT^{/1}(\botSkel)$. To find $\lambda_a^{alg}$, we need to express $Z^{/2}(\lambda_a(\gamma))$ in terms of $Z^{/1}(\gamma)$. Since the Kontsevich integral is compatible with the $s$-filtration, $Z^{/1}(\gamma)=Z^{/1}(\lambda_a(\gamma))$.
 
 The proof thus depends on understanding the Kontsevich integral of $\lambda_a(\gamma)$: see Figure~\ref{fig:AscDecomp} for an expression of $\lambda_a(\gamma)$ as the tangle composition\footnote{We use notation inspired by Drinfel'd associators, but emphasise that these are classical Kontsevich integral calculations, not using the combinatorial construction in terms of the KZ associator. In particular, our definition of $Z$ is in terms of limits (``infinitely near'' and ``infinitely far''), rather than parenthetisations.}
 \begin{equation}\label{eq:Asc}
     \lambda_a(\gamma)=\Phi^{-1}  \kappa \Phi R C.
 \end{equation}
\begin{figure}
\begin{center}
\begin{picture}(300,180)
    \put(0,0){\includegraphics[width=10cm]{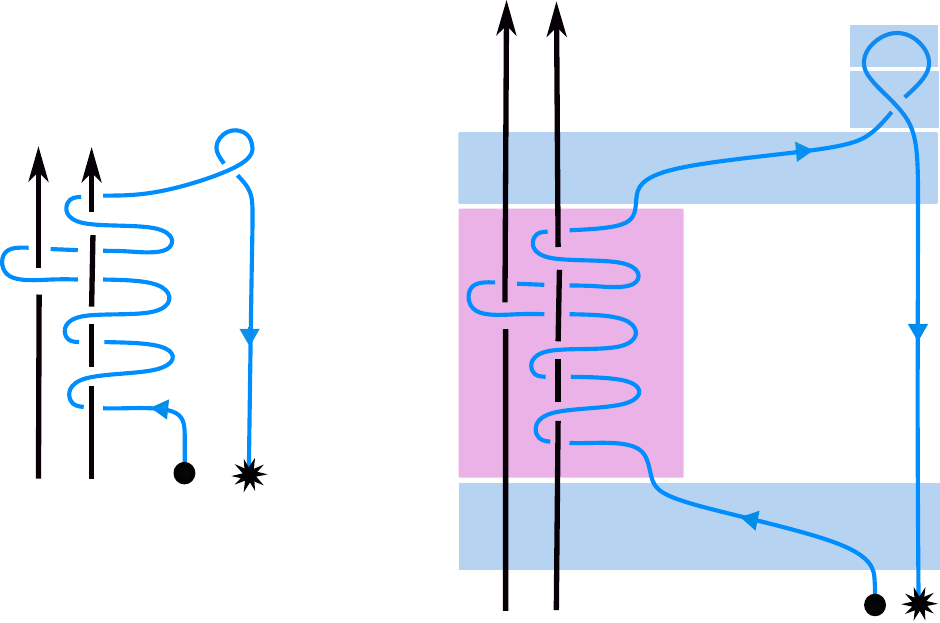}}
    \put(100,80){\Huge =}
    \put(287,170){$C$}
    \put(287,154){R}
    \put(287,133){$\Phi$}
    \put(287,85){$\kappa$}
    \put(287, 25){$\Phi^{-1}$}
\end{picture}
\end{center}
    \caption{Tangle decomposition of the ascending lift $\lambda_a(\gamma)$.}\label{fig:AscDecomp}
\end{figure}
 
 Since the Kontsevich integral is multiplicative with respect to tangle composition,
  \begin{equation}\label{eq:ZAsc}
     Z^{/2}(\lambda_a(\gamma))=Z^{/2}(\Phi^{-1}) Z^{/2}(\kappa) Z^{/2}(\Phi)Z^{/2}(R)Z^{/2}(C).
 \end{equation}
 Since the Kontsevich integral asympototically commutes with ``distant disjoint unions'' \cite[Chapter 8]{CDM_2012}, the values of $C$, $R$ and $\kappa$ include only chords fully contained in the highlighted areas of Figure~\ref{fig:AscDecomp}. In particular, the value of the cap in $C$ includes strand-strand chords only, and has no degree one term \cite{BN2}, thus, $Z^{/2}(C)=1$. The value $Z^{/2}(R)$ can be explicitly computed and is well known to be $1+\frac{t}{2}$, where $t$ denotes a single chord.

 As for $\kappa$, the value $Z^{/2}(\kappa)$ has no strand-strand chords, and the strand-pole chords follow the strand in ascending order. Thus, by Lemma~\ref{lem:GradedAscDesc}, 
 \begin{equation}\label{eq:Zbeta}
 Z^{/2}(\kappa)=\gr\lambda_a(Z^{/1}(\gamma)).
 \end{equation}
In summary, we have 
\begin{equation}\label{eq:ZAsc2}
Z^{/2}(\lambda_a(\gamma))=Z^{/2}(\Phi^{-1}) \gr\lambda_a(Z^{/1}(\gamma)) Z^{/2}(\Phi)\left(1+\frac{t}{2}\right)
\end{equation}
The formula can be further simplified by understanding $Z(\Phi).$ Since values of the Kontsevich integral are group-like, $Z(\Phi)=\exp(\varphi)$ for some primitive $\varphi\in \tA(\uparrow \uparrow \uparrow \downarrow)$. Since deleting either of the non-pole strands of $Z(\Phi)$ simplifies to 1, we have that $(Z(\Phi)-1) \in \tA^{1}$, and in particular $\varphi \in \tA^1$. Thus, 
\begin{equation}\label{eq:ZPhi}
Z^{/2}(\Phi)=1+\varphi \quad \text{with } \varphi\in \tA^{1/2}
\end{equation}
Consequently, $Z^{/2}(\Phi^{-1})=1-\varphi.$ Substituting these values into \eqref{eq:ZAsc2}, and expanding, we obtain
\begin{equation}\label{eq:ZAsc3}
 Z^{/2}(\lambda_a(\gamma)) =  \gr\lambda_a(Z^{/1}(\gamma))+ \left[\gr\lambda_a(Z^{/1}(\gamma)),\varphi\right] + \gr\lambda_a(Z^{/1}(\gamma))\frac{t}{2},
\end{equation}
where the square brackets denote the algebra commutator in $\tA^{/2}(\uparrow \uparrow \uparrow \downarrow)$.

In summary, for a diagram $D\in\tA^{/1}(\botSkel)$, the map given by the formula
\begin{equation}\label{eq:AscAlg}
    \yellowm{\lambda_a^{alg}}(D)= \gr\lambda_a(D)+ \left[\gr\lambda_a(D),\varphi\right] + \gr\lambda_a(D)\frac{t}{2}
\end{equation}
completes the commutative diagram \eqref{eq:FixedSquare} for $\lambda_a$, as required. Note that $\varphi$ does not depend on $D$.

Similarly for $\lambda_d(\gamma)$, Figure~\ref{fig:DescDecomp} shows that
\begin{equation}\label{eq:ZDesc}
Z^{/2}(\lambda_d(\gamma))=Z^{/2}(R^{-1})\left(Z^{/2}(\Phi^{\sharp})Z^{/2}(\kappa^{\sharp})Z^{/2}(\Phi^\sharp)^{-1}\right)^{2,1},
\end{equation}
where $\kappa^{\sharp}$ and $\Phi^\sharp$ are the flip of $\beta$, and $\Phi$, respectively, and the superscript ``\glosm{super2,1}{$2,1$}'' indicates that the strands of these components are swapped.
\begin{figure}
\begin{center}
\begin{picture}(300,160)
    \put(0,0){\includegraphics[width=10cm]{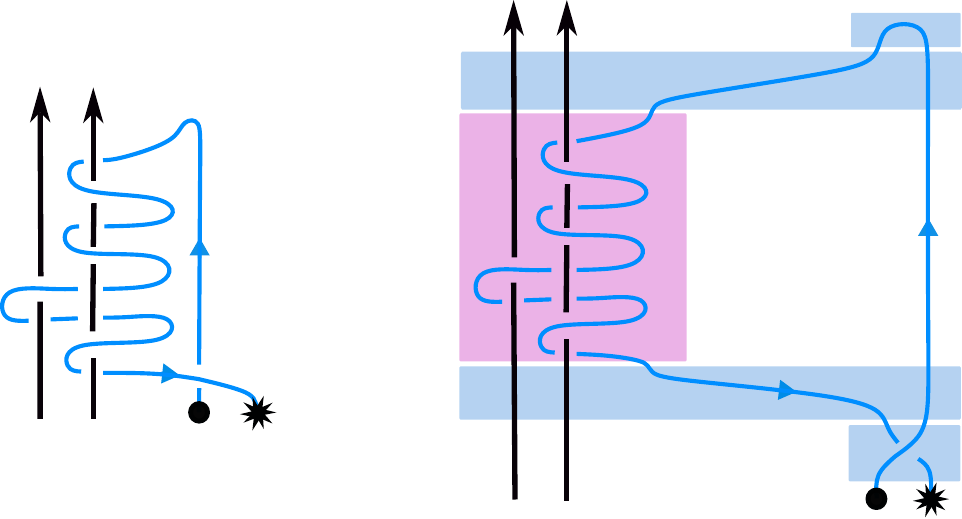}}
    \put(95,75){\Huge =}
    \put(286,140){$C$}
    \put(286,125){$(\Phi^\sharp)^{-1}$}
    \put(286,85){$\kappa^{\sharp}$}
    \put(286, 34){$\Phi^{\sharp}$}
    \put(286, 15){$R^{-1}$}
\end{picture}
\end{center}
    \caption{Tangle decomposition of the descending lift $\lambda_d(\gamma)$.}\label{fig:DescDecomp}
\end{figure}

Since $Z^{/2}(\kappa^{\sharp})$ involves only strand-pole chords, and since the strand is descending, we have by Lemma~\ref{lem:GradedAscDesc}: 
 \begin{equation}\label{eq:Zbetasharp}
 Z^{/2}(\kappa^{\sharp})=\gr\lambda_d(Z^{/1}(\gamma)).
 \end{equation}

By Lemma~\ref{lem:CDflip} we have 
\begin{equation}\label{eq:PhiSharp}
Z^{/2}(\Phi^\sharp)= (Z^{/2}(\Phi))^\sharp= (1+\varphi)^\sharp=1+\varphi^\sharp
\end{equation}
In turn, $\Z^{/2}((\Phi^\sharp)^{-1})=1-\varphi^\sharp$ and $Z^{/2}(R^{-1})= (1-\frac{t}{2})$. Substituting into \eqref{eq:ZDesc}, we obtain:
\begin{equation}\label{eq:ZDesc3}
    Z^{/2}(\lambda_d(\gamma))=\left(1-\frac{t}{2}\right)\left((1+\varphi^\sharp)\gr \lambda_d(Z^{/1}(\gamma))(1-\varphi^\sharp)\right)^{2,1}
\end{equation}
Expanded in $\tA^{/2}$, using that $t,\varphi \in \tA^{1/2}$ this gives
\begin{equation}\label{eq:ZDesc2}
    Z^{/2}(\lambda_d(\gamma))=\gr \lambda_d(Z^{/1}(\gamma)) + \left[\varphi^\sharp , \gr \lambda_d(Z^{/1}(\gamma))\right]^{2,1} - \frac{t}{2}\gr \lambda_d(Z^{/1}(\gamma))
\end{equation}

Therefore, we define
\begin{equation}\label{eq:DescAlg}
    \yellowm{\lambda_d^{alg}}(D)=\gr \lambda_d(D) + \left[\varphi^\sharp, \gr \lambda_d(D)\right]^{2,1}-\frac{t}{2}\gr \lambda_d(D)
\end{equation}
which completes the commutative diagram \eqref{eq:FixedSquare} for $\lambda_d$, as required.

Define $\glosm{lambdaaldbar}{\overline{\lambda}^{alg}}=\lambda_a^{alg}-\lambda_d^{alg}$, and $\lambda^{alg}=\overline{\lambda}^{alg}\circ q$, where $q$ is the projection $\tA_{\nab}^{/2} \to \tA_{\nab}^{/1}$. Then, by definition, $\lambda^{alg}$ makes the diagram \eqref{eq:FixedSquare} commute, completing the proof.
\end{proof}

Since the formula for $\lambda^{alg}$ will be important, we restate it as a proposition:

\begin{prop}\label{prop:LambdaAlg}
The map $\lambda^{alg}$ is defined by $\lambda^{alg}=(\lambda^{alg}_a-\lambda^{alg}_d)\circ q$, where
 \begin{align*}
 \lambda_a^{alg}(D)= & \gr\lambda_a(D)+ \left[\gr\lambda_a(D),\varphi\right] + \gr\lambda_a(D)\frac{t}{2}\\
   \lambda_d^{alg}(D)= &\gr \lambda_d(D) + \left[\varphi^\sharp, \gr \lambda_d(D)\right]^{2,1}-\frac{t}{2}\gr \lambda_d(D).
  \end{align*}  \qed
\end{prop}

\begin{lem}\label{lem:lambda_alg_diagram}
The map $\lambda^{alg}$ fits into the commutative diagram of Figure~\ref{fig:AdjustedSnakefor_gr_cobracket} (solid arrows). 
\end{lem}

\begin{proof}
The only non-empty part of this statement is that $q\circ\lambda^{alg}=0$, that is, the composition of $\lambda^{alg}$ with the projection to $\tA^{/1}(\botSkel)$ is zero. We have seen before that this is true for $\gr\lambda$, and it is shown in \eqref{eq:AscAlg} and \eqref{eq:DescAlg} that $\lambda^{alg}$ differs from $\gr\lambda$ in some correction terms. However, all of these correction terms are in $s$-degree 1 or higher (multiplicatively, 1 in $s$-degree 0). Hence, $q\circ\lambda^{alg}=q\circ \gr\lambda$.
\end{proof}

We denote the induced map of $\lambda^{alg}$ by $\glosm{etaalg}{\eta^{alg}}$ and the composition of $\eta^{alg}$ with $\gr \widecheck{b}$ by $\glosm{etahatalg}{\hat{\eta}^{alg}}$, as shown in Figure~\ref{fig:AdjustedSnakefor_gr_cobracket}.

\begin{figure}
\begin{tikzcd}[]
&\tA_{\nab}^{1/2}(\botSkel) \arrow[r]\arrow[d, "0",swap]  & \tA_{\nab}^{/2}(\botSkel) \arrow[rr] & &\tA^{/1}(\botSkel)\arrow[d,"0"] \arrow[ddlll, dashed, out=-140, in=0, looseness=.6 ,"\hat{\eta}^{alg}", near end]\arrow[dlll,swap, dashed, out=110, in=60, looseness=1 ,"\eta^{alg}"] \arrow[r]& 0\\
0\arrow[r]& \tA_{\nab}^{1/2}(\botSkel ) \arrow[r] \arrow[d,"\gr \widecheck{b}"] 
& \tA_{\nab}^{/2}(\botSkel )\arrow[from=u,"\lambda^{alg}",swap,  crossing over] \arrow[rr]
&& \tA^{/1}(\botSkel )\\
&\tA^{/1}(\circleSkel \, \botSkel) \\
\end{tikzcd}
\caption{The diagram for the self-intersection map, with corrected associated graded maps.}\label{fig:AdjustedSnakefor_gr_cobracket}
\end{figure}

\begin{cor}\label{cor:ZEtaAlg}
    The Kontsevich integral is compatible with the map $\eta$ and its corrected algebraic counterpart $\eta^{alg}$:
    \begin{equation}
        \eta^{alg}\circ Z^{/1}=Z^{/2}\circ \eta, 
    \end{equation}
    and consequently,
    \begin{equation}
        \hat{\eta}^{alg}\circ Z^{/1}=Z^{/1}\circ \hat{\eta}.
    \end{equation}
\end{cor}

\begin{proof}
If all faces of the multi-cube of Figure~\ref{fig:Cube_for_cobracket} commute, then this implies the statement, which itself is the commutativity of the diagonal (red) square. Indeed, all the faces of the multi-cube commute: the only non-trivial statement is the commutativity of the middle face, which was established in Lemma~\ref{lem:LambdaAlg}. Recall that $\hat{\eta}$ (respectively $\hat{\eta}^{alg}$) is obtained from $\eta$ (repsectively, $\eta^{alg}$) through division by $b$ (respectively, division by $a$). Therefore, the second equality follows from the fact that the Kontsevich integral respects the $s$-filtration (Proposition~\ref{prop:ZrespectsS}).
\begin{figure}[ht]
\begin{tikzcd}[row sep=scriptsize, column sep=small]
  & \tTnab^{1/2}(\botSkel) \arrow[dl,"0"] \arrow[rr] \arrow[dd,"Z^{/2}",near start] & & \tTnab^{/2}(\botSkel)\arrow[dl,"\lambda"]\arrow[rr] \arrow[dd,"Z^{/2}",near start] && \tT^{/1}(\botSkel)\arrow[dd,red,"Z^{/1}"]\arrow[dl,"0"] \arrow[dlllll,red, "{\eta}",swap, out=110, in=90,looseness=.6]  \\
\tT^{1/2}(\botSkel)\arrow[rr, crossing over, "\rho", near start] \arrow[dd,red,"Z^{/2}"] & & \tTnab^{/2}(\botSkel)\arrow[rr,crossing over]\ &&\tT^{/1}(\botSkel) \\
 & \tA_{\nab}^{1/2}(\botSkel) \arrow[dl,"0"] \arrow[rr] & & \tA_{\nab}^{/2}(\botSkel) \arrow[dl,"\lambda^{alg}", near start]\arrow[rr] && \tA^{/1}(\botSkel) \arrow[dl,"0"]\arrow[dlllll,red, "\eta^{alg}", out=-110,in=-50 , ,looseness=0.6]\\
\tA^{1/2}(\botSkel) \arrow[rr,] & & \tA_{\nab}^{/2}(\botSkel) \arrow[rr]\arrow[from=uu, "Z^{/2}",near start, crossing over]&& \tA^{/1}(\botSkel ) \arrow[from=uu, "Z^{/1}",near start, crossing over]\\
\end{tikzcd}
\caption{Commutative cube showing the compatibility of $\eta$ and $\eta^{alg}$ with the Kontsevich integral.}\label{fig:Cube_for_cobracket}
\end{figure}
\end{proof}

We have now established that the square \eqref{eq:FailToCommute} commutes up to a controlled error: namely, it commutes if $\gr \lambda$ is replaced by $\lambda^{alg}$, and furthermore, $\lambda^{alg}$ is expressed explicitly in terms of $\gr \lambda_a$ and $\gr \lambda_d$ in Proposition~\ref{prop:LambdaAlg} (and Equations \eqref{eq:AscAlg} and \eqref{eq:DescAlg}). In other words, we achieved the goal (1) stated above Lemma~\ref{lem:LambdaAlg}. Moving on to goal (2), we need to show that the error vanishes after passing to the Turaev cobracket $\delta$. Recall (Figure~\ref{fig:Identifications}) that $\delta:\Cpa \to \Cpa \otimes \Cpa$  descends from $\tilde{\delta}: \Cp \to \Cpa \otimes \Cpa$. In turn, $\tilde{\delta}=Alt \circ (\id \otimes |\cdot |)\circ\mu$ is identified via the isomorphisms $\beta$ with the composition $Alt \circ cl \circ \left(\widecheck{b} \circ \eta\right)=Alt\circ cl \circ \hat{\eta}$ by Theorem~\ref{thm:snake_for_mu} and Proposition~\ref{prop:zeta_to_del}.

\begin{thm}\label{thm:cobrackethomomorphic}
    The Kontsevich integral descends to a homomorphic expansion for the Turaev cobracket.
    Namely, 
    $(Z^{/1}\otimes Z^{/1})\circ \tilde{\delta}=\gr \tilde{\delta} \circ Z^{/1},$
as shown in the diagram below, and consequently, $(Z^{/1}\otimes Z^{/1})\circ {\delta}=\gr {\delta} \circ Z^{/1}.$
\begin{equation}\label{eq:Zdelta}
\begin{tikzcd}
\Cpa\otimes \Cpa \arrow[d, "Z^{/1}\otimes Z^{/1}"] &&\tT^{/1}(\circleSkel)\otimes \tT^{/1}(\circleSkel)  \arrow[ll,"\cong"',"\beta\otimes \beta"]\arrow[d, "Z^{/1}\otimes Z^{/1}"] && \tT^{/1}(\botSkel) \arrow{ll}{Alt\circ cl \circ \check{\eta}}
\arrow[d, "Z^{/1}"] &&\Cp  \arrow[ll,"\cong"',"\beta^{-1}"]\arrow[bend right=20]{llllll}{\tilde{\delta}} \arrow[d, "Z^{/1}"]\\
\hAs \otimes \hAs &&\A^{/1}(\circleSkel)\otimes \A^{/1}(\circleSkel)  \arrow[ll,"\cong"',"\gr\beta \otimes \gr\beta"]&&\A^{/1}(\botSkel)  \arrow[ll, "Alt \circ cl \circ\gr \hat{\eta}"] &&\hAs  \arrow[ll,"\cong"',"\gr\beta^{-1}"] \arrow[bend left=20,swap]{llllll}{\gr \hat{\delta}}
\end{tikzcd}
\end{equation}
    \end{thm}

\begin{proof}
   If $\gr \hat{\eta}$ were replaced with $\hat{\eta}^{alg}$ in the middle of the bottom row of \eqref{eq:Zdelta}, then this would be true by Corollary~\ref{cor:ZEtaAlg}. Thus, it is enough to show that 
   \begin{equation}\label{eq:Goal}
   Alt \circ cl \circ \hat{\eta}^{alg}=Alt \circ cl \circ \gr \hat{\eta}.\end{equation}
   Recall (from Figure~\ref{fig:AdjustedSnakefor_gr_cobracket}) that 
   \[\hat{\eta}^{alg}= \gr \check{\beta} \circ \left(\lambda^{alg}_a-\lambda^{alg}_d\right)=\frac{\lambda^{alg}_a-\lambda^{alg}_d}{a},\]
   where division by $a$ stands for applying the isomorphism $\gr \check{b}=\glosm{acheck}{\check{a}}:\tA^{1/2}\to \tA^{/1}$. Substituting the formulas of Proposition~\ref{prop:LambdaAlg}, we have
   \begin{multline*}\hat{\eta}^{alg}(D)=\frac{(\gr\lambda_a-  \gr \lambda_d)(D)}{a} \\
   +\frac{\left[\gr\lambda_a(D),\varphi\right]  - \left[\varphi^\sharp, \gr \lambda_d(D)\right]^{2,1}+ \gr\lambda_a(D)\frac{t}{2}+\frac{t}{2}\gr \lambda_d(D)}{a}.
   \end{multline*}
   By definition, we have that $\frac{(\gr\lambda_a-  \gr \lambda_d)(D)}{a}=\gr \hat{\eta}(D)$, hence, to prove \eqref{eq:Goal} it is enough to show that the second line of the formula above vanishes after closure and alternation. Namely, set 
   \begin{align}
   \yellowm{\varepsilon_1(D)}=&\frac{\left[\gr\lambda_a(D),\varphi\right]  - \left[\varphi^\sharp, \gr \lambda_d(D)\right]^{2,1}}{a},\label{eq:epi1}\\
\yellowm{\varepsilon_2(D)}=&\frac{\gr\lambda_a(D)\frac{t}{2}+\frac{t}{2}\gr \lambda_d(D)}{a}.\label{eq:epi2}
   \end{align}

   It is sufficient to prove that $Alt(cl(\varepsilon_1(D)))=0$ and $Alt(cl(\varepsilon_2(D)))=0$ for any $D\in \tA^{/1}(\botSkel)$. First note that the only difference between $\gr \lambda_a (D)$ and $\gr \lambda_d (D)$ is the order of chord attachement along the poles. However, in $\tA^{1/2}$ chord endings commute on the poles, by the same argument as Remark~\ref{rem:ChorsOnPoles}. Therefore, as the numerators of both expressions $\varepsilon_1$ and $\varepsilon_2$ are in $\tA^{1/2}$, the quantities $\gr \lambda_a(D)$ and $\gr \lambda_d(D)$ can be interchanged. We replace both quantities by the shorthand $B$. 

   For $\varepsilon_1$, recall from the proof of Lemma~\ref{lem:LambdaAlg} that $\varphi=Z^{/2}(\Phi)-1 \in \tA^{1/2}$. The value $\varphi$ is an infinite series graded by total degree; let $X$ denote a term in $\varphi$. Since $X\in \tA^{1/2}$, we can write $X=vtw$, as shown in Figure~\ref{fig:XFlip}. Here $v$ and $w$ are words in $p$ letters (where $p$ is the number of poles, and words are read in the direction of the strand), and $t$ is the strand-strand chord. By Lemma~\ref{lem:CDflip}, we have $X^\sharp= - {w}t{v}$, as shown in Figure~\ref{fig:XFlip}.

\begin{figure}
\begin{center}
\begin{picture}(210,60)
    \put(0,0){\includegraphics[width=8cm]{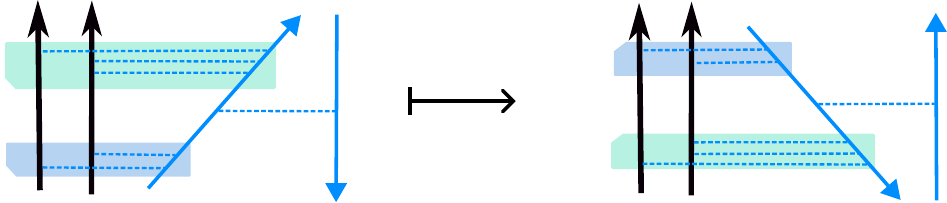}}
    \put(-10,8){ $v$}
    \put(-12,33){ $w$}
    \put(65,14){$t$}
    \put(100,30){ $\#$}
    \put(129,22){\large $-$}
    \put(134,9){ $w$}
    \put(136,31){ $v$}
    \put(208,26){$t$}
\end{picture}
\end{center}
\caption{Flipping the term $X$. We interpret $v$ and $w$ as words in two letters, read in the direction of the middle strand: for example, on both sides, $w=yyx$.}\label{fig:XFlip}
\end{figure}

Now we compute the term of 
$Alt(cl(\varepsilon_1(D)))$ corresponding to the term $X$ in $\varphi$, see Figure~\ref{fig:Epsilon1} for a diagrammatic depiction of this calculation.

\begin{align*}
Alt  (cl &\left(\check{a}\left([B,X]-[X^\sharp,B]^{2,1}\right)\right)) = \\
&=Alt\left(cl\left(\check{a}\left((Bvtw-vtwB)+(wtvB-Bwtv)^{2,1}\right)\right)\right)\\
&=Alt\left(cl\left(|w|\otimes Bv - |wB|\otimes v + |vB| \otimes w -|Bw|\otimes v \right) \right) \\
&= Alt\left(|w|\otimes |Bv| -|Bw|\otimes |v| +|Bv|\otimes |w| -|v|\otimes |Bw| \right)=0.
\end{align*}

\begin{figure}
\begin{center}
\begin{picture}(400,380)
    \put(-10,70){\includegraphics[width=14.5cm]{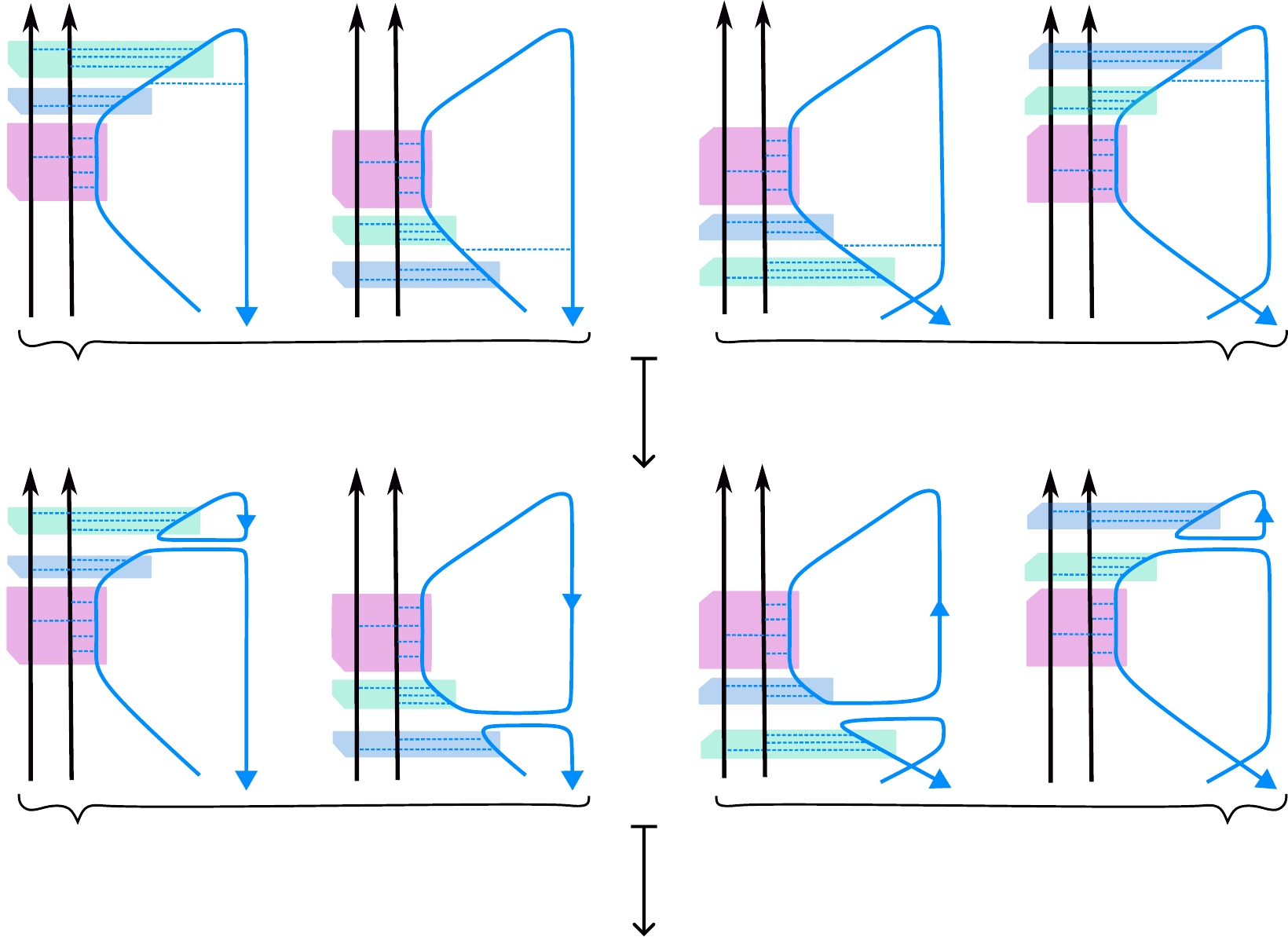}}
    \put(-15,0){\includegraphics[width=15cm]{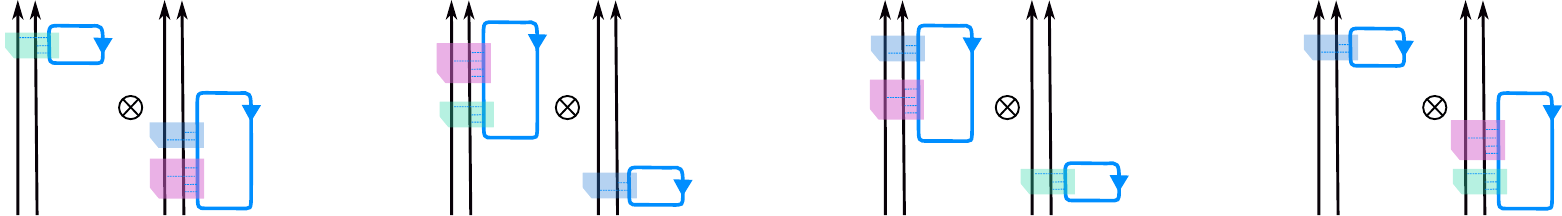}}
    \put(-20,350){ $w$}
    \put(-20,336){ $v$}
    \put(-22,315){ $B$}
    \put(70,305){ \Large $-$}
    \put(82,315){ $B$}
    \put(84,295){ $w$}
    \put(85,280){ $v$}
    \put(185,305){ \Large $+$}
    \put(200,315){ $B$}
    \put(202,295){ $v$}
    \put(202,280){ $w$}
    \put(295,305){ \Large $-$}
    \put(307,350){ $v$}
    \put(305,336){ $w$}
    \put(305,315){ $B$}
    \put(-5,240){ $[B,X]$}
    \put(204 ,235){$\check{a}$}
    \put(350,240){ $[-X^\sharp, B]^{2,1}$}
    \put(-20,202){ $w$}
    \put(-20,185){ $v$}
    \put(-22,168){ $B$}
    \put(70,150){ \Large $-$}
    \put(82,165){ $B$}
    \put(83,145){ $w$}
    \put(85,130){ $v$}
    \put(185,150){ \Large $+$}
    \put(200,165){ $B$}
    \put(203,145){ $v$}
    \put(200,130){ $w$}
    \put(295,150){ \Large $-$}
    \put(307,202){ $v$}
    \put(305,185){ $w$}
    \put(305,165){ $B$}
    \put(-5,90){ $[B,X]$}
    \put(202,85){ $cl$}
    \put(350,90){ $[-X^\sharp, B]^{2,1}$}
    \put(-23,45){\Small  $w$}
    \put(42,20){\Small  $v$}
    \put(38,8){ \Small $B$}
    \put(70,25){ \Large $-$}
    \put(96,40){\Small $B$}
    \put(97,25){\Small $w$}
    \put(160,6){\Small $v$}
    \put(185,25){ \Large $+$}
    \put(216,44){\Small $v$}
    \put(213,30){\Small $B$}
    \put(279,6){\Small $w$}
    \put(295,25){ \Large $-$}
    \put(335,45){\Small $v$}
    \put(397,18){\Small $B$} 
    \put(397,6){\Small $w$}
\end{picture}
\end{center}
\caption{Calculation of the term in $\varepsilon_1(D)$ corresponding to $X$.}\label{fig:Epsilon1}
\end{figure}

The second equality follows the convention that the cyclic word is placed in the first tensor component. 
Since for every term $X$ of $\varphi$, the corresponding term of $Alt(cl(\epsilon_1(D)))$ vanishes, it follows that $Alt(cl(\varepsilon_1(D))=0$, as required.

Finally, for $\varepsilon_2(D)$, we compute:
\begin{align*}Alt(cl(\varepsilon_2(D)))&=Alt\left(cl\left(\check{a}\left(B\frac{t}{2}+\frac{t}{2}B\right)\right)\right)
\\
&= Alt(cl(|1| \otimes B + |B| \otimes 1))=Alt(1 \otimes |B|+|B|\otimes 1)=0. 
\end{align*}
Therefore, both components of the error vanish after closure and alternation, completing the proof.
\end{proof}

\medskip

\newpage

\section{Glossary of notation} \label{sec:glossary}

\noindent
{\small\begin{multicols}{2}

\raggedright\begin{list}{}{
  \renewcommand{\makelabel}[1]{#1\hfil}
}

\glosi{a}{$a$}{formal variable of $\C\tTnab$}{Def.~\ref{def:conway}}
\glosi{acheck}{$\check{a}$}{division by $a$ map}{Thm. \ref{thm:cobrackethomomorphic}}
\glosi{tA}{$\tA$}{associated graded $\ctT$}{Sec.~\ref{sec:t-filtration}}
\glosi{tAnab}{$\tA_\nab$}{associated graded $\ctT_\nab$, isomorphic to $\D$}{Sec.~\ref{sec:Conway}}
\glosi{tAnabS}{$\tA_\nab(S)$}{the image $\gr\iota(\tA(S))$}{Def.~\ref{def:Anot}}
\glosi{tAt}{$\tA_t$}{degree $t$ component of $\tA$}{Sec.~\ref{sec:t-filtration}}
 \glosi{tAtgeqs}{$\tA^{\geq s}$}{$s$-filtered component of $\tA$}{Def.~\ref{def:filtrationQuotientNotation}}
\glosi{tAS}{$\tA(S)$}{admissible chord diagrams on the skeleton $S$}{Sec.~\ref{sec:t-filtration}}
\glosi{tAnabt}{$\tA_{\nab,t}$}{degree $t$ component of $\tA_{\nab}$}{Sec.~\ref{sec:Conway}}
\glosi{tAnabs}{$\tA_{\nab}^s$}{degree $s$ component of $\tA_{\nab}$}{Sec.~\ref{sec:Conway}}
 \glosi{tAnabts}{$\tA_{\nab,t}^s$}{$\tA_{\nab,t}\cap \tA_{\nab}^s$}{Sec.~\ref{sec:Conway}}

 \glosi{Aslashs}{$\tA^{/s}$}{$\tA/\tA^{\geq s}$}{Def.~\ref{def:Anot}} 
\glosi{Anabslashs}{$\tA_{\nab}^{/s}$}{$\tA_{\nab}/\tA_{\nab}^{\geq s}$}{Def.~\ref{def:Anot}} 

\glosi{alt}{Alt}{alternating map}{Def.~\ref{def:cobrac}}

\glosi{b}{$b$}{$e^{\frac{a}{2}}-e^{-\frac{a}{2}}$}{Def.~\ref{def:conway}}
\glosi{bhat}{$\widehat{b}$}{multiplication by $b$ map}{Prop.~\ref{prop:divbybexists}}
\glosi{bcheck}{$\widecheck{b}$}{division by $b$ map}{Prop.~\ref{prop:divbybexists}}
\glosi{beta}{$\beta$}{bottom projection map}{Prop.~\ref{prop:BotProj}}
\glosi{botskel}{$\botSkel$}{bottom tangle}{Sec.~\ref{sec:cobracketinCON}}

\glosi{cl}{$cl$}{closure map}{Sec.~\ref{sec:cobracketinCON}}
\glosi{Cpi}{$\C\pi$}{group algebra of $\pi$}{Sec.~\ref{subsec:IntroGT}}
\glosi{aCpi}{$|\C\pi|$}{homotopy classes of free loops in $D_p$}{Sec.~\ref{subsec:IntroGT}}
\glosi{Ctpi}{$\C\tpi$}{group algebra of $\tpi$}{Sec.~\ref{subsec:IntroGT}}
\glosi{absCtpi}{$|\C\tpi|$}{homotopy classes of immersed free loops in $D_p$}{Sec.~\ref{subsec:IntroGT}}
\glosi{cpba}{$\Cpba$}{$|\Cp|/\C\cl$}{Sec.~\ref{subsec:IntroGT}}
\glosi{ctT}{$\ctT$}{formal linear combinations of oriented tangles in $M_p$}{Sec.~\ref{sec:framed_tangles}}
\glosi{Ctnab}{$\C\tTnab$}{Conway quotient of $\ctT$}{Def.~\ref{def:conway}}
\glosi{CtnabS}{$\C\tTnab(S)$}{the image $\iota(\ctT(S))$}{Def.~\ref{def:conway_skel}}

\glosi{barD}{$\widebar{D}$}{vertical flip of diagram $D$}{Sec.~\ref{sec:cobracketinCON}}
\glosi{D}{$\D$}{$\sqcup_S\D(S)$}{Def.~\ref{def:cdspace}}
\glosi{Dnab}{$\D_\nab$}{Conway quotient of $\D$}{Def.~\ref{def:D_con}}

\glosi{Dt}{$\D_t$}{$\sqcup_S\D_t(S)$}{Def.~\ref{def:cdspace}}
\glosi{Ds}{$\D(S)$}{space of admissible chord diagrams on a skeleton $S$}{Def.~\ref{def:cdspace}}
\glosi{Dst}{$\D_t(S)$}{degree $t$ component of $\D(S)$}{Def.~\ref{def:cdspace}}
\glosi{del}{$\delta$}{Turaev Cobracket}{Def.~\ref{def:cobrac}}
\glosi{grdel}{$\delta_{\gr}$}{graded Turaev Cobracket}{Prop.~\ref{prop:gr_del}}
\glosi{deltilde}{$\tilde{\delta}$}{$cl\circ \mu$}{Def.~\ref{def:cobrac}}
\glosi{Dp}{$D_p$}{$p$-punctured disk}{Sec.~\ref{subsec:IntroGT}}

\glosi{epi1}{$\varepsilon_1,\varepsilon_2$ }{}{Eq. \ref{eq:epi1}, Eq. \ref{eq:epi2}}
\glosi{etahat}{$\eta,\hat{\eta}$}{general notation for induced map from ``$\lambda$''}{Thm.~\ref{thm:bracketsnake} and also Thm~\ref{thm:snake_for_mu}}
\glosi{etaalg}{$\eta^{alg},\hat{\eta}^{alg}$}{}{Fig.~\ref{fig:AdjustedSnakefor_gr_cobracket}}

\glosi{FA}{$\As$}{degree completed free algebra $\As\langle x_1,\cdots, x_p\rangle$}{Sec.~\ref{sec:gr_bialgebra}}
\glosi{absFA}{$|\As|$}{$\As/[\As,\As]$}{Sec.~\ref{sec:gr_bialgebra}}
\glosi{hatFA}{$\widehat{\As}$}{degree completion of $\As$}{Sec.~\ref{sec:identifybracketinCON}}
\glosi{hatabsFA}{$|\widehat{\As}|$}{degree completion of $|\As|$}{Sec.~\ref{sec:identifybracketinCON}}
\glosi{flip}{$(-)^\sharp$}{flip operation}{Sec.~\ref{sec:opsonT}}

\glosi{Gbrac}{$[\cdot,\cdot]_G$}{Goldman Bracket}{Def.~\ref{def:bracket}}
\glosi{grCpi}{$\gr\C\pi$}{graded $\C\pi$}{Sec.~\ref{sec:gr_bialgebra}}
\glosi{grGbrac}{$[\cdot,\cdot]_{\gr G}$}{graded Goldman Bracket}{Prop.~\ref{prop:gr_bracket_def}}

\glosi{I}{$\calI$}{augmentation ideal}{Sec.~\ref{sec:gr_bialgebra}}
\glosi{i}{$\iota$}{ }{Sec.~\ref{sec:Conway}}

\glosi{caLK}{$\calK$}{links or tangles in $\R^3$}{Sec. \ref{sec:hom_exp}}
\glosi{tcaLK}{$\tcalK$}{framed links $\R^3$}{Sec. \ref{subsubsec:Framing}}
\glosi{tcaLKi}{$\tcalK_i$}{filtered component of $\tcalK$}{Sec. \ref{subsubsec:Framing}}

\glosi{lambda}{$\lambda$}{general notation for a difference of two maps, $\lambda= \lambda_1-\lambda_2$ in Thm.~\ref{thm:bracketsnake}, $\lambda=(\lambda_a-\lambda_d)\circ q$ in Thm.~\ref{thm:snake_for_mu}}{}
\glosi{lambda12}{$\lambda_1,\lambda_2$}{stacking products}{Thm.~\ref{thm:bracketsnake}}
\glosi{lambdaa}{$\lambda_a$}{framed ascending lift}{Sec.~\ref{sec:cobracketinCON}}
\glosi{lambdad}{$\lambda_d$}{framed descending lift}{Sec.~\ref{sec:cobracketinCON}}
\glosi{lambdabar}{$\bar{\lambda}$}{$\lambda_a-\lambda_d$}{Sec.~\ref{sec:cobracketinCON}}
\glosi{lambdalg}{$\lambda^{alg}$}{}{Lem.~\ref{lem:LambdaAlg}}
\glosi{lambdaaalg}{$\lambda_a^{alg}$}{}{Eq.~\ref{eq:AscAlg}}
\glosi{lambdadalg}{$\lambda_d^{alg}$}{}{Eq.~\ref{eq:DescAlg}}
\glosi{lambdaaldbar}{$\overline{\lambda}^{alg}$}{$\lambda_a^{alg}-\lambda_d^{alg}$}{Sec.~\ref{sec:cobracketinCON}}

\glosi{Mp}{$M_p$}{$D_p\times I$}{Sec.~\ref{sec:framed_tangles}}
\glosi{mu}{$\mu$}{ self intersection map}{Def.~\ref{def:mu}}
\glosi{grmu}{$\mu_{\gr}$}{graded self intersection map}{Prop.~\ref{prop:gr_mu}}

\glosi{nu}{$\nu$}{path from $*$ to $\bullet$}{Sec.~\ref{subsec:IntroGT}}

\glosi{psi}{$\psi$}{chord contraction map}{Lem.~\ref{lem:psi}}
\glosi{pi}{$\pi$}{$\pi_1(D_p,*)$}{Sec.~\ref{subsec:IntroGT}}
\glosi{tpi}{$\tpi$}{regular homotopy classes of immersed curves in $D_p$}{Sec.~\ref{subsec:IntroGT}}

\glosi{q}{$q$}{general notation of a projection map }{Thm.~\ref{thm:snake_for_mu} and Lem~\ref{lem:lambda_alg_diagram}}

\glosi{super2,1}{$(-)^{2,1}$ }{strands of the 1 and 2 components are swapped}{Lem. \ref{lem:LambdaAlg}}

\glosi{tT}{$\tT$}{framed tangles in $M_p$}{Def. \ref{def:framed_tangles}}
\glosi{tTs}{$\tT^s$}{$s$-filtered component of $\C\tT$}{Sec. \ref{sec:s-sfiltration}}
\glosi{tTst}{$\tT_t^s$}{$\tT_t \cap \tT^s$}{Sec. \ref{sec:s-sfiltration}}
\glosi{tTt}{$\tT_t$}{$t$ filtered component of $\C\tT$}{Sec. \ref{sec:t-filtration}}
\glosi{tTnabs}{$\tT_{\nab}^s$}{$s$-filtered component of $\ctT_\nab$}{Sec.~\ref{sec:Conway}}
\glosi{tTnabt}{$\tT_{\nab, t}$}{$t$-filtered component of $\ctT_\nab$}{Sec.~\ref{sec:Conway}}
\glosi{tTn}{$\tT^{/n}$}{$\tT/ \tT^n $}{Sec.~\ref{sec:Conway}}
\glosi{tTnabn}{$\tT_\nab^{/n}$}{$\ctT_\nab/ \tT_\nab^n $}{Sec.~\ref{sec:Conway}}
\glosi{tT1/2}{$\tT^{1/2}$}{$\tT^1/ \tT^2 $}{Sec.~\ref{sec:Conway}}
\glosi{tTnab1/2}{$\tT_\nab^{1/2}$}{$\tT_\nab^1/ \tT_\nab^2 $}{Sec.~\ref{sec:Conway}}

\glosi{xi}{$\xi$}{inward normal vector}{Sec.~\ref{subsec:IntroGT}}


\glosi{Z}{$Z$}{Kontsevich Integral}{}
\glosi{Znab}{$Z_\nab$}{Kontsevich Integral on $\C\tT_\nab$}{Thm.~\ref{thm:Z_conway}}
\glosi{ZT}{$Z(T)$}{Kontsevich Integral of one-stranded tangles}{Sec.~\ref{sec:cobracketinCON}}
\glosi{zeta}{$\zeta$}{map descending from $\tilde{\zeta}$}{Sec.~\ref{sec:cobracketinCON}}
\glosi{zetatilde}{$\tilde{\zeta}$}{$Alt\circ\hat{\zeta}$}{Sec.~\ref{sec:cobracketinCON}}
\glosi{zetahat}{$\hat{\zeta}$}{$cl\circ\hat{\eta}$}{Sec.~\ref{sec:cobracketinCON}}

\end{list}
\end{multicols}}

\bibliographystyle{alpha}
\bibliography{main.bib}{}
\end{document}